\documentclass[onefignum,onetabnum,reqno]{siamonline171218}

\usepackage{braket,amsfonts}

\usepackage{array}

\usepackage[caption=false]{subfig}
\usepackage{pgfplots}

\newsiamthm{claim}{Claim}
\newsiamremark{remark}{Remark}
\newsiamremark{hypothesis}{Hypothesis}
\crefname{hypothesis}{Hypothesis}{Hypotheses}

\usepackage{algorithmic}

\usepackage{graphicx,epstopdf}
\usepackage{algorithmic}
\newtheorem{assumption}{Assumption}[section]
\newcommand{\Id}{\mathbb{I}}

\newcommand{\R}{\mathbb{R}}

\newcommand{\Rext}{\R\cup\{+\infty\}}

\renewcommand{\set}[1]{\left\{#1\right\}}
\newcommand{\sets}[1]{\{#1\}}
\newcommand{\norm}[1]{\left\Vert#1\right\Vert}

\newcommand{\norms}[1]{\Vert#1\Vert}

\newcommand{\Eproof}{\hfill $\square$}
\newcommand{\prox}{\mathrm{prox}}
\newcommand{\kprox}[2]{\mathrm{prox}_{#1}\big(#2\big)}

\newcommand{\relint}[1]{\mathrm{ri}\left(#1\right)}
\newcommand{\argmin}{\mathrm{arg}\!\displaystyle\min}

\newcommand{\dom}[1]{\mathrm{dom}(#1)}
\newcommand{\zero}[1]{\boldsymbol{0}}

\newcommand{\xb}{x}
\newcommand{\xopt}{x^{\star}}
\newcommand{\yopt}{y^{\star}}

\newcommand{\yb}{y}

\newcommand{\zb}{z}
\newcommand{\wb}{w}
\newcommand{\ub}{u}  
\newcommand{\vb}{v}

\renewcommand{\sb}{s}
\newcommand{\rb}{r}

\newcommand{\bb}{b}
\newcommand{\cb}{c}

\newcommand{\Ab}{A}
\newcommand{\Bb}{B}

\newcommand{\Bc}{\mathcal{B}}

\newcommand{\Qc}{\mathcal{Q}}
\newcommand{\Rc}{\mathcal{R}}

\newcommand{\Yb}{Y}

\newcommand{\Xc}{\mathcal{X}}

\newcommand{\Zc}{\mathcal{Z}}
\newcommand{\Sc}{\mathcal{S}}
\newcommand{\Dc}{\mathcal{D}}
\newcommand{\Lc}{\mathcal{L}}

\newcommand{\Nc}{\mathcal{N}}

\newcommand{\Kc}{\mathcal{K}}

\newcommand{\Tc}{\mathcal{T}}
\newcommand{\lbd}{\lambda} 
\newcommand{\Lbd}{\Lambda} 
\newcommand{\lbdopt}{\lambda^{\star}}

\newcommand{\iprods}[1]{\langle #1\rangle}

\newcommand{\BigO}[1]{\mathcal{O}\left(#1\right)}

\newcommand{\beforesubsec}{\vspace{0ex}}
\newcommand{\aftersubsec}{\vspace{0ex}}
\newcommand{\beforesec}{\vspace{0ex}}
\newcommand{\aftersec}{\vspace{0ex}}
\newcommand{\beforesubsubsec}{\vspace{0ex}}
\newcommand{\aftersubsubsec}{\vspace{0ex}}
\newcommand{\beforeparagraph}{\vspace{0ex}}

\Crefname{ALC@unique}{Line}{Lines}

\usepackage{amsopn}

\usepackage{xspace}
\usepackage{bold-extra}
\usepackage[most]{tcolorbox}

\colorlet{texcscolor}{blue!50!black}
\colorlet{texemcolor}{red!70!black}
\colorlet{texpreamble}{red!70!black}
\colorlet{codebackground}{black!25!white!25}

\lstdefinestyle{siamlatex}{%
  style=tcblatex,
  texcsstyle=*\color{texcscolor},
  texcsstyle=[2]\color{texemcolor},
  keywordstyle=[2]\color{texemcolor},
  moretexcs={cref,Cref,maketitle,mathcal,text,headers,email,url},
}

\tcbset{%
  colframe=black!75!white!75,
  coltitle=white,
  colback=codebackground, % bottom/left side
  colbacklower=white, % top/right side
  fonttitle=\bfseries,
  arc=0pt,outer arc=0pt,
  top=1pt,bottom=1pt,left=1mm,right=1mm,middle=1mm,boxsep=1mm,
  leftrule=0.3mm,rightrule=0.3mm,toprule=0.3mm,bottomrule=0.3mm,
  listing options={style=siamlatex}
}

\newtcblisting[use counter=example]{example}[2][]{%
  title={Example~\thetcbcounter: #2},#1}

\newtcbinputlisting[use counter=example]{\examplefile}[3][]{%
  title={Example~\thetcbcounter: #2},listing file={#3},#1}

\DeclareTotalTCBox{\code}{ v O{} }
{ %fontupper=\ttfamily\color{texemcolor},
  fontupper=\ttfamily\color{black},
  nobeforeafter,
  tcbox raise base,
  colback=codebackground,colframe=white,
  top=0pt,bottom=0pt,left=0mm,right=0mm,
  leftrule=0pt,rightrule=0pt,toprule=0mm,bottomrule=0mm,
  boxsep=0.5mm,
  #2}{#1}

\patchcmd\newpage{\vfil}{}{}{}
\begin{tcbverbatimwrite}{tmp_\jobname_header.tex}
\title{Augmented Lagrangian-Based Decomposition Methods with  Non-Ergodic Optimal Rates}
\author{Quoc Tran-Dinh$^{\ast}$ \and Yuzixuan Zhu\thanks{Department of Statistics and Operations Research, University of North Carolina at Chapel Hill, 333-Hanes Hall, UNC-Chapel Hill, NC27599 (\email{quoctd@email.unc.edu}).}}

\headers{Non-Ergodic Augmented Lagrangian-Based Decomposition Methods}{Q. Tran-Dinh ~$\cdot$~ Y. Zhu}
\end{tcbverbatimwrite}
\input{tmp_\jobname_header.tex}

\ifpdf
\hypersetup{ pdftitle={Guide to Using  SIAM'S \LaTeX\ Style} }
\fi

\begin{document}
\maketitle

%% ------------------------------------------------------------------
%% ABSTRACT
%% ------------------------------------------------------------------
\begin{tcbverbatimwrite}{tmp_\jobname_abstract.tex}
\begin{abstract}
We develop two new variants of alternating direction methods of multipliers (ADMM) and two parallel primal-dual decomposition algorithms to solve a wide range class of constrained convex optimization problems.
Our approach relies on a novel combination of the augmented Lagrangian framework, partial alternating/linearization scheme, Nesterov's acceleration technique, and adaptive strategy.
The proposed algorithms have the following new features compared to existing ADMM variants.
First, they have a Nesterov's acceleration step on the primal variables instead of the dual ones as in several  ADMM variants.
Second, they possess an optimal $\BigO{\frac{1}{k}}$-convergence rate guarantee in a non-ergodic sense without any smoothness or strong convexity-type assumption, where $k$ is the iteration counter. 
When one objective term is strongly convex, our algorithm achieves an optimal $\BigO{\frac{1}{k^2}}$-non-ergodic rate.
Third, our methods have better per-iteration complexity than standard ADMM due to the linearization step in the second subproblem.
Fourth, we provide a set of conditions to derive update rules for algorithmic parameters, and give a concrete update for these parameters as an example.
Finally, when the objective function is separable, our methods can naturally be implemented in a parallel fashion.
We also study two extensions of our methods and a connection to existing  primal-dual methods.
We verify our theoretical development via different numerical examples and compare our methods with some existing state-of-the-art algorithms.
\end{abstract}


\end{tcbverbatimwrite}
\input{tmp_\jobname_abstract.tex}
%% ------------------------------------------------------------------
%% END HEADER
%% ------------------------------------------------------------------

% Main paper.
%%%%%%%%%%%%%%%%%%%%%%%%%%%%%%%%%%%%%%%%%%%%%%%%%
%% 1. Introduction.
%%%%%%%%%%%%%%%%%%%%%%%%%%%%%%%%%%%%%%%%%%%%%%%%%
\beforesec
\section{Introduction}\label{sec:intro}
\aftersec
We study new numerical primal-dual methods to solve the following general and possibly nonsmooth constrained convex optimization problem:
\begin{equation}\label{eq:constr_cvx}
F^{\star} := \displaystyle\min_{\zb := (\xb, \yb)\in\R^{p}}  \Big\{ F(\zb) := f(\xb) + g(\yb) ~~\mathrm{s.t.}~~ \Ab\xb + \Bb\yb = \cb \Big\},
\end{equation}
where $f : \R^{p_1}\to\Rext$ and $g : \R^{p_2}\to\Rext$ are two proper, closed, and convex functions; $p := p_1 + p_2$;
$\Ab\in\R^{n\times p_1}$, $\Bb\in\R^{n\times p_2}$, and $\cb\in\R^n$ are given.
We often assume that we do not know the explicit form of $A$ and $B$, but we can only compute $\Ab\xb$, $\Bb\yb$ and their adjoint $\Ab^{\top}\lbd$ and $\Bb^{\top}\lbd$ for any given $\xb$, $\yb$, and $\lbd$.
Undoubtedly, under only convexity of $f$ and $g$, problem \eqref{eq:constr_cvx} covers many practical models in different fields, see, e.g.,  \cite{Bauschke2011,bottou2016optimization,Boyd2011,Boyd2004,sra2012optimization,wright2017optimization}.

%%% 
\beforeparagraph
\paragraph{Literature review}
In the past fifteen years, large-scale convex optimization has become a very active area. 
Various algorithms have been developed and rediscovered to solve this type of problems. 
Prominent examples include [proximal] gradient and fast gradient \cite{Nesterov2004}, conditional gradient (also called Frank-Wolfe's algorithms) \cite{Jaggi2013}, coordinate descent \cite{Nesterov2011}, mirror-descent \cite{Beck2003,Nemirovskii1983} stochastic gradient descent \cite{Nemirovski2009},  operator splitting \cite{Bauschke2011}, primal-dual first-order \cite{Chambolle2011},  and incremental gradient-type methods \cite{Bertsekas2011}.
Together with algorithms,  supporting theory such as convergence guarantees and complexity analysis are also well-studied, see, e.g., \cite{chambolle2016ergodic,Davis2014a,Davis2014b,He2012b,Lan2013,Lan2013b,shefi2016rate,tseng2008accelerated} and the references quoted therein.
Although many algorithms have been developed, they mainly focus on solving unconstrained composite convex problems or ``simple'' constrained convex problems such as proximal-based or Frank-Wolfe's methods, where projections onto the constrained set can be computed efficiently.
When problem have complex linear constraints as in \eqref{eq:constr_cvx}, solution approaches are rather different. 
Existing methods heavily rely on dual subgradient/gradient algorithms, interior-point and barrier schemes, augmented Lagrangian-based methods such as alternating minimization (AMA) and alternating direction  methods of multipliers (ADMM).
Recently, several variants of primal-dual methods, coordinate descent algorithms, and penalty frameworks have also been developed to solve constrained setting \eqref{eq:constr_cvx} but require a certain set of assumptions \cite{Goldstein2012,Lan2013,Lan2013b,Ouyang2014,xu2017accelerated,xu2018accelerated}.

Our methods developed in this paper is  along the line of augmented Lagrangian and primal-dual framework.
Therefore, we briefly review some notable and recent works in this area that are most related to our algorithms.
The augmented Lagrangian method was dated back from the work of Powell and Hestenes in nonlinear programming in early 1970s \cite{Nocedal2006}.
It soon became a powerful tool to solve nonlinear optimization as well as constrained convex optimization problems.
A comprehensive study of this method can be found in \cite{Bertsekas1996d}.
Alternatively, alternating methods were dated back from von Neumann's work \cite{von2016functional} where we can view it as a special case of coordinate descent-type methods.
The alternating minimization algorithm (AMA) \cite{Tseng1991a} and the alternating direction method of multipliers (ADMM) \cite{Eckstein1992,Lions1979} combine both ideas of the augmented Lagrangian framework and alternating strategy. 
ADMM is widely used in practice, especially in signal and image processing, and data analysis \cite{afonso2010fast,Goldstein2012,Yang2010}. 
\cite{Boyd2011} provides a comprehensive survey of ADMM using in statistical learning.

In terms of algorithms, AMA and ADMM can be viewed as a dual variant of forward-backward and Douglas-Rachford's splitting methods, respectively  \cite{Eckstein1992,Lions1979,Tseng1991a}. 
%It can also be viewed as a variant of augmented Lagrangian methods \cite{Bertsekas1996d}.
Although various variants of AMA and ADMM have been studied in the literature, their three main steps (two primal subproblems, and one dual update) remain the same in most existing papers. 
Some modifications have been injected into ADMM such as relaxation \cite{Chen1994,Davis2014b,Ouyang2014,shefi2016rate}, or dual acceleration \cite{Goldstein2012,Ouyang2014}.
Other extensions to Bregman distances and proximal settings remain essentially the same as the original version, see, e.g., \cite{wang2015convergence,Wang2013a}.
Due to its broad applicability, ADMM is much widely used than AMA, and it performs well in many applications \cite{Boyd2011}.

In terms of theory, while the asymptotic convergence of ADMM has been known for a long time, see, e.g., \cite{Eckstein1992}, its $\BigO{\frac{1}{k}}$-convergence rate seems to be first proved in \cite{He2011}.
Nevertheless, such a rate is achieved through a gap function of its variational inequality reformulation and in an ergodic sense. 
The same $\BigO{\frac{1}{k}}$-non-ergodic rate was then proved in \cite{He2012a}, but still on the sequence of differences $\set{\norms{ \wb^{k+1} - \wb^k}^2}$ of both the primal and dual variables in $\wb$. 
Many other works also focus on theoretical aspects of ADMM by showing its $\BigO{\frac{1}{k}}$-convergence rate in the objective residual $\vert F(\bar{\zb}^k) - F^{\star}\vert$ and the feasibility $\norms{\Ab\bar{\xb}^k + \Bb\bar{\yb}^k - \cb}$. 
Notable papers include \cite{Davis2014,Davis2014b,Goldstein2012,Ouyang2014,shefi2016rate}.
Extensions to stochastic settings as well as multi-blocks formulations have also been intensively studied, e.g., in \cite{chen2016direct,deng2013parallel,lin2015iteration,lin2015global}.
Other researchers were trying to optimize the rate of convergence in certain cases such as \cite{Ghadimi2014,nishihara2015general}.
Most of existing results can show an ergodic convergence rate of  $\BigO{\frac{1}{k}}$ in either gap function or in both objective residual and constraint violation \cite{Davis2014,Davis2014b,Goldstein2012,He2011,Ouyang2014,shefi2016rate,Wei2011}.
This rate is optimal under only convexity and strong duality \cite{woodworth2016tight}.
When one objective function $f$ or $g$ is strongly convex, one can achieve  $\BigO{\frac{1}{k^2}}$ rate as shown in \cite{xu2017accelerated} but it is still on an averaging sequence.
Many papers have attempted to prove linear convergence of ADMM by imposing stronger assumptions, see, e.g., \cite{Deng2012,hong2012linear}.
A recent work \cite{li2016accelerated} proposed a linearized ADMM variant using Nesterov's acceleration step and showed an $\BigO{\frac{1}{k}}$-non-ergodic rate.
This scheme is similar to our scheme \eqref{eq:admm_scheme1c} in the sequel but is different from Algorithm \ref{alg:A0}.
However, our scheme \eqref{eq:admm_scheme1c} is even better than \cite{li2016accelerated}  since it allows one to compute the proximal operators of $f$ and $g$ in parallel instead of alternating as in \cite{li2016accelerated}.

In sparse and low-rank optimization as well as in signal and image processing, non-ergodic rates are more preferable than ergodic ones.
A non-ergodic sequence preserves desired structures characterized by the underlying objective functions such as sparsity, low-rankness, or sharp edges of images. 
Averaging often destroys these properties. 
Hitherto, non-ergodic rate guarantees of ADMM as well as of primal-dual methods have not been well-studied. 
To the best of our knowledge, \cite{li2016accelerated} proposed a non-ergodic variant of ADMM, while \cite{TranDinh2015b} developed a non-ergodic primal-dual method for both composite convex problems and \eqref{eq:constr_cvx}.
In \cite{Chambolle2011}, the authors characterized a non-ergodic rate in the squared distance of the iterates for strongly convex cases, but this rate depends on a tuning parameter and remains suboptimal. 

\beforeparagraph
\paragraph{Our approach}
We propose a novel combination of the augmented Lagrangian (AL) framework and other techniques.
First, we use the AL function as a merit function to measure approximate solutions. 
Second, we incorporate an acceleration step (either Nesterov's momentum \cite{Nesterov1983} or Tseng's variant \cite{tseng2008accelerated}) into the primal steps instead of the dual ones as often seen in ADMM and primal-dual methods \cite{Davis2014,Davis2014b,Goldstein2012,Ouyang2014}.
Third, we alternate the primal subproblem into two subproblems in $\xb$ and $\yb$.
Fourth, we also partly  linearize one subproblem to reduce the per-iteration complexity.
Finally, we combine with an adaptive strategy to derive explicit update rules for parameters and to achieve optimal convergence rates.

\beforeparagraph
\paragraph{Our contribution}
To this end, our contribution can be summarized as follows:
\begin{itemize}
\item[$\mathrm{(a)}$]
We propose two novel primal-dual augmented Lagrangian-based algorithms  to solve \eqref{eq:constr_cvx} under only convexity and zero duality gap assumptions.
The first algorithm can be viewed as a preconditioned accelerated ADMM variant \cite{Chambolle2011}. 
The second one is a primal-dual decomposition method that allows us to fully linearize the augmented term into two subproblems of $x$ and $y$, and solves them in parallel.

\item[$\mathrm{(b)}$]
We prove an optimal $\BigO{\frac{1}{k}}$-convergence rate of both algorithms in terms of the objective residual $\vert F(\bar{\zb}^k) - F^{\star}\vert$ and the feasibility $\norms{\Ab\bar{\xb}^k + \Bb\bar{\yb}^k - \cb}$.
Our rate achieves at the last iterate instead of  [weighted] averaging (i.e., in a non-ergodic sense).

\item[$\mathrm{(c)}$]
When one objective function $f$ or $g$ is strongly convex, we develop a new ADMM variant to exploit this structure.
Our algorithm achieves an optimal $\BigO{\frac{1}{k^2}}$-convergence rate without significantly incurring the per-iteration complexity.
This rate is either in ergodic or non-ergodic sense.
The non-ergodic rate just requires one additional proximal operator of $g$.
When both $f$ and $g$ are strongly convex, we develop a new linearized primal-dual decomposition variant that achieves an optimal $\BigO{\frac{1}{k^2}}$-convergence rate.
This algorithm again can be implemented in parallel.

\item[$\mathrm{(d)}$]
We study two extensions of our algorithms and a connection between our methods and primal-dual methods.
We derive new variants of our algorithms to solve unconstrained composite convex problems which have optimal non-ergodic rates. 
\end{itemize}

In terms of theory, the per-iteration complexity of two algorithmic variants in (a) is better than that of  standard ADMM while they are applicable to solve nonsmooth constrained problems in \eqref{eq:constr_cvx} under the same assumptions as in ADMM or even weaker.\footnote{ADMM requires the solvability of two subproblems, but in our methods, we do not require this assumption.}
The second variant has  better per-iteration complexity and other advantages than the first one.
First, it only requires one proximal operator of $f$ and $g$ instead of solving a general convex subproblem in $\xb$.
Second, it allows one to compute these operators in parallel which can be generalized to handle \eqref{eq:constr_cvx} with separable structures of several objective terms (\textit{cf.} Subsection \ref{subsec:separable_case}).

Our first algorithm, Algorithm \ref{alg:A0}, shares some similarity with \cite{tran2017proximal}.
However, \cite{tran2017proximal} relies on a penalty approach and works on the primal space only.
Our algorithms in this paper are primal-dual methods.
The second variant has some similarity to \cite{li2016accelerated}, but it is a parallel algorithm.
It also shares some similarity with ASGARD in \cite{TranDinh2015b}, but ASGARD relies on smoothing techniques and does not have a dual step.
Algorithm \ref{alg:A0b} developed in (c) achieves the same rate as in \cite{xu2017accelerated}.
However, our algorithm has several advantages compared to \cite{xu2017accelerated}.
First, it linearizes one subproblem in $\yb$.
Second, the convergence rate can achieve in either a partial ergodic or a non-ergodic sense.
Third, all parameters are updated explicitly.
The second variant in (c) achieves an optimal rate in a non-ergodic sense.
To the best of our knowledge, this algorithm is new and its convergence rate has not been known in the literature.

%%% c. Paper organization.
\beforeparagraph
\paragraph{Paper organization}
The rest of this paper is organized as follows.
Section~\ref{sec:prelim_results} recalls the dual problem of \eqref{eq:constr_cvx}, a fundamental assumption, and its optimality condition.
It also provides a key lemma to analyze convergence rates of our algorithms.
Section~\ref{sec:padmm1} presents two algorithms: one variant of ADMM and one primal-dual decomposition method, and analyzes their convergence rate guarantees.
Section~\ref{sec:padmm1b} considers the strongly convex case.
We propose two algorithms to handle two situations.
Section \ref{sec:extensions} deals with some extensions, and Section \ref{sec:connection1} makes a connection to primal-dual first-order methods.
Section~\ref{sec:num_experiments} provides several numerical examples to illustrate our theoretical development and compares with existing methods.
For clarity of exposition, all technical proofs are deferred to the appendices.

%%%%%%%%%%%%%%%%%%%%%%%%%%%%%%%%%%%%%%%%%%%%%%%%%
%%% 2. Preliminaries: Primal and dual, optimality condition.
%%%%%%%%%%%%%%%%%%%%%%%%%%%%%%%%%%%%%%%%%%%%%%%%%
\beforesec
\section{Dual problem and optimality condition}\label{sec:prelim_results}
\aftersec
We first define the dual problem of \eqref{eq:constr_cvx} and recall its optimality condition. 
Then, we provide  a key lemma on approximate solutions.

%%% 2. a. Basic notation:
\beforesubsec
\subsection{Basic notation}
\aftersubsec
We work on finite dimensional spaces, $\R^p$ and $\R^n$, equipped with a standard inner product $\iprods{\cdot,\cdot}$ and Euclidean norm $\norm{\cdot} := \iprods{\cdot, \cdot}^{1/2}$.
Given a proper, closed and convex function $f$, $\dom{f}$ denotes its domain, $\partial{f}(\cdot)$ is its subdifferential, $f^{\ast}(\yb) := \sup_{\xb}\set{\iprods{\yb,\xb} \!-\! f(\xb)}$ is its Fenchel conjugate, and $\kprox{\gamma f}{\xb} := \argmin_{\ub}\set{ f(\ub) \!+\! 1/(2\gamma)\norms{\ub \!-\! \xb}^2}$ is called its the proximal operator, where $\gamma > 0$.
We say that $f$ has \textit{tractably proximal operator} $\prox_{\gamma f}$ if this operator can be computed efficiently, e.g., in a closed form or via a low-order polynomial time algorithm. Examples of such functions can be found, e.g., in \cite{Bauschke2011,Combettes2011,Parikh2013}.
We say that $f$ is $L_f$-smooth if it is differentiable, and its gradient $\nabla{f}$ is Lipschitz continuous on $\dom{f}$ with the Lipschitz constant $L_f \in [0, +\infty)$.
We say that $f$ is $\mu_f$-strongly convex if $f(\cdot) - \frac{\mu_f}{2}\norms{\cdot}^2$ is convex, where $\mu_f > 0$ is its strong convexity parameter.
For a given set $\Xc$, $\relint{\Xc}$ denotes its relative interior.
Other related concepts  can be found in \cite{Bauschke2011,Rockafellar1970}.

%%% 2.b. Dual problem, fundamental assumption, and optimality condition: 
\beforesubsec
\subsection{Dual problem, fundamental assumption, and optimality condition} 
\aftersubsec
We define
\begin{equation*}
\Lc(\xb, \yb, \lbd) := f(\xb) + g(\yb) - \iprods{\Ab\xb + \Bb\yb - \cb, \lbd},
\end{equation*}
as the Lagrange function associated with \eqref{eq:constr_cvx}, where $\lbd$ is the vector of Lagrange multipliers.
The dual function is defined as
\begin{equation*}
\begin{array}{ll}
d(\lbd) &:=   \displaystyle\max_{(\xb,\yb)\in\dom{F}} \Big\{ \iprods{\Ab\xb + \Bb\yb - \cb, \lbd} - f(\xb) - g(\yb) \Big\}  = f^{\ast}(A^{\top}\lbd) + g^{\ast}(\Bb^{\top}\lbd) - \iprods{\cb,\lbd},
\end{array}
\end{equation*}
where $\dom{F} := \dom{f}\times\dom{g}$. % and $f^{\ast}$ and $g^{\ast}$ are the Fenchel conjugates of $f$ and $g$, respectively.
The dual problem of \eqref{eq:constr_cvx} is
\begin{equation}\label{eq:dual_prob}
d^{\star} :=  \min_{\lbd \in \R^n} \Big\{ d(\lbd) \equiv f^{\ast}(A^{\top}\lbd) + g^{\ast}(\Bb^{\top}\lbd) - \iprods{\cb,\lbd}  \Big\}.
\end{equation}
We say that a point $(\xb^{\star},\yb^{\star}, \lbd^{\star}) \in \dom{F}\times\R^n$ is a saddle point of $\Lc$ if for all $(\xb,\yb)\in\dom{F}$, and $\lbd\in\R^n$, one has
\begin{equation}\label{eq:saddle_point}
\Lc(\xb^{\star}, \yb^{\star}, \lbd) \leq \Lc(\xb^{\star}, \yb^{\star},  \lbd^{\star}) \leq \Lc(\xb, \yb, \lbd^{\star}).
\end{equation}
We denote by $\Sc^{\star} := \set{ (\xb^{\star}, \yb^{\star}, \lbd^{\star})}$ the set of saddle points of $\Lc$ satisfying \eqref{eq:saddle_point}, $\Zc^{\star} := \set{(\xb^{\star}, \yb^{\star})}$, 
and by $\Lbd^{\star} := \set{\lbd^{\star}}$ the set of the multipliers $\lbd^{\star}$.
%
%%% 2.c. Fundamental assumption.
In this paper, we rely on the following mild assumption.

%% Assumption A.1.
\begin{assumption}\label{as:A1}
Both functions $f$ and $g$ are proper, closed, and convex.
The set of saddle points $\Sc^{\star}$ of  $\Lc$ is nonempty, and $F^{\star}$ is finite and attainable at some $(\xb^{\star},\yb^{\star}) \in \Zc^{\star}$.
\end{assumption}

We assume that Assumption~\ref{as:A1} holds throughout this paper without recalling it in the sequel.
The optimality condition (or the KKT condition) of \eqref{eq:constr_cvx} can be written as
\begin{equation}\label{eq:opt_cond}
0 \in \partial{f}(\xopt) - \Ab^{\top}\lbdopt, ~~0 \in \partial{g}(\yopt) - \Bb^{\top}\lbdopt, ~~\text{and}~~ \Ab\xopt + \Bb\yopt = \cb.
\end{equation}
Let us assume that the following Slater condition holds:
\begin{equation*}
\relint{\dom{F}}\cap\set{(\xb, \yb) \mid \Ab\xb + \Bb\yb = \cb} \neq\emptyset.
\end{equation*}
Then the optimality condition \eqref{eq:opt_cond} is necessary and sufficient for the strong duality of \eqref{eq:constr_cvx} and \eqref{eq:dual_prob} to hold, i.e., $F^{\star} + D^{\star} = 0$, and the dual solution is attainable and the dual solution set $\Lambda^{\star}$ is bounded, see, e.g., \cite{Bertsekas1999}.

Our goal is to find an approximation $\tilde{\zb}^{\star} := (\tilde{\xb}^{\star}, \tilde{\yb}^{\star})$ to $\zb^{\star}$ of \eqref{eq:constr_cvx} in the following sense:
%% Definition 2.1.
\begin{definition}\label{de:approx_sols}
We say that $\tilde{\zb}^{\star} := (\tilde{\xb}^{\star}, \tilde{\yb}^{\star}) \in \dom{F}$ is an $\varepsilon$-solution of \eqref{eq:constr_cvx} if
\begin{equation*}
\vert F(\tilde{\zb}^{\star}) - F^{\star}\vert \leq \varepsilon~~~\text{and}~~~\norms{\Ab\tilde{\xb}^{\star} + \Bb\tilde{\yb}^{\star} - \cb} \leq \varepsilon.
\end{equation*}
\end{definition}
\noindent The following lemma characterizes this approximate solution of \eqref{eq:constr_cvx} whose proof is in \cite{xu2017accelerated}.

%% Lemma 2.1.
\begin{lemma}\label{le:approx_opt_cond}
Let $\Rc : \R^n \to \R_{+}$ be a continuous function and $\zb = (\xb, \yb) \in\dom{F}$ be a given point. 
If for any $\lbd\in\R^n$, one has
\begin{equation*}
F(\zb) - F^{\star} - \iprods{\lbd, \Ab\xb + \Bb\yb - \cb} \leq \Rc(\lbd),
\end{equation*}
then, for any $\rho > 0$, we have 
\begin{equation*}
F(\zb) - F^{\star} + \rho\norms{\Ab\xb + \Bb\yb - \cb} \leq \sup\set{ \Rc(\lbd) ~\mid~ \norms{\lbd} \leq \rho}.
\end{equation*}
Consequently, if we choose $\rho$ such that $\norms{\lbd^{\star}} < \rho$ and set $\varepsilon_{\rho} := \sup\set{ \Rc(\lbd) \mid \norms{\lbd} \leq \rho}$, then
\begin{equation}\label{eq:approx_opt_cond2}
\tfrac{-\norms{\lbd^{\star}}\varepsilon_{\rho}}{\rho-\norms{\lbd^{\star}}} \leq F(\zb) - F^{\star} \leq \varepsilon_{\rho}, ~~~~\text{and}~~~~ \norms{\Ab\xb + \Bb\yb - \cb} \leq \tfrac{\varepsilon_{\rho}}{\rho - \norms{\lbd^{\star}}}.
\end{equation}
In particular, if we choose $\rho = 2\norm{\lbd^{\star}}$ for a nonzero $\lbd^{\star}$, then we obtain 
\begin{equation*}
\vert F(\zb) - F^{\star}\vert \leq \varepsilon_{\rho}~~~~\text{and}~~~\norms{\Ab\xb + \Bb\yb - \cb} \leq \tfrac{\varepsilon_{\rho}}{\norms{\lbd^{\star}}}.
\end{equation*}
\end{lemma}

%%%%%%%%%%%%%%%%%%%%%%%%%%%%%%%%%%%%%%%%%%%%%%
%%% 3. Proximal Alternating Penalized Algorithms
%%%%%%%%%%%%%%%%%%%%%%%%%%%%%%%%%%%%%%%%%%%%%%
\beforesec
\section{New augmented Lagrangian-based algorithms}\label{sec:padmm1}
\aftersec
We present two new primal-dual augmented Lagrangian-based algorithms.
The first one is essentially a preconditioned accelerated ADMM variant with proximal terms.
The second variant is a primal-dual decomposition algorithm that allows to parallelize proximal operators of $f$ and $g$.

\beforesubsec
\subsection{Preconditioned ADMM with Primal Accelerated Step}
\aftersubsec
We define the augmented Lagrangian function $\Lc_{\rho}$ associated with problem \eqref{eq:constr_cvx} as
\begin{equation}\label{eq:auLag_func}
\Lc_{\rho}(\zb, \lbd) := f(\xb) + g(\yb) - \iprods{\lbd, \Ab\xb + \Bb\yb - \cb} + \frac{\rho}{2}\norm{\Ab\xb + \Bb\yb - \cb}^2,
\end{equation}
where $\zb := (\xb, \yb)$, $\lbd$ is a corresponding multiplier, and $\rho > 0$ is a penalty parameter.

Let us first propose a new variant of ADMM using our approach.
We call this variant is preconditioned alternating direction algorithm of multipliers (PADMM) by adopting the name from \cite{Chambolle2011}.
For our notational convenience, we define the following subproblem.
Given $\hat{\zb}^k := (\hat{\xb}^k, \hat{\yb}^k) \in \dom{F}$, $\hat{\lbd}^k \in\R^n$, $\rho_k > 0$, and $\gamma_k \geq 0$, we consider the following $x$-subproblem: 
\begin{equation}\label{eq:x_cvx_subprob}
\Sc_{\gamma_k}(\hat{\zb}^k, \hat{\lbd}^k; \rho_k) := \argmin_{\xb}\Big\{ f(\xb) -  \iprods{\hat{\lbd}^k, \Ab\xb} + \frac{\rho_k}{2}\norms{\Ab\xb + \Bb\hat{\yb}^k -  \cb}^2 + \frac{\gamma_k}{2}\norms{\xb - \hat{\xb}^k}^2 \Big\}.
\end{equation}
Here, we allow $\gamma_k$ to be zero as long as this problem is solvable. 
For example, if $\Ab = \Id$ or orthogognal, then we can set $\gamma_k = 0$, and the problem \eqref{eq:x_cvx_subprob} still has a unique optimal solution.
Now, we can present our first method in Algorithm~\ref{alg:A0}.

%%%% Algorithm 1.
\begin{algorithm}[hpt!]\caption{{\!}(\textit{Preconditioned Alternating Direction Method of Multipliers}~(\texttt{PADMM}))}\label{alg:A0}
\begin{normalsize}
\begin{algorithmic}[1]
	\STATE {\hskip0ex}\textbf{Initialization:}  Choose $\bar{\zb}^0 := (\bar{\xb}^0, \bar{\yb}^0) \in \dom{F}$, $\hat{\lbd}^0 \in\R^{n}$, $\gamma_0 \geq 0$, and $\rho_0 > 0$. 
	\STATE{\hskip2ex}Set  $\tilde{\zb}^0 := \bar{\zb}^0$.
	\vspace{1ex}
	\STATE \textbf{For $k := 0$ to $k_{\max}$ perform}
		\vspace{1ex}
		\vspace{1ex}		
		\STATE{\hskip2ex}\label{step:A1_param_update}
                 Update $\tau_k := \frac{1}{k+1}$, ~$\rho_k := \rho_0(k+1)$, ~$\beta_k := 2\rho_0\norms{\Bb}^2(k+1)$, and $\eta_k := \frac{\rho_0}{2}$.
                 \vspace{1ex}
		\STATE{\hskip2ex}\label{step:admm_step}Update 
		$\left\{\begin{array}{ll}
		\hat{\zb}^k &:= (1-\tau_k)\bar{\zb}^k + \tau_k\tilde{\zb}^k \vspace{1ex}\\
		\bar{\xb}^{k+1} &:= \Sc_{\gamma_k}(\hat{\zb}^k, \hat{\lbd}^k;\rho_k) \vspace{1ex}\\
		\bar{\yb}^{k+1} &:= \kprox{g/\beta_k}{\hat{\yb}^k - \tfrac{1}{\beta_k}\Bb^{\top}\big(\rho_k(\Ab\bar{\xb}^{k+1} + \Bb\hat{\yb}^k - \cb) - \hat{\lbd}^k\big)} \vspace{1ex}\\
		\tilde{\zb}^{k+1} &:= \tilde{\zb}^k + \tfrac{1}{\tau_k}(\bar{\zb}^{k+1} - \hat{\zb}^k) \vspace{1ex}\\
		\hat{\lbd}^{k+1} &:= \hat{\lbd}^k - \eta_k(\Ab\tilde{\xb}^{k+1} + \Bb\tilde{\yb}^{k+1} - \cb).
                 \end{array}\right.$
		\STATE{\hskip2ex}\label{step:A1_gamma_update}
		Update $\gamma_{k+1}$ such that $0 \leq \gamma_{k+1} \leq \left(\frac{k+2}{k+1}\right)\gamma_{k}$ if necessary. 
                 \vspace{1ex}
	\STATE\textbf{End~for}
\end{algorithmic}
\end{normalsize}
\end{algorithm}

%%% Per-iteration complexity.
\paragraph{Per-iteration complexity}
Using the first and fourth lines of Step \ref{step:admm_step}, we can eliminate $\tilde{\zb}^k$ as 
\begin{equation*}
\hat{\zb}^{k+1} := \bar{\zb}^{k+1} + \tfrac{\tau_{k+1}(1-\tau_k)}{\tau_k}\big(\bar{\zb}^{k+1} - \bar{\zb}^k\big).
\end{equation*}
In this case, $\hat{\lbd}^k$ is updated as
\begin{equation*}
\hat{\lbd}^{k+1} := \hat{\lbd}^k - \tfrac{\eta_k}{\tau_k}\Big(\Ab\bar{\xb}^{k+1} + \Bb\bar{\yb}^{k+1} - \cb - (1-\tau_k)\big(\Ab\bar{\xb}^k + \Bb\bar{\yb}^{k} - \cb \big) \Big).
\end{equation*}
The per-iteration complexity of Algorithm \ref{alg:A0} consists of the solution of the $x$-subproblem \eqref{eq:x_cvx_subprob}, one proximal operator of $g$, one matrix vector multiplication $(\Ab\bar{\xb}^k, \Bb\bar{\yb}^k)$, and one adjoint operator $\Bb^{\top}\hat{\sb}^k$ at the third line of Step \ref{step:admm_step}.
Clearly, due to the linearization of the $y$-subproblem,  the per-iteration complexity of Algorithm \ref{alg:A0} is better than that of standard ADMM \cite{Boyd2011}, especially, when $\Ab$ is non-orthogonal.

Before analyzing the convergence of Algorithm~\ref{alg:A0}, we make the following remarks.
\begin{itemize}
\item 
First, the proximal term in \eqref{eq:x_cvx_subprob} only makes this problem to be well-defined.
If \eqref{eq:x_cvx_subprob} is solvable, then we can remove this proximal term and obtain
\begin{equation*} 
\Sc(\hat{\yb}^k, \hat{\lbd}^k; \rho_k) := \argmin_{\xb}\Big\{ f(\xb) -  \iprods{\hat{\lbd}^k, \Ab\xb} + \tfrac{\rho_k}{2}\norms{\Ab\xb + \Bb\hat{\yb}^k -  \cb}^2 \Big\}.
\end{equation*}
When $\Ab$ is identical or orthogonal (i.e., $\Ab^{\top}\Ab = \Id$), we can solve this problem in a closed form by using proximal operator of $f$ as $\Sc(\hat{\yb}^k, \hat{\lbd}^k; \rho_k) := \kprox{f/\rho_k}{\Ab^{\top}(c - \Bb\hat{\yb}^k + \rho_k^{-1}\hat{\lbd}^k)}$.
Otherwise, we can use first-order methods to solve this problem and it has a linear convergence rate due to strong convexity of \eqref{eq:x_cvx_subprob}.

\item 
Second, $\gamma_k$ can be updated decreasingly, can be fixed, or can be increased with the fastest rate of $\gamma_k := \gamma_0(k+1)$. 
The simplest way is to fix $\gamma_k := \gamma_0 > 0$ for all $k\geq 0$.

\item 
Third, we directly inject a Nesterov's accelerated step to the primal problem instead of the dual one as in \cite{Goldstein2012,Ouyang2014}.
This step can be simplified as above to reduce per-iteration complexity.

\item
Fourth, the dual step-size $\eta_k$ is fixed at $0.5\rho_0$ which is different from $\rho_k$, the penalty parameter.
$\rho_k$ is increasing with the rate $k$ in this algorithm.
Hence, Algorithm \ref{alg:A0} can be viewed as a relaxed ADMM variant \cite{Davis2014b,Ouyang2014} with the relaxation factor $\omega_k := \frac{1}{2(k+1)}$.

\item 
Fifth, if we set $\tau_k := 1$ in Algorithm \ref{alg:A0}, then Step \ref{step:A1_param_update} reduces to the preconditioned ADMM variant studied in \cite[Section 4.3.]{Chambolle2011} for the case $\Ab = \Id$.

\item
Finally, our parameter update rule is derived from the condition \eqref{eq:param_coditions5} in the appendix.
There are many ways to update these parameters.
For example, we first update $\tau_k$ with the rate of $\BigO{\frac{1}{k}}$. 
Then, we update $\rho_{k} := \frac{\rho_{k-1}}{1-\tau_k}$ and $\beta_k := 2\norms{\Bb}^2\rho_k$.
In Algorithm \ref{alg:A0}, we provide concrete update rules that only depend on one parameter $\rho_0$, which needs to be selected at the initialization stage.
\end{itemize}
The non-ergodic convergence rate of Algorithm~\ref{alg:A0} is stated in the following theorem whose proof can be found in Appendix \ref{sec:convergence_analysis_of_A0}.

%%% Theorem 3.1.
\begin{theorem}\label{th:admm-convergence1}
Let $\sets{\bar{\zb}^k}$ be the sequence generated by Algorithm~\ref{alg:A0}. 
Then, we have
\begin{equation}\label{eq:admm-convergence1}
\vert F(\bar{\zb}^k) - F^{\star}\vert \leq \frac{\bar{R}_0^2}{k}~~~~\text{and}~~~~ \norms{\Ab\bar{\xb}^k + \Bb\bar{\yb}^k - c} \leq \frac{\bar{R}_0^2}{\norms{{\lbd}^{\star}}k},~~~\text{for all}~k\geq 1,
\end{equation}
where $\bar{R}_0^2 := \frac{\gamma_0}{2}\norms{\bar{\xb}^0 - \xb^{\star}}^2 + \rho_0\norms{\Bb}^2\norms{\bar{\yb}^0 - \yb^{\star}}^2 +  \tfrac{1}{\rho_0}\big(2\norms{\lbd^{\star}} - \norms{\hat{\lbd}^0}\big)^2$.

Consequently, $\set{\bar{\zb}^k}$ globally converges to a solution $\zb^{\star}$ of \eqref{eq:constr_cvx} at an optimal $\BigO{\frac{1}{k}}$-rate in a non-ergodic sense, i.e., $\vert F(\bar{\zb}^k) - F^{\star}\vert \leq \BigO{\frac{1}{k}}$ and $\norms{\Ab\bar{\xb}^k + \Bb\bar{\yb}^k - \cb} \leq \BigO{\frac{1}{k}}$.
\end{theorem}

Note that if $\gamma_0 = 0$, then $\bar{R}_0^2 :=  \rho_0\norms{\Bb}^2\norms{\bar{\yb}^0 - \yb^{\star}}^2 +  \tfrac{1}{\rho_0}\big(2\norms{\lbd^{\star}} - \norms{\hat{\lbd}^0}\big)^2$, which is independent of $\bar{\xb}^0$. Moreover, by minimizing $\bar{R}_0^2$ with respect to $\rho_0 > 0$, we can find that the optimal value of $\rho_0$ is $\rho_0 := \frac{\vert 2\norms{\lbd^{\star}} - \norms{\hat{\lbd}^0}\vert}{\norms{\Bb}\norms{\bar{\yb}^0 - \yb^{\star}}}$, which unfortunately depends on the solutions $\lambda^{\star}$ and $\yb^{\star}$.
However, it also guides a rough way to select $\rho_0$ in concrete applications where we can bound $\norms{\bar{\yb}^0 - \yb^{\star}}$ and $\vert 2\norms{\lbd^{\star}} - \norms{\hat{\lbd}^0}\vert$.

%%% 3.1. Non-ergodic parallelized linearized ADMM variant.
\beforesubsec
\subsection{Parallel primal-dual decomposition algorithm}\label{subsec:nepl_admm}
\aftersubsec
Now, we can modify Algorithm \ref{alg:A0} to obtain a parallel variant.
Given $\bar{\zb}^0\in\dom{F}$ and $\hat{\lbd}^0  \in \R^n$, we set $\tilde{\zb}^0 := \bar{\zb}^0$ and update
\begin{equation}\label{eq:admm_scheme1c}
\left\{\begin{array}{lll}
&\hat{\zb}_k ~~~:= (1-\tau_k)\bar{z}_k + \tau_k\tilde{z}_k & \vspace{1ex}\\
&\hat{\ub}^k ~~~:= \rho_k(\Ab\hat{\xb}^k + \Bb\hat{\yb}^k - \cb) - \hat{\lbd}^k & \vspace{1ex}\\
&{\!\!\!\!}\left.\begin{array}{l}
\bar{\xb}^{k+1} := \kprox{f/\gamma_k}{\hat{\xb}^k - \tfrac{1}{\gamma_k}\Ab^{\top}\hat{\ub}^k} \vspace{1ex}\\
\bar{\yb}^{k+1} := \kprox{g/\beta_k}{\hat{\yb}^k - \tfrac{1}{\beta_k}\Bb^{\top}\hat{\ub}^k}
\end{array}\right] & \text{(Parallel step)}\vspace{1ex}\\
&\tilde{\zb}^{k+1} := \tilde{\zb}^k + \tfrac{1}{\tau_k}(\bar{\zb}^{k+1} - \hat{\zb}^k) & \vspace{1ex}\\
&\hat{\lbd}^{k+1} := \hat{\lbd}^k - \eta_k(\Ab\tilde{\xb}^{k+1} + \Bb\tilde{\yb}^{k+1} - \cb). &
\end{array}\right.
\end{equation}
The parameter $\tau_k$ and $\rho_k$ are updated as in Algorithm \ref{alg:A0}, but $\gamma_k$ and $\beta_k$ are updated as
\begin{equation}\label{eq:update_etak}
\gamma_k := 2\rho_k\norms{\Ab}^2~~~ \text{and}~~~\beta_k := 2\rho_k\norms{\Bb}^2.
\end{equation}
The convergence of the parallel variant \eqref{eq:admm_scheme1c} and \eqref{eq:update_etak} is stated in the following corollary whose proof is given in Appendix \ref{apdx:co:convergence1c}.

%%% Corollary 3.2.
\begin{corollary}\label{co:convergence1c}
Let $\sets{\bar{\zb}^k}$ be the sequence generated by \eqref{eq:admm_scheme1c} and \eqref{eq:update_etak}.
Then, the conclusions of Theorem \ref{th:admm-convergence1} still hold with $\bar{R}_0^2 := \rho_0\norms{\Ab}^2\norms{\bar{\xb}^0 - \xb^{\star}}^2 + \rho_0\norms{\Bb}^2\norms{\bar{\yb}^0 - \yb^{\star}}^2 +  \tfrac{1}{\rho_0}\big(2\norms{\lbd^{\star}} - \norms{\hat{\lbd}^0}\big)^2$.
\end{corollary}
From Corollary \ref{co:convergence1c}, we can show that the optimal choice of $\rho_0$ is 
\begin{equation*}
\rho_0 := \frac{\vert 2\norms{\lbd^{\star}} - \norms{\hat{\lbd}^0}\vert }{\big(\norms{\Ab}^2\norms{\bar{\xb}^0 - \xb^{\star}}^2 + \norms{\Bb}^2\norms{\bar{\yb}^0 - \yb^{\star}}^2  \big)^{1/2}},
\end{equation*}
which again depends on $\lambda^{\star}$ and $\zb^{\star}$.
We note that the nonergodic convergence rate of a linearized ADMM algorithm has been studied in \cite{li2016accelerated}.
However, our scheme \eqref{eq:admm_scheme1c} allows us to compute $\bar{\xb}^{k+1}$ and $\bar{\yb}^{k+1}$ in parallel instead of alternating as in \cite{li2016accelerated}.
This is a major advantage compared to \cite{li2016accelerated}, especially when $f$ is separable as we can see in Subsection \ref{subsec:separable_case}.
Moreover, our analysis here is much simpler, and we provide a more cleaner update for parameters.

%%% 4. PLADMM for the partially strong convexity.
\beforesec
\section{Primal-dual augmented Lagrangian-based algorithms under strong convexity}\label{sec:padmm1b}
\aftersec
We develop two primal-dual augmented Lagrangian-based algorithms to handle strongly convex case.
The first one handles the case  when one objective function $f$ or $g$ is strongly convex.
It can be viewed as a variant of ADMM.
The second algorithm tackles the case when $f$ and $g$ are both strongly convex, which is again a variant of the primal-decomposition scheme \eqref{eq:admm_scheme1c}.

%%%% 4.3. Non-ergodic partially linearized ADMM variant.
\beforesubsec
\subsection{Preconditioned ADMM: Either $f$ or $g$ is strongly convex}\label{subsec:nonergodic_admm2}
\aftersubsec
We often meet  problem instances of \eqref{eq:constr_cvx}, where there is only one objective function $f$ or $g$ is strongly convex.
In this case, the entire problem is nonstrongly convex.
Without loss of generality, we assume that $g$ is strongly convex with the strong convexity parameter $\mu_g > 0$.
We propose a new method to solve  \eqref{eq:constr_cvx} for $\mu_g$-strongly convex $g$, which can achieves $\BigO{\frac{1}{k^2}}$ convergence rate in either ergodic or non-ergodic sense.

The proposed algorithm is a combination of four techniques: alternating direction, Nesterov and Tseng's acceleration, linearization, and adaptive strategies.
We first alternate between $x$ and $y$.
The subproblem in $y$ is linearized in order to use the proximal operator of $g$.
Then, we inject Nesterov's acceleration step into $x$, while combining Tseng's acceleration step in $y$.
The complete algorithm is described in Algorithm \ref{alg:A0b}.

%%%% Algorithm 3.
\begin{algorithm}[hpt!]\caption{{\!}(\texttt{PADMM} for solving \eqref{eq:constr_cvx} with strongly convex $g$){\!\!\!\!}}\label{alg:A0b}
\begin{normalsize}
\begin{algorithmic}[1]
	\STATE{\hskip0ex}\textbf{Initialization:} 
	\STATE{\hskip2ex}Choose $\bar{\zb}^0 := (\bar{\xb}^0, \bar{\yb}^0) \in \dom{F}$, $\hat{\lbd}^0 \in\R^{n}$,  $\rho_0 \in \left(0, \frac{\mu_g}{4\norms{\Bb}^2}\right]$ and $\gamma_0 \geq 0$. 
	\vspace{0.5ex}
	\STATE{\hskip2ex}Initialize  $\tau_0 := 1$ and $\tilde{\zb}^0 := \bar{\zb}^0$.
	\vspace{0.5ex}
	\STATE\textbf{For $k := 0$ to $k_{\max}$ perform}
		\vspace{0.5ex}
                 \STATE{\hskip2ex}\label{step:a3:param_update}
                 Update $\rho_k :=  \frac{\rho_0}{\tau_k^2}$, ~$\gamma_k := \gamma_0$, ~$\beta_k := 2\rho_k\norms{\Bb}^2$, and $\eta_k := \frac{\rho_0}{2\tau_k}$.
          	\vspace{0.5ex}
		\STATE{\hskip2ex}\label{step:scvx_admm_step}
		Update 
		$\left\{\begin{array}{ll}
		\hat{\zb}^k &:= (1-\tau_k)\bar{\zb}^k + \tau_k\tilde{\zb}^k \vspace{1ex}\\
		\bar{\xb}^{k+1} &:= \Sc_{\gamma_k}(\hat{\zb}^k, \hat{\lbd}^k; \rho_k) \vspace{1ex}\\
		\tilde{\xb}^{k+1} &:= \tilde{\xb}^k + \tfrac{1}{\tau_k}(\bar{\xb}^{k+1} - \hat{\xb}^k) \vspace{1ex}\\
		\tilde{\yb}^{k+1} &:= \kprox{g/(\tau_k\beta_k)}{\tilde{\yb}^k - \tfrac{1}{\tau_k\beta_k}\Bb^{\top}\big(\rho_k(\Ab\bar{\xb}^{k+1} + \Bb\hat{\yb}^k - \cb) - \hat{\lbd}^k\big)} \vspace{1ex}\\
		\hat{\lbd}^{k+1} &:= \hat{\lbd}^k - \eta_k(A\tilde{x}_{k+1} + B\tilde{y}_{k+1} - \cb).
                 \end{array}\right.$
		\vspace{0.5ex}
    	         \STATE\label{step:scvx_admm_step2} Update $\bar{\yb}^{k+1}$ using \textbf{one} of the following two \textbf{options}:                 
		\[\left[\begin{array}{lll}
		\bar{\yb}^{k\!+\!1} &:= (1-\tau_k)\bar{\yb}^k + \tau_k\tilde{\yb}^{k+1} & \text{(Averaging step)} \vspace{1ex}\\
		\bar{\yb}^{k\!+\!1} &:= \kprox{g/(\rho_k\norms{\Bb}^2)}{\hat{\yb}^k \! -\! \tfrac{1}{\rho_k\norms{\Bb}^2}\Bb^{\top}\big(\rho_k(\Ab\bar{\xb}^{k+1} + \Bb\hat{\yb}^k \!-\! \cb) \!-\! \hat{\lbd}^k\big)} &\text{(Proximal step)}.
		\end{array}\right.\]
		\STATE{\hskip2ex}\label{step:a3:update_tau}Update $\tau_{k+1} := \frac{\tau_k}{2}\left(\sqrt{\tau_k^2 + 4} - \tau_k\right)$.
	\STATE\textbf{End~for}
\end{algorithmic}
\end{normalsize}
\end{algorithm}

Before analyzing the convergence of Algorithm \ref{alg:A0b}, we make the following remarks.
\begin{itemize}
\item[$\mathrm{(a)}$]
First, Algorithm \ref{alg:A0b} linearizes the $y$-subproblem to reduce the per-iteration complexity as in Algorithm \ref{alg:A0}.
Step \ref{step:scvx_admm_step} of  Algorithm \ref{alg:A0b} combines both Nesterov's acceleration step \cite{Nesterov1983} in $x$ and  Tseng's variant \cite{tseng2008accelerated} in $y$.

\item[$\mathrm{(b)}$]
Second, we can update $\bar{\yb}^{k+1}$ with two different options. One can take a weighted averaging without incurring much extra cost.
The other is to compute an additional proximal operator of $g$, which requires additional cost but can avoid averaging.

\item[$\mathrm{(c)}$]
Third, we can use different update rules for parameters in Algorithm \ref{alg:A0b}.
These update rules can be derived from the conditions \eqref{eq:th41_param_coditions6} of  Lemma \ref{le:th41_step1}.
For simplicity of presentation, we only provide one concrete update as in Algorithm \ref{alg:A0b}.

\end{itemize}
The following theorem estimates a global convergence rate of Algorithm \ref{alg:A0b} whose proof can be found in Appendix~\ref{apdx:convergence_analysis_A0b}.

%%% Theorem 4.2.
\begin{theorem}\label{th:convergence3b}
Assume that $g$ is $\mu_g$-strongly convex with $\mu_g > 0$.
Let $\sets{\bar{\zb}^k}$ be the sequence generated by Algorithm \ref{alg:A0b}.
Then the following guarantees hold:
\begin{equation}\label{eq:convergence3b}
\vert F(\bar{\zb}^k) - F^{\star}\vert \leq \frac{2\bar{R}_0^2}{(k+2)^2}~~~~\text{and}~~~~~\norms{\Ab\bar{\xb}^k + \Bb\bar{\yb}^k - c} \leq \frac{2\bar{R}_0^2}{\norms{\lbd^{\star}}(k+2)^2},
\end{equation}
where $\bar{R}_0^2 := \frac{2}{\rho_0}\big(2\norms{\lbd^{\star}} - \norms{\hat{\lbd}^0}\big)^2 + \gamma_0\norms{\bar{\xb}^0 - \xb^{\star}}^2  + 2\rho_0\norms{\Bb}^2\norms{\bar{\yb}^0 - \yb^{\star}}^2$.

Consequently, $\set{\bar{\zb}^k}$  converges to a solution $\zb^{\star}$ of \eqref{eq:constr_cvx} at  $\BigO{\frac{1}{k^2}}$-rate either in an ergodic sense if the \textbf{averaging step} is used or in a non-ergodic sense if the \textbf{proximal step} is used.
\end{theorem}

%%%% Parallel Linearized ADMM variant: when both f and g are strongly convex.
\beforesubsec
\subsection{Parallel primal-dual decomposition algorithm: Both $f$ and $g$ are strongly convex}\label{subsec:parallel_admm1b}
\aftersubsec
When both $f$ and $g$ are strongly convex, i.e., $f$ is $\mu_f$-strongly convex and $g$ is $\mu_g$-strongly convex with $\mu_f > 0$ and $\mu_g > 0$, respectively, we can modify Algorithm \ref{alg:A0b} to obtain the following primal-dual decomposition scheme:
\begin{equation}\label{eq:admm_scheme3_admm2}
\left\{ \begin{array}{lll}
&\hat{\zb}^k ~~~:= (1-\tau_k)\bar{\zb}^k + \tau_k\tilde{\zb}^k &\vspace{1ex}\\
&\hat{\ub}^k ~~~:= \rho_k(\Ab\hat{\xb}^k + \Bb\hat{\yb}^k - \cb) - \hat{\lbd}^k &\vspace{1ex}\\
&{\!\!\!\!}\left.\begin{array}{l}
\tilde{\xb}^{k+1} := \kprox{f/(\tau_k\gamma_k)}{\tilde{\xb}^k - \tfrac{1}{\tau_k\gamma_k}\Ab^{\top}\hat{\ub}^k} \vspace{1ex}\\
\tilde{\yb}^{k+1} := \kprox{g/(\tau_k\beta_k)}{\tilde{\yb}^k - \tfrac{1}{\tau_k\beta_k}\Bb^{\top}\hat{\ub}^k}\vspace{1ex}\\
\end{array}\right] &\text{(Parallel step)}\vspace{1ex}\\
&\hat{\lbd}^{k+1} := \hat{\lbd}^k - \eta_k(A\tilde{\xb}^{k+1} + B\tilde{\yb}^{k+1} - \cb). &
\end{array}\right.
\end{equation}
Then, we update $\bar{\zb}^{k+1}$ based on one of the following two options:
\begin{equation}\label{eq:admm_scheme3b_admm2}
\left[\begin{array}{lll}
&\bar{\zb}^{k+ 1} ~~~~:= (1-\tau_k)\bar{\zb}^k + \tau_k\tilde{\zb}^{k+1} & \text{(Averaging step)} \vspace{1ex}\\
&\left\{\begin{array}{l}
\bar{\xb}^{k+ 1} := \kprox{f/(\rho_k\norms{\Ab}^2)}{\hat{\xb}^k  -  \tfrac{1}{\rho_k\norms{\Ab}^2}\Ab^{\top}\hat{\ub}^k}\vspace{1ex}\\
\bar{\yb}^{k+1} := \kprox{g/(\rho_k\norms{\Bb}^2)}{\hat{\yb}^k  -  \tfrac{1}{\rho_k\norms{\Bb}^2}\Bb^{\top}\hat{\ub}^k}
\end{array}\right] &\text{(Parallel proximal step).}
\end{array}\right.
\end{equation}
The parameters are updated similarly as in Algorithm \ref{alg:A0b}. That is
\begin{equation}\label{eq:co41_param_update10}
\left\{\begin{array}{ll}
&\tau_{k+1} := \frac{\tau_k}{2}\big[ (\tau_k^2 + 4)^{1/2} - \tau_k\big]~~\text{with}~\tau_0 := 1,\vspace{1ex}\\
&\rho_k := \frac{\rho_0}{\tau_k^2},~~~\text{with}~ \rho_0 \in \left(0, \min\set{\frac{\mu_f}{4\norms{\Ab}^2}, \frac{\mu_g}{4\norms{\Bb}^2}}\right], \vspace{1ex}\\
&\gamma_k := 2\rho_k\norms{\Ab}^2,~~\beta_k := 2\rho_k\norms{\Bb}^2,~\text{and}~~\eta_k := \frac{\rho_0}{2\tau_k}.
\end{array}\right.
\end{equation}
The convergence of this variant is stated in the following corollary, whose proof is similar to Theorem \ref{th:convergence3b} and we briefly present it in Appendix \ref{apdx:co:convergence3c}.

%%% Corollary 4.2.
\begin{corollary}\label{co:convergence3c}
Assume that $f$ is $\mu_f$-strongly convex with $\mu_f > 0$, and $g$ is $\mu_g$-strongly convex with $\mu_g > 0$ in \eqref{eq:constr_cvx}.
Let $\sets{(\bar{\zb}^k, \hat{\lbd}^k)}$ be the sequence generated by Algorithm \ref{alg:A0b} using \eqref{eq:admm_scheme3_admm2} and \eqref{eq:admm_scheme3b_admm2} with $0 < \rho_0 \leq \min\set{\frac{\mu_f}{4\norms{\Ab}^2}, \frac{\mu_g}{4\norms{\Bb}^2}}$.
Then the conclusion of Theorem \ref{th:convergence3b} still hold with $\bar{R}_0^2 := \frac{2}{\rho_0}\big(2\norms{\lbd^{\star}} - \norms{\hat{\lbd}^0}\big)^2 + 2\rho_0\norms{\Ab}^2\norms{\bar{\xb}^0 - \xb^{\star}}^2  + 2\rho_0\norms{\Bb}^2\norms{\bar{\yb}^0 - \yb^{\star}}^2$.
\end{corollary}

The per-iteration complexity of the variant \eqref{eq:admm_scheme3_admm2}-\eqref{eq:admm_scheme3b_admm2} is better than that of Algorithm~\ref{alg:A0b} if $\Ab$ is non-orthogonal.
Each iteration of the variant \eqref{eq:admm_scheme3_admm2}-\eqref{eq:admm_scheme3b_admm2} only requires the proximal operator of $f$ and $g$, and $Ax$, $A^{\top}u$, $By$, and $B^{\top}v$.
Moreover, the computation of both $\tilde{\zb}^k$ and $\bar{\zb}^k$ can be carried out in parallel. 
Note that the variant  \eqref{eq:admm_scheme3_admm2}-\eqref{eq:admm_scheme3b_admm2} achieves the same $\BigO{\frac{1}{k^2}}$-rate as known from the literature, but this algorithmic variant is new and has a non-ergodic rate guarantee compared to \cite{necoara2014iteration}.
The strong convexity assumption in Theorem \ref{th:convergence3b} and Corollary \ref{co:convergence3c} can be replaced by a weaker condition called ``quasi-strong convexity'' assumption in \cite{necoara2015linear}.

%%%%%%%%%%%%%%%%%%%%%%%%%%%%%%%%%%%%%%%%%%%%%
%%%% 5. Extensions
%%%%%%%%%%%%%%%%%%%%%%%%%%%%%%%%%%%%%%%%%%%%%
\beforesec
\section{Extensions}\label{sec:extensions}
\aftersec
We can extend Algorithm~\ref{alg:A0} and Algorithm \ref{alg:A0b} and their variants to handle more general problems than \eqref{eq:constr_cvx}.
We consider two extensions in the following subsections: smooth + nonsmooth objectives and separable settings.

\beforesubsec
\subsection{Smooth + Nonsmooth objective functions}
\aftersubsec
In this extension, we consider \eqref{eq:constr_cvx} with the objective function $F$ defined as 
\begin{equation*}
F(\zb) := \underbrace{f_1(\xb) + f_2(\xb)}_{f(\xb)} + \underbrace{g_1(\yb) + g_2(\yb)}_{g(\yb)}, 
\end{equation*}
where $f_1$ and $g_1$ are smooth with $L_{f_1}$- and $L_{g_1}$-Lipschitz gradient, respectively, and $f_2$ and $g_2$ are proper, closed, and convex with tractably proximal operators.
In this case, two subproblems in Algorithm~\ref{alg:A0} become
\begin{equation*}
\left\{\begin{array}{l}
\bar{\xb}^{k\!+\!1} := \displaystyle\argmin_x\Big\{ f_2(\xb) + \iprods{\nabla{f_1}(\hat{\xb}^k) \!-\! \Ab^{\top}\hat{\lbd}^k, \xb \!-\! \hat{\xb}^k} + \tfrac{\rho_k}{2}\norms{\Ab\xb \!+\! \Bb\hat{\yb}^k \!-\! \cb}^2 + \tfrac{\hat{\gamma}_k}{2}\norms{\xb \!-\! \hat{\xb}^k}^2 \Big\} \vspace{1ex}\\
\bar{\yb}^{k\!+\!1} := \displaystyle\argmin_y\Big\{ g_2(\yb) \!+\! \iprods{\nabla{g_1}(\hat{\yb}^k) \!+\! \Bb^{\top}\big(\rho_k(\Ab\bar{\xb}^{k\!+\!1} \!+\! \Bb\hat{\yb}^k \!-\! \cb) - \hat{\lbd}^k\big), \yb \!-\! \hat{\yb}^k}  \!+\! \tfrac{\hat{\beta}_k}{2}\norms{\yb \!-\! \hat{\yb}^k}^2 \Big\},
\end{array}\right.
\end{equation*}
where $\hat{\gamma}_k := \gamma_k + L_{f_1}$ and $\hat{\beta}_k := \beta_k + L_{g_1}$.
For the variant \eqref{eq:admm_scheme1c}, we can linearize the augmented terms again while keeping other parts as in these two subproblems.
Then, we can adapt Algorithm \ref{alg:A0} and its variant \eqref{eq:admm_scheme1c} as well as Algorithm \ref{alg:A0b} and its variant \eqref{eq:admm_scheme3_admm2}-\eqref{eq:admm_scheme3b_admm2} to solve this problem.
The convergence guarantees of these variants are very similar to Theorem \ref{th:admm-convergence1}, Corollary \ref{co:convergence1c} as well as Theorem \ref{th:convergence3b} and Corollary \ref{co:convergence3c}.
Hence, we omit the details here.

\beforesubsec
\subsection{Separable constrained convex optimization}\label{subsec:separable_case}
\aftersubsec
The parallel variants suggest that we can extend the schemes \eqref{eq:admm_scheme1c}-\eqref{eq:update_etak} and \eqref{eq:admm_scheme3_admm2}-\eqref{eq:admm_scheme3b_admm2} to solve the following separable problem:
\begin{equation}\label{eq:Fz_decomp}
F^{\star} := \min_{\zb} \Big\{ F(\zb) := \sum_{i=1}^Nf_i(\zb_{[i]}) ~~\mid~~\sum_{i=1}^N\Ab_i\zb_{[i]} = \cb \Big\},
\end{equation}
where $f_i$ has a tractably proximal operator for $i=1,\cdots, N$.

When $f_i$ is nonstrongly convex, we can apply \eqref{eq:admm_scheme1c}-\eqref{eq:update_etak} to solve \eqref{eq:Fz_decomp}, where the subproblems can be solved \textbf{in parallel} for $i=1,\cdots, N$ as
\begin{equation*}
{\!\!}\left\{\begin{array}{ll}
\hat{\ub}^k &:= \rho_k\big(\sum_{i=1}^N\Ab_i\hat{\zb}^k_{[i]} - \cb \big) - \hat{\lbd}^k \vspace{1ex}\\
\bar{\zb}^{k+1}_{[i]} {\!\!\!\!}&:= \argmin_{\zb_{[i]}}\Big\{ f_i(\zb_{[i]}) - \iprods{\Ab_i^{\top}\hat{\ub}^k, \zb_{[i]} - \hat{\zb}_{[i]}^k} + \tfrac{\gamma_k}{2}\norms{\zb_{[i]} - \hat{\zb}^k_{[i]}}^2 \Big\} \equiv \prox_{f_i/\gamma_k}\big(\hat{\zb}_{[i]}^k - \tfrac{1}{\gamma_k}\Ab_i^{\top}\hat{\ub}^k\big).
\end{array}\right.{\!\!\!}
\end{equation*}
When $f_i$ is $\mu_{f_i}$-strongly convex, we can apply \eqref{eq:admm_scheme3_admm2}-\eqref{eq:admm_scheme3b_admm2} to solve \eqref{eq:Fz_decomp}, where the subproblems can be solved \textbf{in parallel} for $i=1,\cdots, N$ as
\begin{equation*}
\bar{\zb}^{k+1}_{[i]} = \argmin_{\zb_{[i]}}\Big\{ f_i(\zb_{[i]}) - \iprods{\Ab_i^{\top}\hat{\ub}^k, \zb_{[i]} - \hat{\zb}^k_{[i]}} + \tfrac{\beta_k\tau_k}{2}\norms{\zb_{[i]} - \tilde{\zb}^k_{[i]}}^2 \Big\} \equiv \prox_{f_i/(\tau_k\beta_k)}\big(\tilde{\zb}_{[i]}^k - \tfrac{1}{
\tau_k\beta_k}\Ab_i^{\top}\hat{\ub}^k\big).
\end{equation*}
The other steps remain the same as in these original algorithms.
Since convergence analysis of these extensions follows the same arguments of the proof of Theorems \ref{th:admm-convergence1} and \ref{th:convergence3b} and Corollaries \ref{co:convergence1c} and \ref{co:convergence3c}, we omit the details.

%%%%%%%%%%%%%%%%%%%%%%%%%%%%%%%%%%%%%%%%%%%%%
%%%% 5. Connection to primal-dual first-order methods
%%%%%%%%%%%%%%%%%%%%%%%%%%%%%%%%%%%%%%%%%%%%%
\beforesec
\section{Connection to primal-dual first-order methods}\label{sec:connection1}
\aftersec
Primal-dual first-order methods for solving convex optimization problems become extremely popular in recent years.
Among these, Chambolle-Pock's method \cite{Chambolle2011} and primal-dual hybrid gradient algorithms \cite{Esser2010a,Goldstein2013} are perhaps the most notable ones.
In this section, we derive two variants of Algorithm \ref{alg:A0} and Algorithm \ref{alg:A0b}, respectively, to solve composite convex optimization problems.
We show how these variants relate to the primal-dual first-order methods.

We consider the following composite convex optimization problem with linear operator:
\begin{equation}\label{eq:comp_cvx}
F^{\star} := \min_{\yb\in\R^{p_2}}\set{ F(\yb) := f(\Bb\yb) + g(\yb) },
\end{equation}
where $f : \R^{n}\to\Rext$ and $g : \R^{p_2}\to\Rext$ are proper, closed, and convex, and $\Bb$ is a linear bounded operator from $\R^{p_2} \to\R^{n}$.
By introducing $\xb = \Bb\yb$, we can reformulate \eqref{eq:comp_cvx} into \eqref{eq:constr_cvx} with $F(\yb) = F(\zb) = f(\xb) + g(\yb)$ and a linear constraint $\xb - \Bb\yb = 0$.

Let us apply Algorithm \ref{alg:A0} to solve the constrained reformulation of \eqref{eq:comp_cvx}.
Since $\Ab = \Id$, we can choose $\gamma_k = 0$.
Hence, the main step of this variant becomes
\begin{equation*}
\left\{\begin{array}{ll}
\bar{\xb}^{k+1} &:= \kprox{f/\rho_k}{\Bb\hat{\yb}^k + \rho_k^{-1}\hat{\lbd}^k}  = \frac{1}{\rho_k}\big[\rho_k\Bb\hat{\yb}^k  + \hat{\lbd}^k - \kprox{\rho_kf^{\ast}}{\rho_k\Bb\hat{\yb}^k - \hat{\lbd}^k} \big], \vspace{1ex}\\
\bar{\yb}^{k+1} &:= \kprox{g/\beta_k}{\hat{\yb}^k - \tfrac{1}{\beta_k}\Bb^{\top}\big(\rho_k(\Bb\hat{\yb}^k - \bar{\xb}^{k+1}) + \hat{\lbd}^k\big)}. \vspace{1ex}\\
\end{array}\right.
\end{equation*}
Here, we use  the Moreau  identity $\kprox{\gamma f}{\vb} + \gamma \kprox{f^{\ast}/\gamma}{\vb/\gamma} = \vb$ of proximal operators.

Let $\breve{\xb}^{k+1} := \kprox{\rho_kf^{\ast}}{\hat{\lbd}^k + \rho_k\Bb\hat{\yb}^k}$ and $\hat{\lbd}^0 := \boldsymbol{0}^n$.
Then, after a few elementary arrangements, we arrive at the following scheme:
\begin{equation}\label{eq:pd_scheme1}
\left\{\begin{array}{ll}
\breve{\xb}^{k+1} &:= \kprox{\rho_kf^{\ast}}{\hat{\lbd}^k + \rho_k\Bb\hat{\yb}^k} \vspace{1ex}\\
\bar{\yb}^{k+1} &:= \kprox{g/\beta_k}{\hat{\yb}^k - \frac{1}{\beta_k}\Bb^{\top}\breve{\xb}^{k+1}} \vspace{1ex}\\
\bar{\xb}^{k+1} &:= \Bb\hat{\yb}^k  + \frac{1}{\rho_k}(\hat{\lbd}^k - \breve{\xb}^{k+1}) \vspace{1ex}\\
\hat{\yb}^{k+1} &:= \bar{\yb}^{k+1} + \frac{(1-\tau_k)\tau_{k+1}}{\tau_k}(\bar{\yb}^{k+1} - \bar{\yb}^k) \vspace{1ex}\\
\hat{\lbd}^{k+1} &:= \hat{\lbd}^k - \frac{\eta_k}{\tau_k}(\bar{\xb}^{k+1} - (1-\tau_k)\bar{\xb}^k - \Bb(\bar{\yb}^{k+1} - (1-\tau_k)\bar{\yb}^k)).
\end{array}\right.
\end{equation}
The parameters are updated as in Algorithm \ref{alg:A0}.
Hence, this scheme solves \eqref{eq:comp_cvx}.

If $g$ is $\mu_g$-strongly convex with $\mu_g > 0$, then we can apply Algorithm \ref{alg:A0b} to solve the constrained reformulation of \eqref{eq:comp_cvx}.
The main step of this variant becomes
\begin{equation}\label{eq:pd_scheme2}
\left\{\begin{array}{ll}
\hat{\yb}^k &:= (1-\tau_k)\bar{\yb}^k + \tau_k\tilde{\yb}^k\vspace{1ex}\\
\breve{\xb}^{k+1} &:= \kprox{\rho_kf^{\ast}}{\hat{\lbd}^k + \rho_k\Bb\hat{\yb}^k} \vspace{1ex}\\
\tilde{\yb}^{k+1} &:= \kprox{g/(\tau_k\beta_k)}{\tilde{\yb}^k - \frac{1}{\tau_k\beta_k}\Bb^{\top}\breve{\xb}^{k+1}} \vspace{1ex}\\
%\tilde{\xb}^{k+1} &:= \tilde{\xb}^k + \frac{1}{\tau_k}\left(\bar{\xb}^{k+1} - \hat{\xb}^k\right) \vspace{1ex}\\
\bar{\xb}^{k+1} &:= \Bb\hat{\yb}^k + \frac{1}{\rho_k}(\hat{\lbd}^k - \breve{\xb}^{k+1}) \vspace{1ex}\\
\bar{\yb}^{k+1} &:= (1-\tau_k)\bar{\yb}^k + \tau_k\tilde{\yb}^{k+1} \vspace{1ex}\\
%\hat{\xb}^{k+1} &:= \bar{\xb}^{k+1} + \frac{(1-\tau_k)\tau_{k+1}}{\tau_k}(\bar{\xb}^{k+1} - \bar{\xb}^k) \vspace{1ex}\\
\hat{\lbd}^{k+1} &: = \hat{\lbd}^k - \frac{\eta_k}{\tau_k}(\bar{\xb}^{k+1} - (1-\tau_k)\bar{\xb}^k) + \eta_k\Bb\tilde{\yb}^{k+1}.
\end{array}\right.
\end{equation}
The parameters are updated as in Algorithm \ref{alg:A0b}.
Clearly, we can view both \eqref{eq:pd_scheme1} and \eqref{eq:pd_scheme2} as primal-dual methods for solving \eqref{eq:comp_cvx}.
By eliminating some intermediate steps, we can show that the per-iteration complexity of these schemes remain essentially the same as existing primal-dual methods.
We believe that these schemes are new in the literature.
Note that, in \eqref{eq:pd_scheme2} we only choose the \textbf{averaging step} in Algorithm \ref{alg:A0b}.
We can certainly choose the \textbf{proximal step} to avoid averaging, but it requires an additional proximal operator of $g$. 

The convergence of both schemes \eqref{eq:pd_scheme1} and \eqref{eq:pd_scheme2} is summarized into the following theorem. 

%% Theorem 5.1.
\begin{theorem}\label{th:convergence5}
Let $f$ in \eqref{eq:comp_cvx} be $L_f$-Lipschitz continuous on $\dom{F}$, i.e., $\vert f(\xb) - f(\hat{\xb})\vert \leq L_f\norm{\xb - \hat{\xb}}$ for all $\xb,\hat{\xb}\in\dom{F}$ and \eqref{eq:comp_cvx}  has an optimal solution $\yb^{\star}$.
\begin{itemize}
\item[]$\mathrm{(a)}$~Let $\set{\bar{\yb}^k}$ be the sequence generated by \eqref{eq:pd_scheme1}. Then
\begin{equation}\label{eq:convergence5a}
F(\bar{\yb}^k) - F(\yb^{\star}) \leq  \frac{2\rho_0^2\norms{\Bb}^2\norms{\bar{\yb}^0 - \yb^{\star}}^2 + 2L_f^2}{\rho_0k}.
\end{equation}
\item[]$\mathrm{(b)}$~
Assume additionally that $g$ is $\mu_g$-strongly convex with $\mu_g > 0$,  then the sequence $\set{ \bar{\yb}^k}$ generated by \eqref{eq:pd_scheme2} satisfies
\begin{equation}\label{eq:convergence5b}
F(\bar{\yb}^k) - F(\yb^{\star}) \leq  \frac{2L_f^2 +  8\rho_0^2\norms{\Bb}^2\norms{\bar{\yb}^0 - \yb^{\star}}^2}{\rho_0(k+2)^2}.
\end{equation}
\end{itemize}
\end{theorem}

%%% The proof of Theorem 5.1.
\begin{proof}
With $F^{\star} = F(\yb^{\star})$, we note that $F(\bar{\yb}^k) - F^{\star} = f(\Bb\bar{\yb}^k) + g(\bar{\yb}^k) - F^{\star} =  f(\Bb\bar{\yb}^k) - f(\bar{\xb}^k) + f(\bar{\xb}^k) + g(\bar{\yb}^k) - F^{\star} \leq f(\bar{\xb}^k) + g(\bar{\yb}^k) - F^{\star} + L_f\norms{\bar{\xb}^k - \Bb\bar{\yb}^k} = F(\bar{\zb}^k) - F^{\star} + L_f\norms{\bar{\xb}^k - \Bb\bar{\yb}^k}$.
Now, we apply Theorem \ref{th:admm-convergence1} to \eqref{eq:pd_scheme1} with $\hat{\lbd}^0 := \boldsymbol{0}^n$, $\gamma_0 := 0$, and $\xb^{\star} = \Bb\yb^{\star}$, we obtain
\begin{equation*}
\Lc_{\rho_k}(\bar{\zb}^{k+1},\lbd) - F^{\star} \leq \tfrac{1}{k+1}\big[\rho_0\norms{\Bb}^2\norms{\bar{\yb}^0 - \yb^{\star}}^2 +  \tfrac{1}{\rho_0}\norms{\lbd}^2\big].
\end{equation*}
From Lemma~\ref{le:approx_opt_cond}, we use $\rho = L_f$ to obtain
\begin{align*}
F(\bar{\yb}^k) - F^{\star}  &\leq F(\bar{\zb}^k) - F^{\star} + L_f\norms{\bar{\xb}^k - \Bb\bar{\yb}^k}  \leq \sup_{\norms{\lbd} \leq L_f}\set{\Lc_{\rho_k}(\bar{\zb}^{k},\lbd) - F^{\star} } + L_f\norms{\bar{\xb}^k - \Bb\bar{\yb}^k} \vspace{1ex}\\
&\leq \sup_{\norms{\lbd} \leq L_f}\set{\frac{1}{k}\left[\rho_0\norms{\Bb}^2\norms{\bar{\yb}^0 - \yb^{\star}}^2 +  \tfrac{\rho_0}{4}\norms{\lbd}^2\right] } + L_f\norms{\bar{\xb}^k - \Bb\bar{\yb}^k}  \vspace{1ex}\\
&= \frac{2}{k}\left[\rho_0\norms{\Bb}^2\norms{\bar{\yb}^0 - \yb^{\star}}^2 +  \tfrac{1}{\rho_0}L_f^2\right],
\end{align*}
which leads to \eqref{eq:convergence5a}.

From the proof of Theorem \ref{th:convergence3b} with $\hat{\lbd}^0 := \boldsymbol{0}^n$, $\gamma_0 := 0$, and $\xb^{\star} = \Bb\yb^{\star}$, we obtain
\begin{equation*} 
\Lc_{\rho_k}(\bar{\zb}^{k+1},\lbd) - F(\zb^{\star}) \leq \frac{1}{(k+2)^2}\left[\tfrac{1}{\rho_0}\norms{\lbd}^2 +  4\rho_0\norms{\Bb}^2\norm{\bar{\yb}^0 - \yb^{\star}}^2\right].
\end{equation*}
Using $\rho = L_f$ from Lemma~\ref{le:approx_opt_cond}, the last inequality implies
\begin{align*}
F(\bar{\yb}^k) - F^{\star}  &\leq F(\bar{\zb}^k) - F^{\star} + L_f\norms{\bar{\xb}^k - \Bb\bar{\yb}^k}  \leq \sup_{\norms{\lbd} \leq L_f}\set{\Lc_{\rho_k}(\bar{\zb}^{k},\lbd) - F^{\star} } + L_f\norms{\bar{\xb}^k - \Bb\bar{\yb}^k}\vspace{1ex}\\
&\leq \frac{2}{(k+2)^2}\left[ \tfrac{1}{\rho_0}L_f^2 +  4\rho_0\norms{\Bb}^2\norms{\bar{\yb}^0 - \yb^{\star}}^2\right],
\end{align*}
which leads to \eqref{eq:convergence5b}.
\end{proof}
%%% End of the proof.

%We note that we can customize other schemes developed above to solve \eqref{eq:comp_cvx}.
%For simplicity of presentation, we skip the details of these variants.

%%%%%%%%%%%%%%%%%%%%%%%%%%%%%%%%%%
%%% 7. Numerical experiments.
%%%%%%%%%%%%%%%%%%%%%%%%%%%%%%%%%%
\beforesec
\section{Numerical experiments}\label{sec:num_experiments}
\aftersec
In this section, we provide several numerical examples in imaging science to illustrate our theoretical development.
For convenience of references, we call Algorithm \ref{alg:A0} \texttt{PADMM}, its parallel variant  \eqref{eq:admm_scheme1c} \texttt{ParPD}, and Algorithm \ref{alg:A0b} \texttt{scvx-PADMM}. 
All the experiments are implemented in Matlab R2014b, running on a MacBook Pro. Retina, 2.7GHz Intel Core i5 with 16Gb RAM.
%Our experiment is implemented in Matlab 2014b running on a MacBook Pro (Retina, 2.7 GHz Intel Core i5, 16GB 1867 MHz).

%%% 7.1. The LAD problem
\beforesubsec
\subsection{The $\ell_1$-Regularized Least Absolute Derivation (LAD)}\label{subsec:LAD}
\aftersubsec
We consider the following $\ell_1$-regularized least absolute derivation (LAD) problem widely studied in the literature:
\begin{equation}\label{eq:LAD}
F^{\star} := \min_{\yb\in\R^{p_2}}\Big\{ F(\yb) := \norms{\Bb\yb - \cb}_1 + \kappa\norms{\yb}_1 \Big\},
\end{equation}
where $\Bb\in\R^{n\times p_2}$ and $\cb\in\R^n$ are given, and $\kappa > 0$ is a regularization parameter.
This problem is completely nonsmooth. 
If we introduce $\xb := \Bb\yb - \cb$, then we can reformulate \eqref{eq:LAD} into \eqref{eq:constr_cvx} with two objective functions $f(\xb) := \norms{\xb}_1$ and $g(\yb) := \kappa\norms{\yb}_1$ and a linear constraint $-\xb + \Bb\yb = \cb$.

We use problem \eqref{eq:LAD} to verify our theoretical results presented in Theorem \ref{th:admm-convergence1}, Corollary \ref{co:convergence1c}, and Theorem \ref{th:convergence3b}.
We implement Algorithm \ref{alg:A0}, its parallel primal-dual decomposition scheme \eqref{eq:admm_scheme1c}, and Algorithm \ref{alg:A0b}.
We compare these algorithms with ASGARD \cite{TranDinh2015b} and its restarting variant, Chambolle-Pock's method \cite{Chambolle2011}, and standard ADMM \cite{Boyd2011}.
For ADMM, we reformulate \eqref{eq:LAD} into the following constrained setting:
\begin{equation*}
\min_{\xb, \yb, \zb} \Big\{ \norms{\xb}_1 + \kappa\norms{\zb}_1 ~\mid -\xb + \Bb\yb = \cb, ~\yb - \zb = 0 \Big\}
\end{equation*}
to avoid expensive subproblems.
We solve the subproblem in $x$ using a preconditioned conjugate gradient method (PCG) with at most $20$ iterations or up to $10^{-5}$ accuracy.

We generate a matrix $\Bb$ using standard Gaussian distribution $\Nc(0, 1)$ without and with correlated columns, and normalize it to get unit column norms.
The observed vector $\cb$ is generated as $\cb := \Bb\xb^{\natural} + \hat{\sigma}\Lc(0,1)$, where $\xb^{\natural}$ is a given $s$-sparse vector drawn from $\Nc(0,1)$, and $\hat{\sigma} = 0.01$ is the variance of noise generated from a Laplace distribution $\Lc(0, 1)$.
For problems of the size $(m, n, s) = (2000, 700, 100)$, we tune to get a regularization parameter $\kappa = 0.5$. 

We test these algorithms on two problem instances.
The configuration is as follows:
\begin{itemize}
\item For \texttt{PADMM} and \texttt{ParPD}, we set $\rho_0 := 5$, which is obtained by upper bounding $\frac{2\norms{\lbd^{\star}}}{\norms{\Bb}\norms{\yb^0 - \yb^{\star}}}$ as suggested by the theory.
Here, $\yb^{\star}$ and $\lbd^{\star}$ are computed with the best accuracy by using an interior-point algorithm in MOSEK.
\item For \texttt{scvx-PADMM} we set $\rho_0 = \frac{1}{4\norms{\Bb}^2}$ by choosing $\mu_g = 0.5$.
\item For Chambolle-Pock's method, we run two variants.
In the first variant, we set step-sizes $\tau = \sigma = \tfrac{1}{\norms{\Bb}}$, and in the second one we choose $\tau = 0.01$ and $\sigma = \tfrac{1}{\norms{\Bb}^2\tau}$ as suggested in \cite{Chambolle2011}, and it works better than $\tau = \frac{1}{\norms{\Bb}}$. We name these variants by \texttt{CP} and \texttt{CP-0.01}, respectively.
\item For ADMM, we tune different penalty parameters and arrive at $\rho = 10$ that works best in this experiment.
\end{itemize}
The result of two problem instances are plotted in Figure \ref{fig:LAD_run12}.
Here, \texttt{ADMM-1} and \texttt{ADMM-10} stand for ADMM with $\rho = 1$ and $\rho = 10$, respectively.
\texttt{CP} and \texttt{CP-0.01} are the first and second variants of Chambolle-Pock's method, respectively.
\texttt{ASGARD-rs} is a restarting variant of ASGARD, and \texttt{avg-} stands for the relative objective residuals evaluated at the averaging sequence in Chambolle-Pock's method and ADMM.
Note that the $\BigO{\frac{1}{k}}$-rate of these two methods is proved for this averaging sequence.
\begin{figure}[htp!]
\begin{center}
\vspace{-1ex}
\includegraphics[width=1\linewidth]{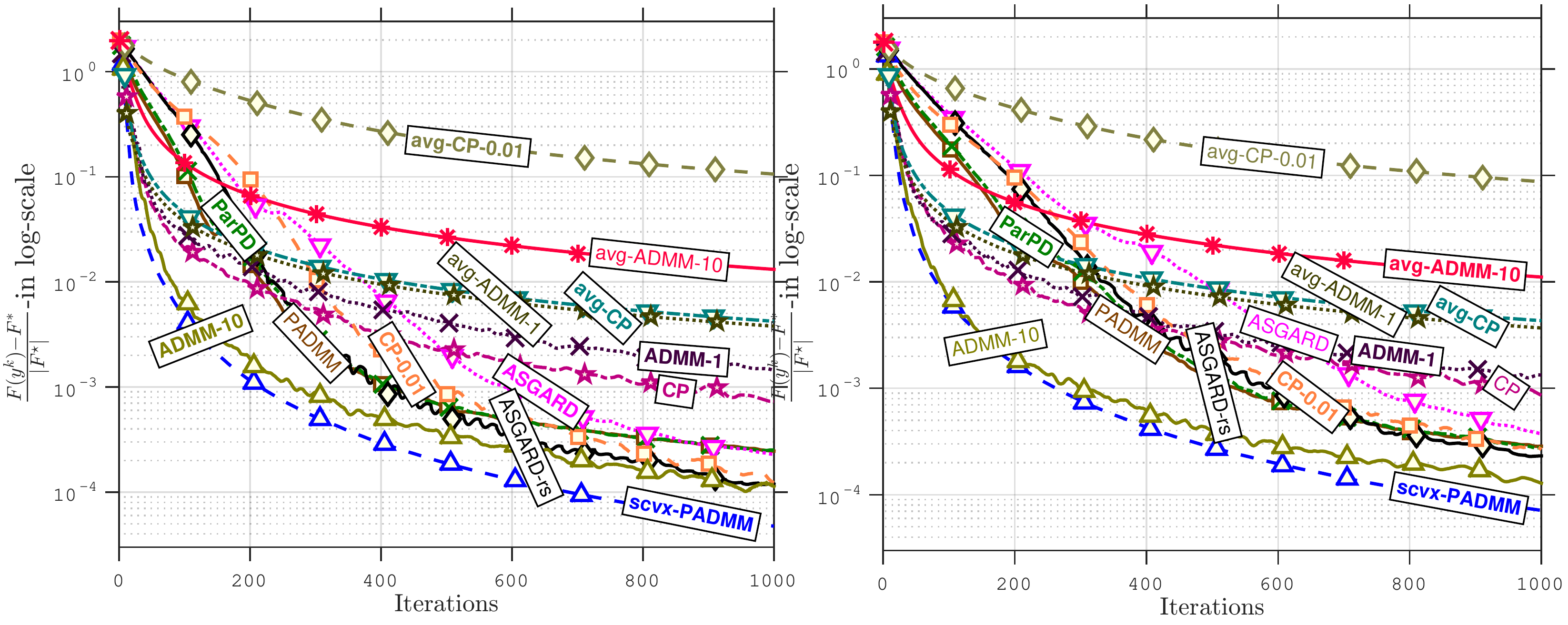}
\vspace{-4ex}
\caption{\footnotesize Convergence behavior of $9$ algorithmic variants on two instances of  \eqref{eq:LAD} after $1000$ iterations. 
Left: Without correlated columns; Right: With $50\%$ correlated columns. 
}\label{fig:LAD_run12}
\vspace{-1ex}
\end{center}
\end{figure}

We can observe from Figure \ref{fig:LAD_run12} that \texttt{scvx-PADMM} is the best. 
Both \texttt{PADMM} and \texttt{ParPD} have the same performance in this example and are comparable with \texttt{CP-0.01}, \texttt{ADMM-10} and \texttt{ASGARD-rs} in the first case, but is slightly worse than \texttt{ADMM-10} in the second case.
Note that ADMM requires to solve a linear system by PCG which is always slower than other methods including \texttt{PADMM} and \texttt{ParPD}.
\texttt{CP-0.01} works better than \texttt{CP} in late iterations but is slow in early iterations.  
\texttt{ASGARD} and \texttt{ASGARD-rs} remain comparable with \texttt{CP-0.01}.
Since both Chambolle-Pock's method and ADMM have  $\BigO{\frac{1}{k}}$-convergence rate on the averaging sequence, we also evaluate the relative objective residuals and plot them in Figure \ref{fig:LAD_run12}.
Clearly, this sequence shows its $\BigO{\frac{1}{k}}$-rate but this rate is much slower than the last iterate sequence in all cases.
It is also much slower than \texttt{PADMM} and \texttt{ParPD}, where both schemes have a theoretical guarantee.

\beforesubsec
\subsection{Image denoising, image deconvolution, and image inpainting}\label{subsec:image_deblurring}
\aftersubsec
In this subsection, we focus on solving $3$ fundamental problems in image processing: denoising, deconvolution, and inpainting.
These problems can be cast into the following composite convex model:
\begin{equation}\label{eq:image_dn}
F^{\star} := \min_{Y\in\R^{p\times q}} \Big\{ F(Y) := \kappa \Psi(\Kc(Y) - \cb) + \norms{Y}_{\mathrm{TV}} \Big\}.
\end{equation}
Here, $\Psi$ is a given data fidelity convex function, $\Kc$ is a linear operator, $\cb$ is a noisy, blurred, or missing image, $\kappa > 0$ is a regularization parameter, and $\norms{\cdot}_{\mathrm{TV}}$ is a total variation norm.
If $\Kc$ is identical, then we have a denoising problem.
If it is a deconvolution, then we obtain a deblurring problem.
If $\Kc$ is a mask operator, then we have an inpainting problem \cite{Chambolle2011}.

\subsubsection{Total-variation based image denosing}\label{subsubsec:denoising}
Let us consider the denoising problem with $\Kc = \Id$.
We choose three different functions $\Psi$ as follows: $\Psi(\cdot) := \frac{1}{2}\norms{\cdot}^2_2$,  $\Psi(\cdot) = \norms{\cdot}_1$, and $\Psi(\cdot) = \norms{\cdot}_2$.
We implement $2$ variants of Algorithm \ref{alg:A0} called \texttt{PADMM}, $2$ variants of its parallel version  \eqref{eq:admm_scheme1c} called \texttt{ParPD}, and Algorithm \ref{alg:A0b} called \texttt{scvx-PADMM}.
We compare these variants with standard ADMM \cite{Boyd2011}, Chambolle-Pock's methods \cite{Chambolle2011}, and their line search variants \cite{malitsky2016first}.
\begin{itemize}
\item For \texttt{PADMM} and \texttt{ParPD}, we follow exactly the update rules in our theoretical analysis.
Here, the operator $\Bc$ represents the TV-norm as $\norms{Y}_{\mathrm{TV}} = \norms{\Bc Y}_{2,1}$.
We choose $\rho_0 := \norms{\Bc}^2$ in the ROF and TV-$\ell_2$ models, and $\rho := \frac{1}{4}\norms{\Bc}^2$ in the TV-$\ell_1$-norm model. 
\item \texttt{PADMM-v2} and \texttt{ParPD-v2} are variants of \texttt{PADMM} and \texttt{ParPD}, respectively, where we update $\tau_k$ as $\tau_k := \frac{2}{k+2}$, i.e., with the same $\BigO{\frac{1}{k}}$-rate. 
We also choose  $\rho_0 := 0.3$, $0.2$, and $0.1$ in the ROF, TV-$\ell_1$, and TV-$\ell_2$ models, respectively.
\item In Chambolle-Pock's method, called \texttt{CP}, we choose $\tau = 0.01$ in the ROF model and $\tau = 0.02$ in the TV-$\ell_1$ and TV-$\ell_2$ models as suggested in \cite{Chambolle2011}.
\item In Chambolle-Pock's algorithm for strongly convex functions, called \texttt{scvx-CP}, we choose $\tau = \frac{1}{\norms{\Bc}}$ and $\mu_g = 0.7\kappa$ as in \cite{Chambolle2011}.
\item For the linesearch variants of Chambolle-Pock's algorithms, called \texttt{Ls-CP}, we set their parameters as $\tau_0 = 0.01$ for the ROF model or $\tau = 0.02$ for the others as in \texttt{CP}, $\beta = \sigma_0/\tau_0$, $\mu = 0.5$, and $\delta = 0.9$ as suggested in \cite{malitsky2016first}.
\item For \texttt{ADMM}, since we use the same trick as in \cite{Chan2011} to split the problem into three variables as in Subsection \ref{subsec:LAD} to avoid the expensive subproblem, we solve the underlying linear system with at most either $20$ PCG iterations or up to $10^{-5}$ accuracy.
We find that the penalty parameter $\rho = 10$ in ADMM works best (see also in \cite{Chambolle2011} as a suggestion).
\end{itemize}
We test these algorithms on $4$ images of size $512\times 512$: \texttt{barbara},  \texttt{boat}, \texttt{peppers}, and \texttt{cameraman} which are widely used in the literature.
We generate noisy images using the same procedure as in \cite{Chambolle2011} with Gaussian noise of variance $\hat{\sigma} = 0.1$ for the ROF and TV-$\ell_2$-norm models.
For the TV-$\ell_1$-norm model, we add ``salt and pepper'' noise with variance of $0.25$.
We also use the isotropic TV norm in all examples.

The results and performance of these algorithms are reported in Table \ref{tbl:im_denoising} after $300$ iterations, where PSNR is the Peak signal-to-noise ratio, and $F(\yb^k)$ is the objective value at the last iterate, and \texttt{Time} is time in second.

\begin{table}[hpt!]
\newcommand{\cell}[1]{{\!\!\!}#1{\!\!\!}}
\begin{footnotesize}
\begin{center}
\caption{The results of $10$ algorithms on 3 models of the image denoising problem after $300$ iterations}\label{tbl:im_denoising}
%\vspace{-3ex}
\begin{tabular}{l | rrr | rrr | rrr | rrr }\hline
\cell{Algorithm} & \cell{Time} &~ \cell{PSNR} & \cell{$F(y^k)$} &  \cell{Time} & ~\cell{PSNR} & \cell{$F(y^k)$} &  \cell{Time} & ~\cell{PSNR\!\!\!}&~~~~~\cell{$F(y^k)$} &  \cell{Time} &~ \cell{PSNR} &~~~~~\cell{$F(y^k)$}  \\ \hline
& \multicolumn{3}{c|}{\cell{\texttt{barbara} ($512\!\times\! 512$)}} & \multicolumn{3}{c|}{\cell{\texttt{boat}  ($512\!\times\! 512$)}} & \multicolumn{3}{c|}{\cell{\texttt{peppers}  ($512\!\times\! 512$)}} & \multicolumn{3}{c}{\cell{\texttt{cameraman}  ($512\!\times\! 512$)}} \\ \hline
\multicolumn{13}{c}{The ROF model ($\kappa = 16$)} \\ \hline
\cell{Noisy image} & \cell{-} & \cell{20.13} & \cell{51367.96} & \cell{-} & \cell{20.11} & \cell{47989.28} & \cell{-} & \cell{20.16} & \cell{46572.70} & \cell{-} & \cell{20.40} & \cell{44597.08} \\ \hline 
\cell{\texttt{PADMM}} & \cell{8.31} & \cell{25.76} & \cell{27049.90} & \cell{9.06} & \cell{27.99} & \cell{24023.41} & \cell{8.09} & \cell{29.52} & \cell{22250.56} & \cell{8.22} & \cell{29.59} & \cell{21143.43} \\ 
\cell{\texttt{ParPD}} & \cell{9.86} & \cell{25.79} & \cell{27050.15} & \cell{9.21} & \cell{27.97} & \cell{24024.61} & \cell{8.99} & \cell{29.50} & \cell{22251.93} & \cell{9.77} & \cell{29.56} & \cell{21144.17} \\ 
\cell{\texttt{PADMM-v2}} & \cell{8.37} & \cell{25.77} & \cell{27042.83} & \cell{8.89} & \cell{27.95} & \cell{24017.06} & \cell{8.38} & \cell{29.44} & \cell{22244.78} & \cell{8.90} & \cell{29.51} & \cell{21137.13} \\ 
\cell{\texttt{ParPD-v2}} & \cell{8.60} & \cell{25.77} & \cell{27042.91} & \cell{8.76} & \cell{27.95} & \cell{24017.14} & \cell{8.35} & \cell{29.44} & \cell{22244.85} & \cell{8.53} & \cell{29.51} & \cell{21137.20} \\ 
\cell{\texttt{scvx-PADMM}} & \cell{10.29} & \cell{25.77} & \cell{27042.94} & \cell{9.45} & \cell{27.95} & \cell{24017.19} & \cell{9.39} & \cell{29.43} & \cell{22244.91} & \cell{9.59} & \cell{29.50} & \cell{21137.28} \\ 
\cell{\texttt{CP}} & \cell{11.09} & \cell{25.78} & \cell{27042.72} & \cell{10.90} & \cell{27.95} & \cell{24016.96} & \cell{10.61} & \cell{29.44} & \cell{22244.70} & \cell{10.58} & \cell{29.51} & \cell{21137.19} \\ 
\cell{\texttt{scvx-CP}} & \cell{11.69} & \cell{25.77} & \cell{27042.58} & \cell{11.17} & \cell{27.95} & \cell{24016.80} & \cell{11.33} & \cell{29.44} & \cell{22244.49} & \cell{11.17} & \cell{29.51} & \cell{21136.79} \\ 
\cell{\texttt{Ls-CP}} & \cell{15.35} & \cell{25.78} & \cell{27042.68} & \cell{13.73} & \cell{27.95} & \cell{24016.90} & \cell{14.59} & \cell{29.44} & \cell{22244.63} & \cell{14.58} & \cell{29.51} & \cell{21137.06} \\ 
\cell{\texttt{scvx-Ls-CP}} & \cell{18.64} & \cell{25.78} & \cell{27042.58} & \cell{16.48} & \cell{27.95} & \cell{24016.79} & \cell{15.45} & \cell{29.44} & \cell{22244.49} & \cell{17.02} & \cell{29.51} & \cell{21136.79} \\ 
\cell{\texttt{ADMM}} & \cell{31.34} & \cell{25.78} & \cell{27042.78} & \cell{36.04} & \cell{27.95} & \cell{24017.03} & \cell{28.24} & \cell{29.44} & \cell{22244.79} & \cell{31.89} & \cell{29.51} & \cell{21137.34} \\ 
\hline
\multicolumn{13}{c}{The TV-$\ell_1$-norm model $(\kappa = 1.5)$} \\ \hline
\cell{Noisy image} & \cell{-} & \cell{11.31} & \cell{103961.55} & \cell{-} & \cell{11.51} & \cell{100899.36} & \cell{-} & \cell{11.32} & \cell{99938.89} & \cell{-} & \cell{11.10} & \cell{98975.10} \\ \hline 
\cell{\texttt{PADMM}} & \cell{10.97} & \cell{24.86} & \cell{61046.10} & \cell{8.72} & \cell{28.79} & \cell{58057.21} & \cell{10.56} & \cell{30.86} & \cell{55839.51} & \cell{8.56} & \cell{31.85} & \cell{54867.38} \\ 
\cell{\texttt{ParPD}} & \cell{10.00} & \cell{24.86} & \cell{61052.20} & \cell{11.01} & \cell{28.76} & \cell{58062.53} & \cell{10.40} & \cell{30.84} & \cell{55844.72} & \cell{9.16} & \cell{31.78} & \cell{54875.35} \\ 
\cell{\texttt{PADMM-v2}} & \cell{8.55} & \cell{24.85} & \cell{61037.30} & \cell{9.23} & \cell{28.77} & \cell{58048.58} & \cell{8.92} & \cell{30.85} & \cell{55831.60} & \cell{9.28} & \cell{31.76} & \cell{54858.82} \\ 
\cell{\texttt{ParPD-v2}} & \cell{8.77} & \cell{24.85} & \cell{61037.38} & \cell{9.81} & \cell{28.79} & \cell{58048.61} & \cell{9.56} & \cell{30.91} & \cell{55831.79} & \cell{8.85} & \cell{31.81} & \cell{54858.76} \\ 
\cell{\texttt{scvx-PADMM}} & \cell{11.01} & \cell{24.86} & \cell{61036.52} & \cell{9.93} & \cell{28.81} & \cell{58047.66} & \cell{10.80} & \cell{30.89} & \cell{55830.99} & \cell{10.39} & \cell{31.87} & \cell{54857.20} \\ 
\cell{\texttt{CP}} & \cell{14.50} & \cell{24.85} & \cell{61032.30} & \cell{10.45} & \cell{28.78} & \cell{58043.44} & \cell{10.71} & \cell{30.85} & \cell{55827.85} & \cell{12.64} & \cell{31.87} & \cell{54852.63} \\ 
\cell{\texttt{Ls-CP}} & \cell{15.88} & \cell{24.85} & \cell{61032.05} & \cell{13.28} & \cell{28.78} & \cell{58043.15} & \cell{14.59} & \cell{30.86} & \cell{55827.49} & \cell{16.30} & \cell{31.88} & \cell{54852.19} \\ 
\cell{\texttt{ADMM}} & \cell{47.18} & \cell{24.85} & \cell{61032.26} & \cell{49.13} & \cell{28.81} & \cell{58043.16} & \cell{48.23} & \cell{30.88} & \cell{55828.13} & \cell{51.12} & \cell{31.91} & \cell{54853.10} \\ 
\hline
\multicolumn{13}{c}{The TV-$\ell_2$-norm model $(\kappa = 280)$} \\ \hline
\cell{Noisy image} & \cell{-} & \cell{20.12} & \cell{51370.15} & \cell{-} & \cell{20.09} & \cell{48207.70} & \cell{-} & \cell{20.15} & \cell{46685.68} & \cell{-} & \cell{20.39} & \cell{44631.92} \\ \hline 
\cell{\texttt{PADMM}} & \cell{8.17} & \cell{23.04} & \cell{19758.70} & \cell{9.77} & \cell{25.13} & \cell{18459.69} & \cell{10.33} & \cell{28.34} & \cell{17684.45} & \cell{8.73} & \cell{27.63} & \cell{16883.26} \\ 
\cell{\texttt{ParPD}} & \cell{9.05} & \cell{23.02} & \cell{19771.01} & \cell{10.16} & \cell{25.06} & \cell{18466.01} & \cell{10.70} & \cell{28.24} & \cell{17688.90} & \cell{9.85} & \cell{27.54} & \cell{16886.00} \\ 
\cell{\texttt{PADMM-v2}} & \cell{7.99} & \cell{23.14} & \cell{19748.06} & \cell{10.35} & \cell{25.30} & \cell{18450.08} & \cell{9.90} & \cell{28.40} & \cell{17678.10} & \cell{8.86} & \cell{27.69} & \cell{16876.89} \\ 
\cell{\texttt{ParPD-v2}} & \cell{8.35} & \cell{23.14} & \cell{19748.16} & \cell{10.25} & \cell{25.31} & \cell{18450.16} & \cell{9.86} & \cell{28.40} & \cell{17678.14} & \cell{9.19} & \cell{27.69} & \cell{16876.93} \\ 
\cell{\texttt{scvx-PADMM}} & \cell{9.98} & \cell{23.14} & \cell{19749.30} & \cell{10.32} & \cell{25.30} & \cell{18451.63} & \cell{11.90} & \cell{28.43} & \cell{17679.44} & \cell{10.07} & \cell{27.71} & \cell{16878.87} \\ 
\cell{\texttt{CP}} & \cell{11.05} & \cell{23.12} & \cell{19748.78} & \cell{10.25} & \cell{25.27} & \cell{18451.29} & \cell{11.66} & \cell{28.39} & \cell{17679.29} & \cell{11.18} & \cell{27.68} & \cell{16880.50} \\ 
\cell{\texttt{Ls-CP}} & \cell{15.17} & \cell{23.12} & \cell{19748.05} & \cell{13.54} & \cell{25.27} & \cell{18450.42} & \cell{16.96} & \cell{28.39} & \cell{17678.70} & \cell{14.83} & \cell{27.68} & \cell{16878.77} \\ 
\cell{\texttt{ADMM}} & \cell{38.27} & \cell{23.12} & \cell{19747.10} & \cell{39.58} & \cell{25.27} & \cell{18449.24} & \cell{37.93} & \cell{28.39} & \cell{17677.60} & \cell{42.54} & \cell{27.67} & \cell{16877.02} \\ 
\hline
\end{tabular}
\end{center}
\end{footnotesize}
\vspace{-3ex}
\end{table}

From Table \ref{tbl:im_denoising}, we make the following observation.
\begin{itemize}
\item \texttt{PADMM}, \texttt{PADMM-v2}, \texttt{ParPD}, and \texttt{ParPD-v2} have similar performance as  \texttt{CP} in terms of PSNR and computational time.
\item \texttt{PADMM} and  \texttt{ParPD} slightly give a worse objective value than all the other methods.
\item \texttt{scvx-CP} and \texttt{scvx-Ls-CP} give the best objective values, but  \texttt{PADMM-v2}, \texttt{ParPD-v2}, and \texttt{scvx-PADMM} work  well and are comparable in terms of objective values.
\item \texttt{Ls-CP} and \texttt{scvx-Ls-CP} are slower than their non-linesearch versions due to additional computation. 
They also require more parameters to be selected, while only slightly improving the results.
\item \texttt{ADMM} is the slowest method due to an expensive subproblem that is solved by PCG.
\item Note that \texttt{ParPD} can be implemented in parallel, but we do not exploit it here.
\item When the problem is strongly convex, \texttt{scvx-CP} works well compared to \texttt{scvx-PADMM}.
However, \texttt{scvx-PADMM} still works with non-strongly convex problems as we see in the TV-$\ell_1$-norm or TV-$\ell_2$-norm models.
\end{itemize}
Figure \ref{fig:cameraman_house_tvl1} shows the convergence of $10$ algorithmic variants on  the TV-$\ell_1$-norm  model  and the TV-$\ell_2$-norm model for the \texttt{peppers} and \texttt{cameraman} images, respectively.
\begin{figure}[htp!]
\begin{center}
\includegraphics[width=1\linewidth]{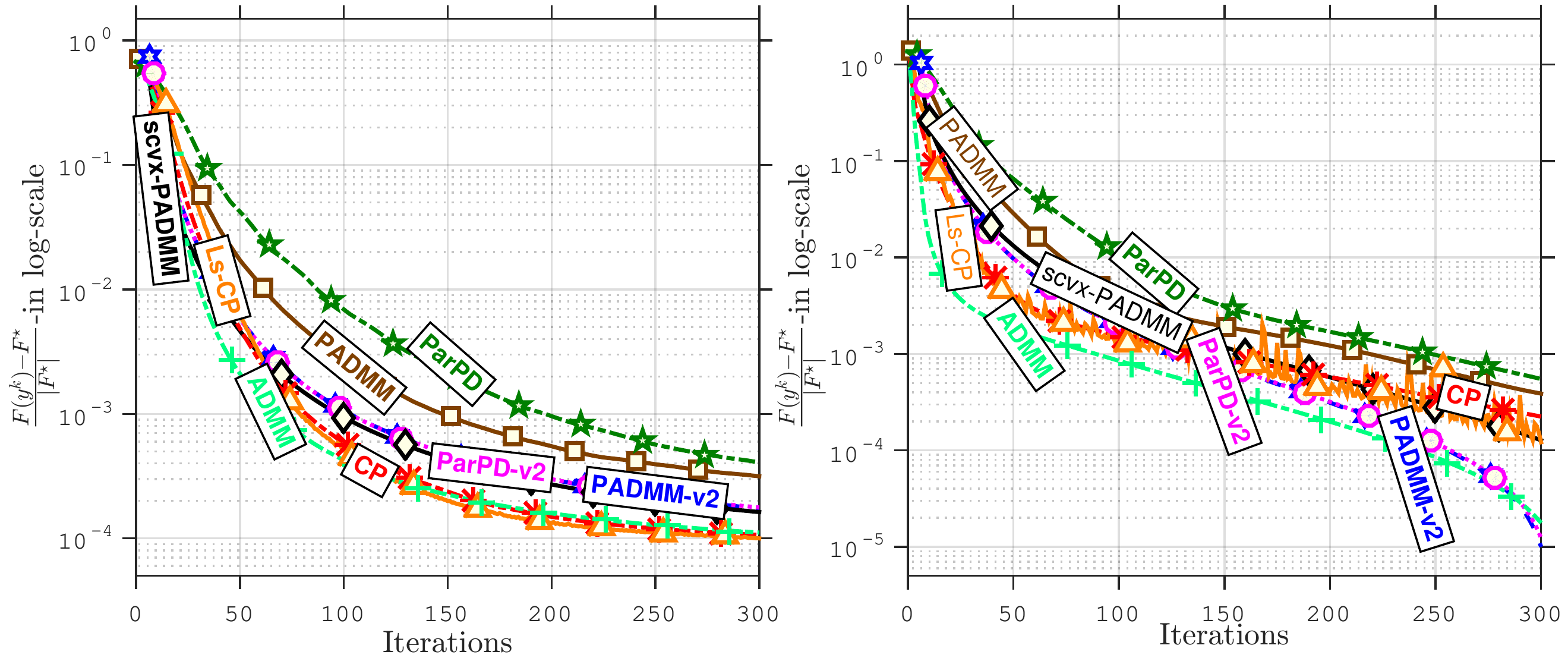} 
\vspace{-4ex}
\caption{\footnotesize Convergence behavior of $8$ algorithmic variants on two instances after $300$ iterations. 
Left: \texttt{peppers} with the TV-$\ell_1$-norm model; Right: \texttt{cameraman} with TV-$\ell_2$-norm model. 
}\label{fig:cameraman_house_tvl1}
\end{center}
\end{figure}
The left-plot indicates that the \texttt{CP} method and its variants are slightly better than \texttt{PADMM} and its variants.
However, the right-plot shows an opposite case where \texttt{PADMM} and its variants improve over the \texttt{CP} method and its variants.
As an illustration, the original, noisy, and recovered images of \texttt{peppers} are plotted in Figure \ref{fig:denoising_image}.
The quality of recovered images in this figure is reflected through PSNR on the top of each plot.
\begin{figure}[htp!]
\begin{center}
\includegraphics[width=1\linewidth]{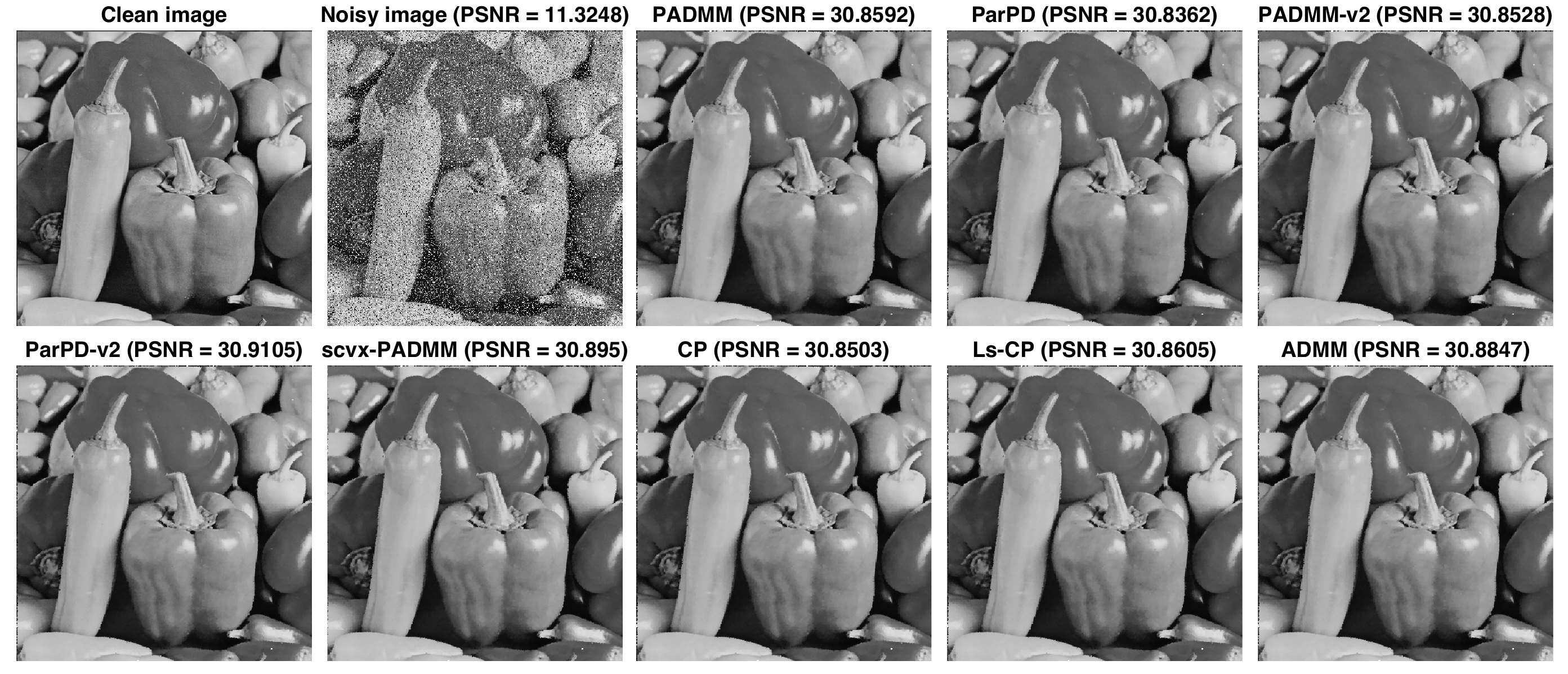}
\vspace{-4ex}
\caption{\footnotesize The denoised images of $8$ algorithmic variants on \texttt{peppers} using the TV-$\ell_1$-norm model.}\label{fig:denoising_image}
\end{center}
\end{figure}

\beforesubsubsec
\subsubsection{Image deconvolution with TV-norm}
\aftersubsubsec
We consider a well-studied deconvolution problem in image processing represented in the model \eqref{eq:image_dn}.
Here, $\Kc$ is a deconvolution operator with the point spread function (PSF) as studied in \cite{Chambolle2011}, and $\Psi(\cdot) := \frac{1}{2}\norms{\cdot}_F^2$. 
Since $\norms{Y}_{\mathrm{TV}} = \norms{\Bc Y}_{2,1}$, by introducing $x = \Bc Y$, we can reformulate \eqref{eq:image_dn} into \eqref{eq:constr_cvx}, where $f(x) = \norms{x}_{2,1}$, and $g(Y) = \frac{1}{2}\norms{\Kc(Y) - \cb}^2_F$.
Due to a special form of $\Kc$, $\prox_{\gamma g}$ can be computed in a closed form using FFTs \cite{Chambolle2011,Chan2011}.

We implement again $8$ algorithms in Subsection \ref{subsubsec:denoising} to solve this problem.
For  \texttt{PADMM} and \texttt{ParPD}, we set $\rho_0 := \frac{1}{2}\norms{\Bc}^2$, and for  \texttt{PADMM-v2} and \texttt{ParPD-v2}, we use $\rho_0 := 0.03$.
For \texttt{CP}, we set the step-size $\tau = 0.01$, and for its linesearch variant, we also use $\tau_0 := 0.01$, and update the other parameters as in Subsection \ref{subsubsec:denoising}.
For \texttt{ADMM}, we choose the penalty parameter $\rho = 10$.
For \texttt{scvx-PADMM}, we set $\mu_g :=  \frac{1}{2}\norms{\Bc}^2$, and for \texttt{scvx-CP}, we set $\tau_0 := \frac{1}{\norms{\Bc}}$.
We test these algorithms on $4$ images of size $512\times 512$: \texttt{barbara},  \texttt{boat}, \texttt{peppers}, and \texttt{cameraman} in Subsection \ref{subsubsec:denoising}.
We generate noisy images using the same procedure as in \cite{Chambolle2011} and set the regularization parameter $\kappa$ at $\kappa = 720$.
The results of this test is reported in Table \ref{tbl:im_deblurring}.

\begin{table}[hpt!]
\newcommand{\cell}[1]{{\!\!}#1{\!\!\!}}
\begin{footnotesize}
\begin{center}
\caption{The results of tested algorithms on 3 models of the image deconvolution problem after $300$ iterations}\label{tbl:im_deblurring}
%\vspace{-3ex}
\begin{tabular}{l | rrr | rrr | rrr | rrr }\hline
\cell{Algorithm} & \cell{Time} & \cell{PSNR} & \cell{$F(y^k)$} &  \cell{Time} & \cell{PSNR} & \cell{$F(y^k)$} &  \cell{Time} & \cell{PSNR} & \cell{$F(y^k)$} &  \cell{Time} & \cell{PSNR} & \cell{$F(y^k)$}  \\ \hline
& \multicolumn{3}{c|}{\cell{\texttt{barbara} ($512\!\times\! 512$)}} & \multicolumn{3}{c|}{\cell{\texttt{boat}  ($512\!\times\! 512$)}} & \multicolumn{3}{c|}{\cell{\texttt{peppers}  ($512\!\times\! 512$)}} & \multicolumn{3}{c}{\cell{\texttt{cameraman}  ($512\!\times\! 512$)}} \\ \hline
\cell{Blurred image} & \cell{-} & \cell{20.06} & \cell{110839.33} & \cell{-} & \cell{20.93} & \cell{72708.88} & \cell{-} & \cell{21.42} & \cell{99051.61} & \cell{-} & \cell{20.56} & \cell{95063.34} \\ \hline 
\cell{PADMM} & \cell{17.37} & \cell{24.10} & \cell{15130.31} & \cell{19.93} & \cell{26.14} & \cell{14389.55} & \cell{18.00} & \cell{29.09} & \cell{13750.55} & \cell{16.87} & \cell{29.14} & \cell{13418.93} \\ 
\cell{ParPD} & \cell{21.46} & \cell{24.10} & \cell{15130.13} & \cell{17.29} & \cell{26.14} & \cell{14389.83} & \cell{19.34} & \cell{29.09} & \cell{13750.65} & \cell{18.14} & \cell{29.14} & \cell{13419.01} \\ 
\cell{PADMM-v2} & \cell{18.80} & \cell{24.10} & \cell{15127.37} & \cell{16.42} & \cell{26.16} & \cell{14386.30} & \cell{18.48} & \cell{29.11} & \cell{13747.44} & \cell{16.57} & \cell{29.18} & \cell{13414.25} \\ 
\cell{ParPD-v2} & \cell{17.82} & \cell{24.11} & \cell{15127.46} & \cell{16.21} & \cell{26.17} & \cell{14386.44} & \cell{18.18} & \cell{29.11} & \cell{13747.63} & \cell{16.87} & \cell{29.19} & \cell{13414.50} \\ 
\cell{scvx-PADMM} & \cell{18.12} & \cell{24.11} & \cell{15127.45} & \cell{18.56} & \cell{26.16} & \cell{14386.05} & \cell{19.21} & \cell{29.11} & \cell{13747.31} & \cell{21.07} & \cell{29.18} & \cell{13414.24} \\ 
\cell{CP} & \cell{21.18} & \cell{24.10} & \cell{15123.88} & \cell{21.91} & \cell{26.15} & \cell{14382.23} & \cell{25.53} & \cell{29.10} & \cell{13744.64} & \cell{26.15} & \cell{29.17} & \cell{13409.89} \\ 
\cell{scvx-CP} & \cell{24.20} & \cell{24.10} & \cell{15123.56} & \cell{20.62} & \cell{26.15} & \cell{14381.78} & \cell{22.41} & \cell{29.10} & \cell{13744.17} & \cell{27.19} & \cell{29.17} & \cell{13408.84} \\ 
\cell{Ls-CP} & \cell{25.42} & \cell{24.10} & \cell{15123.81} & \cell{23.55} & \cell{26.15} & \cell{14382.09} & \cell{24.62} & \cell{29.09} & \cell{13744.42} & \cell{26.94} & \cell{29.17} & \cell{13410.00} \\ 
\cell{scvx-Ls-CP} & \cell{28.12} & \cell{24.10} & \cell{15123.86} & \cell{25.27} & \cell{26.15} & \cell{14382.50} & \cell{27.91} & \cell{29.09} & \cell{13744.75} & \cell{27.84} & \cell{29.17} & \cell{13410.01} \\ 
\cell{ADMM} & \cell{54.03} & \cell{24.10} & \cell{15124.11} & \cell{51.47} & \cell{26.15} & \cell{14382.48} & \cell{47.93} & \cell{29.09} & \cell{13744.82} & \cell{58.10} & \cell{29.17} & \cell{13410.24} \\ 
\hline
\end{tabular}
\end{center}
\end{footnotesize}
\vspace{-2ex}
\end{table}

In this test, our algorithms give the same PSNR and computational time as \texttt{CP} and \texttt{scvx-CP}, but has a slight worse objective value than these \texttt{CP} methods.
\texttt{CP} and \texttt{scvx-CP} and their linesearch variants are comparable with  \texttt{ADMM} in terms of objective values and PSNR, but \texttt{ADMM} is much slower.
Note that \texttt{CP} only has a theoretical guarantee on the averaging sequence. 
As seen from Subsection \ref{subsec:LAD}, this sequence gives a worse rate than the sequence of the last iterates as we use here.

\beforesubsubsec
\subsubsection{Image inpainting with TV-norm}
\aftersubsubsec
Our third example is a well-studied image inpainting problem, which is also covered by \eqref{eq:image_dn}.
Here, the linear operator $\Kc$ is simply a projection of the input image onto a subset of available pixels and $\Psi(\cdot) := \frac{1}{2}\norms{\cdot}_F^2$. 

We again implement  \texttt{PADMM}, \texttt{PADMM-v2}, \texttt{ParPD}, \texttt{ParPD-v2}, \texttt{scvx-PADMM}, \texttt{CP}, \texttt{Ls-CP}, and \texttt{ADMM} to solve this problem.
For  \texttt{PADMM} and \texttt{ParPD}, we set $\rho_0 := \frac{1}{2}\norms{\Bc}$, and for  \texttt{PADMM-v2} and \texttt{ParPD-v2}, we use $\rho_0 := 0.01$.
For \texttt{scvx-PADMM}, we set $\mu_g :=  \frac{1}{4}\norms{\Bc}^2$.
For \texttt{CP}, we set the step-size $\tau = 0.02$, and for its linesearch variant, we also use $\tau_0 := 0.02$, and update the other parameters as in Subsection \ref{subsubsec:denoising}.
For \texttt{ADMM}, we choose the penalty parameter $\rho = 10$.

We test these algorithms on the above $4$ images of size $512\times 512$.
We generate noisy images using the same procedure as in \cite{Chambolle2011} with $80\%$ missing pixels.
We set $\kappa = 32$.
The results of this test is reported in Table \ref{tbl:im_inpainting}.

\begin{table}[hpt!]
\newcommand{\cell}[1]{{\!\!}#1{\!\!\!}}
\begin{footnotesize}
\begin{center}
\caption{The results of $8$ algorithms on the image inpainting problem after $300$ iterations}\label{tbl:im_inpainting}
%\vspace{-3ex}
\begin{tabular}{l | rrr | rrr | rrr | rrr }\hline
\cell{Algorithm} & \cell{Time} & \cell{PSNR} & \cell{$F(y^k)$} &  \cell{Time} & \cell{PSNR} & \cell{$F(y^k)$} &  \cell{Time} & \cell{PSNR} & \cell{$F(y^k)$} &  \cell{Time} & \cell{PSNR} & \cell{$F(y^k)$}  \\ \hline
& \multicolumn{3}{c|}{\cell{\texttt{barbara} ($512\!\times\! 512$)}} & \multicolumn{3}{c|}{\cell{\texttt{boat}  ($512\!\times\! 512$)}} & \multicolumn{3}{c|}{\cell{\texttt{peppers}  ($512\!\times\! 512$)}} & \multicolumn{3}{c}{\cell{\texttt{cameraman}  ($512\!\times\! 512$)}} \\ \hline
\cell{Missing image} & \cell{-} & \cell{7.08} & \cell{46097.70} & \cell{-} & \cell{6.37} & \cell{41166.59} & \cell{-} & \cell{6.71} & \cell{41351.33} & \cell{-} & \cell{6.52} & \cell{34662.62} \\ \hline 
\cell{PADMM} & \cell{8.42} & \cell{21.98} & \cell{5715.34} & \cell{9.90} & \cell{22.23} & \cell{3746.34} & \cell{8.62} & \cell{22.83} & \cell{3381.66} & \cell{12.00} & \cell{22.69} & \cell{2774.65} \\ 
\cell{ParPD} & \cell{9.81} & \cell{21.99} & \cell{5715.15} & \cell{12.19} & \cell{22.24} & \cell{3746.13} & \cell{10.70} & \cell{22.84} & \cell{3381.54} & \cell{13.42} & \cell{22.70} & \cell{2774.52} \\ 
\cell{PADMM-v2} & \cell{8.31} & \cell{22.10} & \cell{5699.51} & \cell{9.88} & \cell{22.31} & \cell{3730.93} & \cell{11.14} & \cell{23.15} & \cell{3365.90} & \cell{9.75} & \cell{23.03} & \cell{2732.80} \\ 
\cell{ParPD-v2} & \cell{10.69} & \cell{22.10} & \cell{5699.73} & \cell{10.33} & \cell{22.32} & \cell{3731.21} & \cell{9.39} & \cell{23.16} & \cell{3366.22} & \cell{11.50} & \cell{23.04} & \cell{2733.12} \\ 
\cell{scvx-PADMM} & \cell{12.63} & \cell{22.12} & \cell{5694.81} & \cell{11.45} & \cell{22.30} & \cell{3726.33} & \cell{10.29} & \cell{23.18} & \cell{3362.59} & \cell{10.52} & \cell{23.06} & \cell{2728.61} \\ 
\cell{CP} & \cell{10.67} & \cell{21.90} & \cell{5687.80} & \cell{11.76} & \cell{22.08} & \cell{3714.63} & \cell{12.24} & \cell{22.56} & \cell{3355.18} & \cell{11.62} & \cell{22.52} & \cell{2720.37} \\ 
\cell{Ls-CP} & \cell{14.06} & \cell{21.92} & \cell{5684.18} & \cell{16.52} & \cell{22.07} & \cell{3711.73} & \cell{17.47} & \cell{22.59} & \cell{3351.31} & \cell{18.30} & \cell{22.49} & \cell{2713.49} \\ 
\cell{ADMM} & \cell{54.34} & \cell{21.93} & \cell{5678.73} & \cell{53.72} & \cell{22.05} & \cell{3706.29} & \cell{51.86} & \cell{22.74} & \cell{3343.72} & \cell{66.07} & \cell{22.17} & \cell{2701.83} \\ 
\hline
\end{tabular}
\end{center}
\end{footnotesize}
\vspace{-1ex}
\end{table}
From this table, we can see that:
\begin{itemize}
\item Our new algorithms are comparable with the state-of-the-art \texttt{CP}. \texttt{CP} gives a slightly better objective value, but \texttt{PADMM-v2} and \texttt{ParPD-v2} give better PSNR.
\item \texttt{ADMM} still works well and gives comparable PSNR, and better objective value than the others, but it is much slower.
\item \texttt{PADMM} and \texttt{ParPD} follow exactly our theory but still produce comparable results with the last iterate sequence of \texttt{CP} and \texttt{ADMM} in terms of PSNR.
\item In theory, \texttt{scvx-PADMM} is not applicable to solve this problem due to non-strong convexity, but by setting $\mu_g = \frac{1}{4}\norms{\Bc}^2$, it still performs well.
\end{itemize}
To illustrate the output of these algorithms, we show the recovered images on the \texttt{cameraman} image in Figure \ref{fig:im_inpainting}.
Clearly, with $80\%$ missing data, these algorithms are still able to recover good quality images.
\begin{figure}[htp!]
\begin{center}
\includegraphics[width=1\linewidth]{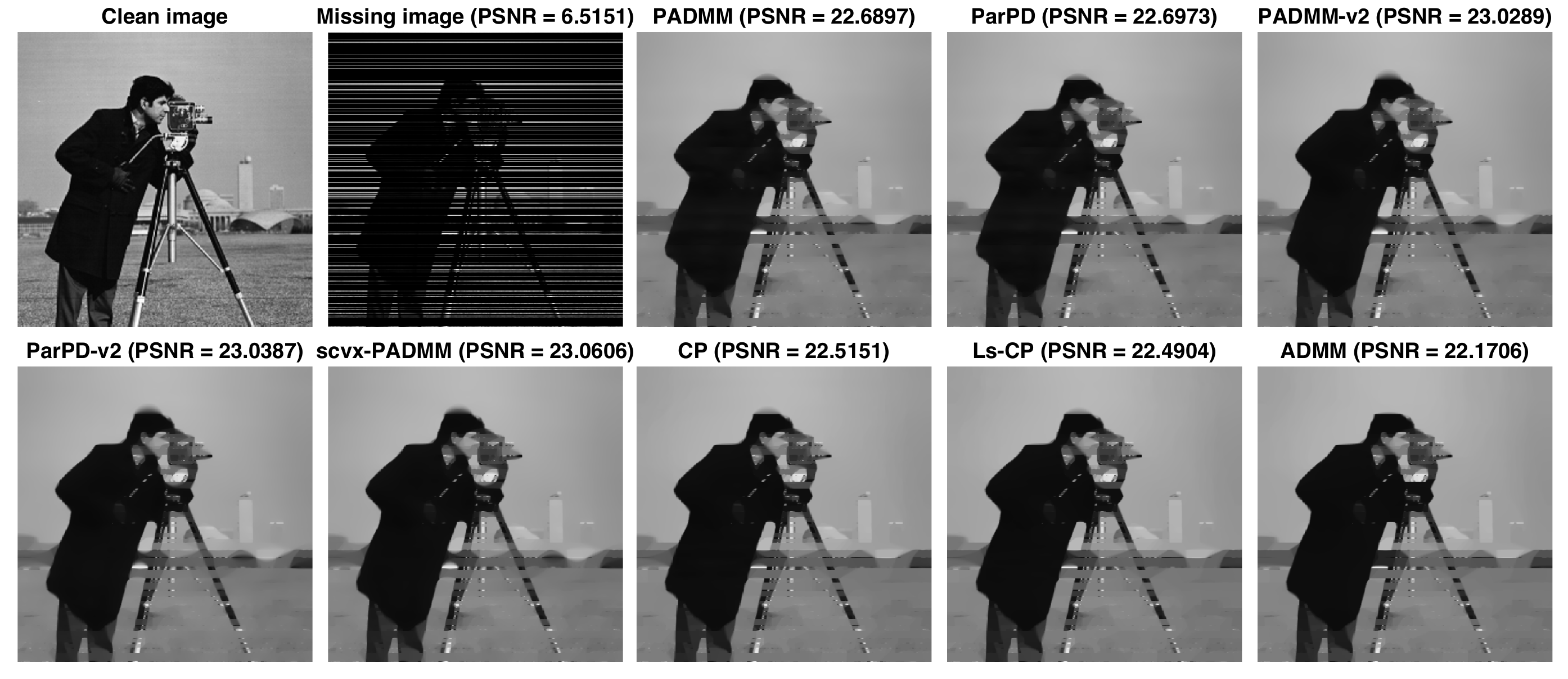}
\vspace{-4ex}
\caption{\footnotesize Recovered images from $80\%$ missing pixels of $8$ algorithmic variants on \texttt{Cameraman}.}\label{fig:im_inpainting}
\vspace{-1ex}
\end{center}
\end{figure}

%%%% Example 3.
\beforesubsec
\subsection{Image compression using compressive sensing}\label{subsec:im_basis_pursuit}
\aftersubsec
In this last example, we consider the following constrained convex optimization model in compressive sensing of images:
\begin{equation}\label{eq:basis_pursuit}
\min_{Y\in\R^{p_1\times p_2}}\Big\{f(Y) :=  \Vert\Dc Y\Vert_{2,1} ~\mid~ \Lc(Y) = \bb \Big\},
\end{equation}
where $\Dc$ is 2D discrete gradient operator representing a total variation (isotropic) norm, $\Lc:\R^{p_1\times p_2}\to\R^n$ is a linear operator obtained  from a subsampled  transformation scheme \cite{knoll2011adapted}, and $\bb\in\R^n$ is a compressive measurement vector \cite{baldassarre2016learning}.
Our goal is to recover a good image $Y$ from a small amount of measurement $\bb$ obtained via a model-based measurement operator $\Lc$.
To fit into our template \eqref{eq:constr_cvx}, we introduce $\xb = \Dc Y$ to obtain two linear constraints $\Lc(Y) = \bb$ and $-\xb + \Dc Y = 0$.
In this case, the constrained reformulation of \eqref{eq:basis_pursuit} becomes
\begin{equation*}
F^{\star} := \min_{x, Y}\Big\{ F(\zb) :=  \norms{x}_{2,1} \mid \xb - \Dc Y = 0, ~\Lc(Y) = \bb\Big\},
\end{equation*}
where $f(x) = \norms{x}_{2,1}$, and $g(Y) = 0$.

We now apply Algorithm~\ref{alg:A0}, its parallel variant \eqref{eq:admm_scheme1c}, and Algorithm~\ref{alg:A0b} to solve this problem and compare them with the \texttt{CP} method in \cite{Chambolle2011} and \texttt{ADMM} \cite{Boyd2011}. 
We also compare our methods with a line-search variant \texttt{Ls-CP} of \texttt{CP} recently proposed in \cite{malitsky2016first}. 

In \texttt{CP} and \texttt{Ls-CP}, we tune the step-size $\tau$ and find that $\tau = 0.01$ works well.
The other parameters of  \texttt{Ls-CP} are set as in the previous examples.
For \texttt{PADMM} and \texttt{ParPD}, we use $\rho_0 := 2\norms{\Bc}^2$, and for  \texttt{PADMM-v2} and \texttt{ParPD-v2}, we use  $\rho_0 := 10\norms{\Bc}^2$.
We also set $\mu_g := 2\norms{\Bc}^2$ in \texttt{scvx-PADMM}.
For the standard \texttt{ADMM} algorithm, we tune its penalty parameter and find that $\rho := 20$ works best.

We test all the algorithms on $4$ MRI images: \texttt{MRI-of-knee}, \texttt{MRI-brain-tumor}, \texttt{MRI-hands}, and \texttt{MRI-wrist}.\footnote{These images are from \href{https://radiopaedia.org/cases/4090/studies/6567}{https://radiopaedia.org/cases/4090/studies/6567} and \href{https://www.nibib.nih.gov/science-education/science-topics/magnetic-resonance-imaging-mri}{https://www.nibib.nih.gov}}
We follow the procedure in \cite{knoll2011adapted} to generate the samples using a sample rate of $25\%$. 
Then, the vector of measurements $\cb$ is computed from $\cb := \Lc(Y^{\natural})$, where $Y^{\natural}$ is the original image.

%We run all the algorithms with $500$ iterations.
\begin{table}[ht!]
\newcommand{\cell}[1]{{\!\!}#1{\!\!\!}}
\begin{center}
\begin{footnotesize}
\caption{Performance and results of $8$ algorithms on $4$ MRI images}\label{tbl:MRI_results}
\begin{tabular}{| c | rrrrr | rrrrr |}\hline
\cell{Algorithms} & \cell{$f(\bar{\Yb}^k)$} & \cell{$\frac{\Vert\Lc(\bar{\Yb}^k) -\bb\Vert}{\norm{\bb}}$} & \cell{\texttt{Error}} & \cell{PSNR} & \cell{Time[s]} & \cell{$f(\bar{\Yb}^k)$} & \cell{$\frac{\Vert\Lc(\bar{\Yb}^k) - \bb\Vert}{\norm{\bb}}$} & \cell{\texttt{Error}} & \cell{PSNR} & \cell{Time[s]} \\ \hline
& \multicolumn{5}{c|}{\texttt{MRI-knee ($779 \times 693$)}} &  \multicolumn{5}{c|}{\texttt{MRI-brain-tumor ($630\times 611$)}} \\ \hline
\cell{PADMM} & \cell{24.350} & \cell{2.637e-02} & \cell{4.672e-02} & \cell{83.93} & \cell{80.15} &\cell{36.101} & \cell{2.724e-02} & \cell{6.575e-02} & \cell{79.50} & \cell{53.77} \\
\cell{ParPD} & \cell{24.335} & \cell{2.539e-02} & \cell{4.676e-02} & \cell{83.93} & \cell{98.38} &\cell{36.028} & \cell{2.738e-02} & \cell{6.595e-02} & \cell{79.47} & \cell{52.71} \\
\cell{PADMM-v2} & \cell{28.862} & \cell{7.125e-05} & \cell{4.143e-02} & \cell{84.98} & \cell{73.56} &\cell{39.317} & \cell{5.226e-05} & \cell{6.310e-02} & \cell{79.85} & \cell{52.97} \\
\cell{ParPD-v2} & \cell{29.183} & \cell{7.247e-05} & \cell{4.007e-02} & \cell{85.27} & \cell{95.49} &\cell{39.594} & \cell{5.338e-05} & \cell{6.258e-02} & \cell{79.93} & \cell{51.64} \\
\cell{scvx-PADMM} & \cell{24.633} & \cell{2.295e-02} & \cell{4.424e-02} & \cell{84.41} & \cell{87.96} &\cell{36.783} & \cell{2.184e-02} & \cell{5.780e-02} & \cell{80.62} & \cell{65.12} \\
\cell{CP} & \cell{24.897} & \cell{2.674e-02} & \cell{4.629e-02} & \cell{84.01} & \cell{101.22} &\cell{37.745} & \cell{3.613e-02} & \cell{7.896e-02} & \cell{77.91} & \cell{63.71} \\
\cell{Ls-CP} & \cell{24.955} & \cell{2.638e-02} & \cell{4.659e-02} & \cell{83.96} & \cell{166.11} &\cell{38.139} & \cell{3.414e-02} & \cell{7.485e-02} & \cell{78.37} & \cell{103.12} \\
\cell{ADMM} & \cell{25.071} & \cell{2.556e-02} & \cell{4.654e-02} & \cell{83.97} & \cell{902.79} &\cell{38.941} & \cell{2.895e-02} & \cell{6.135e-02} & \cell{80.10} & \cell{655.81} \\
\hline
& \multicolumn{5}{c|}{\texttt{MRI-hands ($1024 \times 1024$)}} &  \multicolumn{5}{c|}{\texttt{MRI-wrist ($1024\times 1024$)}} \\ \hline
\cell{PADMM} & \cell{45.207} & \cell{2.081e-02} & \cell{2.765e-02} & \cell{91.37} & \cell{146.41} &\cell{29.459} & \cell{1.802e-02} & \cell{3.224e-02} & \cell{90.04} & \cell{152.51} \\
\cell{ParPD} & \cell{45.207} & \cell{2.081e-02} & \cell{2.765e-02} & \cell{91.37} & \cell{140.41} &\cell{29.459} & \cell{1.802e-02} & \cell{3.224e-02} & \cell{90.04} & \cell{148.12} \\
\cell{PADMM-v2} & \cell{48.679} & \cell{7.336e-05} & \cell{2.074e-02} & \cell{93.87} & \cell{138.65} &\cell{30.578} & \cell{8.516e-05} & \cell{2.572e-02} & \cell{92.00} & \cell{146.05} \\
\cell{ParPD-v2} & \cell{48.858} & \cell{7.483e-05} & \cell{2.008e-02} & \cell{94.15} & \cell{148.79} &\cell{30.768} & \cell{8.766e-05} & \cell{2.473e-02} & \cell{92.34} & \cell{146.64} \\
\cell{scvx-PADMM} & \cell{45.426} & \cell{1.820e-02} & \cell{2.588e-02} & \cell{91.95} & \cell{154.35} &\cell{29.403} & \cell{1.647e-02} & \cell{3.131e-02} & \cell{90.29} & \cell{157.35} \\
\cell{CP} & \cell{45.723} & \cell{2.489e-02} & \cell{3.895e-02} & \cell{88.40} & \cell{159.74} &\cell{30.052} & \cell{2.032e-02} & \cell{3.661e-02} & \cell{88.93} & \cell{165.58} \\
\cell{Ls-CP} & \cell{53.640} & \cell{2.724e-02} & \cell{3.924e-02} & \cell{88.33} & \cell{254.94} &\cell{39.396} & \cell{2.353e-02} & \cell{3.856e-02} & \cell{88.48} & \cell{284.29} \\
\cell{ADMM} & \cell{45.985} & \cell{2.034e-02} & \cell{3.443e-02} & \cell{89.47} & \cell{1691.53} &\cell{29.922} & \cell{1.825e-02} & \cell{3.686e-02} & \cell{88.88} & \cell{1503.56} \\
\hline
\end{tabular}
\end{footnotesize}
\end{center}
\end{table}
The performance and results of these algorithms are summarized in Table~\ref{tbl:MRI_results}, where $f(\bar{Y}^k) := \norms{\Dc \bar{Y}^k}_{2,1}$ is the objective value, $\texttt{Error} := \frac{\Vert \bar{\Yb}^k - \Yb^{\natural}\Vert_F}{\Vert\Yb^{\natural}\Vert_F}$ presents the relative  error between the original image $\Yb^{\natural}$ to the reconstruction $\bar{\Yb}^k$ after $k = 300$ iterations.

We observe the following facts from the results of Table ~\ref{tbl:MRI_results}.
\begin{itemize}
\item \texttt{PADMM}, \texttt{ParPD}, and \texttt{scvx-PADMM} are comparable with \texttt{CP} in terms of computational time, PSNR,  objective values, and solution errors.
\item \texttt{PADMM-v2} and \texttt{ParPD-v2} give better PSNR and solution errors, but have slightly worse objective value than the others.
\item \texttt{Ls-CP} is slower than our methods due to additional computation.
\item \texttt{ADMM} gives similar result in terms of the objective values, solution errors, and PSNR, but it is much slower than other methods.
\end{itemize}
The   reconstructed images of \texttt{MRI-wrist} are  revealed in Figure~\ref{fig:MRI_image}. 
As seen from this plot, the quality of recovery image is very close to the original image for the sampling rate of $25\%$.
\begin{figure}[ht!]
\begin{center}
\includegraphics[width=1\linewidth]{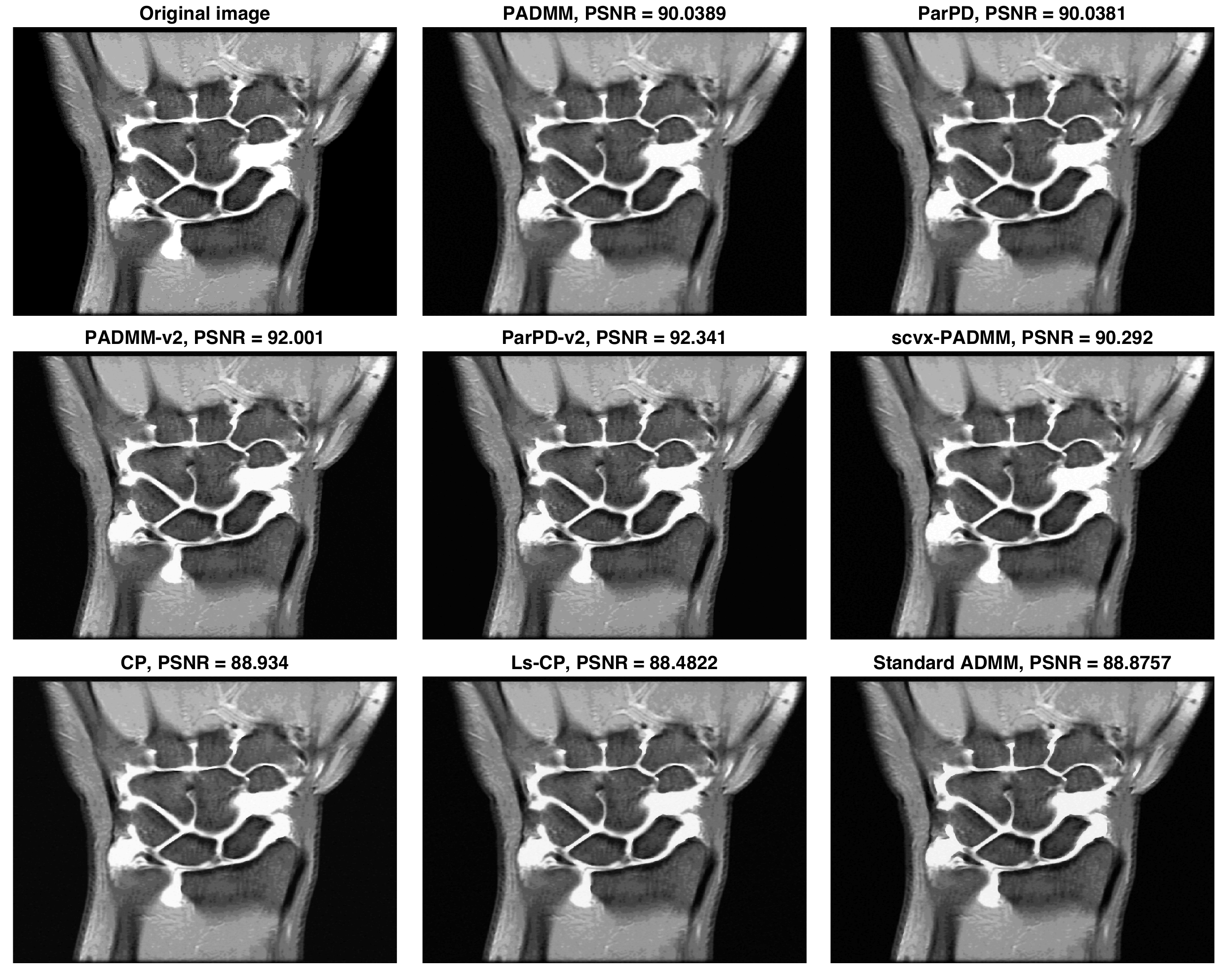}%MRI_of_Knee_cs}
\vspace{-4ex}
\caption{The original image and the reconstructed images of $8$ algorithms.}\label{fig:MRI_image}
\vspace{-1ex}
\end{center}
\end{figure}

%%%% Acknowledgments.
\vspace{0.5ex}
\begin{footnotesize}
\noindent \textbf{Acknowledgments:} 
This work is partly supported by the NSF-grant, DMS-1619884, USA.
\end{footnotesize}
%%% End of Acknowledgments.

%%%%%%%%%%%%%%%%%%%%%%%%%%%%%%%%%%%%%%%%%%%%%%%%%%%%%%%%%%%%%
%%%% Appendix - Technical proof.
%%%%%%%%%%%%%%%%%%%%%%%%%%%%%%%%%%%%%%%%%%%%%%%%%%%%%%%%%%%%%
\appendix
\normalsize
\beforesec
\section{Auxiliary lemmas}\label{sec:appendix}
\aftersec
We first provide the following two auxiliary lemmas which will be used in our convergence analysis.
%%%%%%%
%% Lemma 1.
\begin{lemma}\label{lm:useful_lemma}
$\mathrm{(a)}$~Given $\ub, \vb\in\R^n$, and $\tau\in [0, 1]$, we have
\begin{equation}\label{eq:basic_pro1}
\begin{array}{ll}
\Tc_{\tau}(\ub,\vb) &:= \frac{(1-\tau)}{2}\norms{\ub - \vb}^2 + \frac{\tau}{2}\norms{\ub}^2 - \frac{\tau(1-\tau)}{2}\norms{\vb}^2 = \frac{1}{2}\norms{\ub - (1-\tau)\vb}^2.
\end{array}
\end{equation}
$\mathrm{(b)}$~For any $\ub, \vb, \rb \in\R^p$, we have the following Pythagoras identity
\begin{equation}\label{eq:3points_equality}
2\iprods{\ub - \vb, \vb - \rb} = \norms{\ub - \rb}^2 - \norms{\vb - \rb}^2 - \norms{\ub - \vb}^2.
\end{equation}
$\mathrm{(c)}$~Let $\Lc_{\rho}$ be defined by \eqref{eq:auLag_func}.
Then, for any $\lbd, \hat{\lbd}\in\R^n$ and $\zb\in\dom{F}$, we have
\begin{equation}\label{eq:basic_pro3}
\begin{array}{ll}
\Lc_{\rho}(\zb, \hat{\lbd}) &= \Lc_{\rho}(\zb, \lbd) +  \iprods{\lbd - \hat{\lbd}, \Ab\xb + \Bb\yb - \cb}, \vspace{1ex}\\
\Lc_{\rho_k}(\zb, \lbd) &= \Lc_{\rho_{k-1}}(\zb, \lbd) + \frac{(\rho_k - \rho_{k-1}}{2}\norms{\Ab\xb + \Bb\xb - \cb}^2.
\end{array}
\end{equation}
\end{lemma}
%%%%
One key component in our analysis is the following function:
\begin{equation}\label{eq:phi_func}
\phi_{\rho}(\zb, \lbd) := \frac{\rho}{2}\norms{\Ab\xb + \Bb\yb - \cb}^2 - \iprods{\lbd, \Ab\xb + \Bb\yb - \cb}.
\end{equation}
Given $\hat{\zb}^k := (\hat{\xb}^k, \hat{\yb}^k)$, $\hat{\bar{\zb}}^{k+1} := (\bar{\xb}^{k+1}, \hat{\yb}^k) \in \R^p$ and $\hat{\lbd}^k\in\R^n$, we define two  linear functions:
\begin{equation}\label{eq:linear_func}
\begin{array}{ll}
\hat{\ell}^k_{\rho}(\zb) &:= \phi_{\rho}(\hat{\zb}^k, \hat{\lbd}^k) + \iprods{\nabla_{\xb}{\phi_{\rho}}(\hat{\zb} _k, \hat{\lbd}^k), \xb - \hat{\xb}^k} + \iprods{\nabla_{\yb}{\phi_{\rho}}(\hat{\zb} _k, \hat{\lbd}^k), \yb - \hat{\yb}^k}, \vspace{1ex}\\
\hat{\bar{\ell}}^k_{\rho}(\zb) &:= \phi_{\rho}(\hat{\bar{\zb}}^{k+1}, \hat{\lbd}^k) + \iprods{\nabla_{\xb}{\phi_{\rho}}(\hat{\bar{\zb}}^{k+1}, \hat{\lbd}^k), \xb - \bar{\xb}^{k+1}} + \iprods{\nabla_{\yb}{\phi_{\rho}}(\hat{\bar{\zb}}^{k+1}, \hat{\lbd}^k), \yb - \hat{\yb}^k}, \end{array}
\end{equation}
and two quadratic functions
\begin{equation}\label{eq:Qk_func}
\begin{array}{ll}
\hat{\bar{\Qc}}_{\rho_k}^k(\yb) &:= \phi_{\rho_k}(\hat{\bar{\zb}}^{k+1}, \hat{\lbd}^k) + \iprods{\nabla_y\phi_{\rho_k}(\hat{\bar{\zb}}^{k+1},\hat{\lbd}^k), \yb - \hat{\yb}^k} + \tfrac{\rho_k\norms{\Bb}^2}{2}\norms{\yb - \hat{\yb}^k}^2, \vspace{1ex}\\
\hat{\Qc}_{\rho_k}^k(\zb) &:= \phi_{\rho_k}(\hat{\zb}^k, \hat{\lbd}^k) + \iprods{\nabla_x\phi_{\rho_k}(\hat{\zb}^k,\hat{\lbd}^k), \xb - \hat{\xb}^k} + \iprods{\nabla_y\phi_{\rho_k}(\hat{\zb}^k,\hat{\lbd}^k), \yb - \hat{\yb}^k} \vspace{1ex}\\
& + \tfrac{\rho_k\norms{\Ab}^2}{2}\norms{\xb - \hat{\xb}^k}^2 + \tfrac{\rho_k\norms{\Bb}^2}{2}\norms{\yb - \hat{\yb}^k}^2.
\end{array}
\end{equation}
For our convenience, we also define the following four vectors
\begin{equation}\label{eq:s_terms}
\begin{array}{llll}
\hat{\sb}^k &:= \Ab\hat{\xb}^k + \Bb\hat{\yb}^k  - \cb,  ~~~& \hat{\bar{\sb}}^{k+1} &:= \Ab\bar{\xb}^{k+1} + \Bb\hat{\yb}^k - \cb, \vspace{1ex}\\
\bar{\sb}^k &:= \Ab\bar{\xb}^k + \Bb\bar{\yb}^k - \cb,~~~ \text{and}~~~&\bar{\sb}^{k+1} &:= \Ab\bar{\xb}^{k+1} + \Bb\bar{\yb}^{k+1} - \cb.
\end{array}
\end{equation}
%Using the definition of $\hat{\ell}^k_{\rho}$, $\hat{\bar{\ell}}_{\rho}^k$, and $\phi_{\rho}$, we can easily show that
%\begin{equation}\label{eq:l_property1}
%\begin{array}{lll}
%\hat{\ell}^k_{\rho}(\zb) = \phi_{\rho}(\zb, \hat{\lbd}^k) - \frac{\rho}{2}\norms{\Ab(\xb - \hat{\xb}^k) + \Bb(\yb - \hat{\yb}^k)}^2,~~&\forall \zb := (\xb, \yb) \in\R^p,\vspace{1ex}\\
%\hat{\bar{\ell}}^k_{\rho}(\zb) = \phi_{\rho}(\zb, \hat{\lbd}^k) - \frac{\rho}{2}\norms{\Ab(\xb - \bar{\xb}^{k+1}) + \Bb(\yb - \hat{\yb}^k)}^2,~~&\forall \zb := (\xb, \yb) \in\R^p.
%\end{array}
%\end{equation}
%Let $\zb^{\star} = (\xb^{\star}, \yb^{\star}) \in\Zc^{\star}$ be an optimal solution of \eqref{eq:constr_cvx} satisfying $\Ab\xb^{\star} + \Bb\yb^{\star} = \cb$.
%Then, $\phi_{\rho}(\zb^{\star}, \lbd) = 0$ for any $\rho > 0$ and $\lbd\in\R^n$.
%If we define $\hat{\sb}^k := \Ab\hat{\xb}^k + \Bb\hat{\yb}^k  - \cb$, $\hat{\bar{\sb}}^{k+1} := \Ab\bar{\xb}^{k+1} + \Bb\hat{\yb}^k - \cb$, $\bar{\sb}^{k+1} := \Ab\bar{\xb}^{k+1} + \Bb\bar{\yb}^{k+1} - \cb$, and $\bar{\sb}^k := \Ab\bar{\xb}^k + \Bb\bar{\yb}^k - \cb$, then \eqref{eq:l_property1} implies 
%In addition, we will repeatedly use the following elementary lemma and we skip the proof.
%%
Then, we have the following lemma.

%% Lemma 2.
\begin{lemma}\label{lm:useful_lemma2}
$\mathrm{(a)}$ Let $\zb^{\star} = (\xb^{\star}, \yb^{\star})\in\R^p$ be such that $\Ab\xb^{\star} + \Bb\yb^{\star} = \cb$, $\hat{\ell}^k_{\rho}$ and $\hat{\bar{\ell}}_{\rho}^k$ be defined by \eqref{eq:linear_func}, and $s$-vectors be defined by \eqref{eq:s_terms}.
Then, we have
\begin{equation}\label{eq:l_property2}
\begin{array}{llll}
\hat{\ell}^k_{\rho}(\zb^{\star}) &= -\frac{\rho}{2}\norms{\hat{\sb}^k}^2 ~~~&\text{and}~~~~~~\hat{\ell}^k_{\rho}(\bar{\zb}^k) & = \phi_{\rho}(\bar{\zb}^k, \hat{\lbd}^k) -\frac{\rho}{2}\norms{\bar{\sb}^k - \hat{\sb}^k}^2, \vspace{1ex}\\
\hat{\bar{\ell}}^k_{\rho}(\zb^{\star}) &= -\frac{\rho}{2}\norms{\hat{\bar{\sb}}^{k+1}}^2 ~~~&\text{and}~~~~~~\hat{\bar{\ell}}^k_{\rho}(\bar{\zb}^k) & = \phi_{\rho}(\bar{\zb}^k, \hat{\lbd}^k) -\frac{\rho}{2}\norms{\bar{\sb}^k - \hat{\bar{\sb}}^{k+1}}^2.
\end{array}
\end{equation}
$\mathrm{(b)}$~Let   $\hat{\bar{\Qc}}_{\rho_k}^k$ and $\hat{\Qc}_{\rho_k}^k$ be defined by \eqref{eq:Qk_func}.
Then
\begin{equation}\label{eq:l_property2b}
\phi_{\rho_k}(\bar{\xb}^{k+1}, \yb, \hat{\lbd}^k) \leq \hat{\bar{\Qc}}_{\rho_k}^k(\yb)~~\text{for}~~\yb\in\R^{p_2}, ~~~~\text{and}~~~~\phi_{\rho_k}(\zb, \hat{\lbd}^k) \leq \hat{\Qc}_{\rho_k}^k(\zb)~~\text{for}~\zb\in\R^p.
\end{equation}
\end{lemma}

%%%% B. Convergence analysis of Algorithm 1.
\beforesec
\section{Convergence analysis of Algorithm \ref{alg:A0} and its parallel variant \eqref{eq:admm_scheme1c}}\label{sec:convergence_analysis_of_A0}
\aftersec

The proof of Theorem \ref{th:admm-convergence1} and Corollary \ref{co:convergence1c} relies on the following descent lemma.

%% Lemma 1.
\begin{lemma}\label{le:descent_lemma}
Assume that $\Lc_{\rho}$ is defined by \eqref{eq:auLag_func}, and $\hat{\ell}_{\rho}^k$ and $\hat{\bar{\ell}}^k_{\rho}$ are defined by \eqref{eq:linear_func}.
\begin{itemize}
\item[]$\mathrm{(a)}$~Let $\bar{\zb}^{k+1}$ be computed by Step \ref{step:admm_step} of Algorithm \ref{alg:A0}. Then, for any $\zb\in\dom{F}$, we have
\begin{equation}\label{eq:descent_pro2}
\begin{array}{ll}
\Lc_{\rho_k}(\bar{\zb}^{k+1},\hat{\lbd}^k) &\leq F(\zb) + \hat{\bar{\ell}}_{\rho_k}^k(\zb) + \gamma_k\iprods{\bar{\xb}^{k+1} - \hat{\xb}^k, \xb - \hat{\xb}^k} - \gamma_k\norms{\bar{\xb}^{k+1} - \hat{\xb}^k}^2 \vspace{1ex}\\
& +~ \beta_k\iprods{\bar{\yb}^{k+1} - \hat{\yb}^k, \yb - \hat{\yb}^k} - \tfrac{(2\beta_k - \rho_k\norms{\Bb}^2)}{2}\norms{\bar{\yb}^{k+1} -\hat{\yb}^k}^2.
\end{array}
\end{equation}
\item[]$\mathrm{(b)}$~Let $\bar{\zb}^{k+1}$ be computed by \eqref{eq:admm_scheme1c}. Then, for any $\zb\in\dom{F}$, we have
\begin{equation}\label{eq:descent_pro3}
{\!\!\!\!\!}\begin{array}{ll}
\Lc_{\rho_k}(\bar{\zb}^{k+1},\hat{\lbd}^k) {\!\!\!}&\leq F(\zb) + \hat{\ell}_{\rho_k}^k(\zb) +  \gamma_k\iprods{\bar{\xb}^{k+1} - \hat{\xb}^k, \xb - \hat{\xb}^k} - \tfrac{(2\gamma_k - \rho_k\norms{\Ab}^2)}{2}\norms{\bar{\xb}^{k+1} - \hat{\xb}^k}^2 \vspace{1ex}\\
& +~ \beta_k\iprods{\bar{\yb}^{k+1} - \hat{\yb}^k, \yb - \hat{\yb}^k} - \tfrac{(2\beta_k - \rho_k\norms{\Bb}^2)}{2}\norms{\bar{\yb}^{k+1} -\hat{\yb}^k}^2.
\end{array}{\!\!\!\!\!}
\end{equation}
\end{itemize}
\end{lemma}

%%% The proof of Lemma A.1.
\begin{proof}
(a)~ From Lemma \ref{lm:useful_lemma2}(b), we have
\begin{equation}\label{eq:lm_a1_est1}
{\!\!\!\!}\begin{array}{ll}
\phi_{\rho_k}(\bar{\zb}^{k+1}, \hat{\lbd}^k) \leq \phi_{\rho_k}(\hat{\bar{\zb}}^{k+1}, \hat{\lbd}^k) + \iprods{\nabla_{\yb}\phi_{\rho_k}(\hat{\bar{\zb}}^{k+1}, \hat{\lbd}^k), \bar{\yb}^{k+1} - \hat{\yb}^k} + \frac{\rho_k\norms{\Bb}^2}{2}\norms{ \bar{\yb}^{k+1} - \hat{\yb}^k}^2. 
\end{array}{\!\!\!\!}
\end{equation}
The optimality conditions of the $x$-subproblem and $\bar{\yb}^{k+1}$ at Step \ref{step:admm_step} of  Algorithm~\ref{alg:A0} are
\begin{equation}\label{eq:lm_a1_opt_cond1}
\left\{\begin{array}{lll}
0 &= \nabla{f}(\bar{\xb}^{k+1})  + \nabla_{\xb}{\phi_{\rho_k}}(\hat{\bar{\zb}}^{k+1}, \hat{\lbd}^k) + \gamma_k(\bar{\xb}^{k+1} - \hat{\xb}^k), ~&\nabla{f}(\bar{\xb}^{k+1})\in\partial{f}(\bar{\xb}^{k+1}), \vspace{1ex}\\
0 &= \nabla{g}(\bar{\yb}^{k+1}) + \nabla_{\yb}{\phi_{\rho_k}}(\hat{\bar{\zb}}^{k+1}, \hat{\lbd}^k) + \beta_k(\bar{\yb}^{k+1} - \hat{\yb}^k), ~&\nabla{g}(\bar{\yb}^{k+1})\in\partial{g}(\bar{\yb}^{k+1}).
\end{array}\right.
\end{equation}
Using the convexity of $f$ and $g$, for any $\xb\in\dom{f}$, $\yb \in\dom{g}$, we have
\begin{equation}\label{eq:lm_a1_est2}
\begin{array}{lll}
f(\bar{\xb}^{k+1})  &\leq f(\xb) + \iprods{\nabla{f}(\bar{\xb}^{k+1}), \bar{\xb}^{k+1} - \xb}, &\nabla{f}(\bar{\xb}^{k+1}) \in\partial{f}(\bar{\xb}^{k+1}), \vspace{1ex}\\
g(\bar{\yb}^{k+1}) &\leq g(\yb) + \iprods{\nabla{g}(\bar{\yb}^{k+1}), \bar{\yb}^{k+1} - \yb}, &\nabla{g}(\bar{\yb}^{k+1}) \in\partial{g}(\bar{\yb}^{k+1}).
\end{array}
\end{equation}
Combining \eqref{eq:lm_a1_est1}, \eqref{eq:lm_a1_opt_cond1}, and \eqref{eq:lm_a1_est2}, and then using the definition \eqref{eq:auLag_func} of $\Lc_{\rho}$, for any $\zb = (\xb, \yb) \in\dom{F}$, we can derive that
\begin{equation*} 
{\!\!\!\!}\begin{array}{ll}
\Lc_{\rho_k}(\bar{\zb}^{k+1}, \hat{\lbd}^k) {\!\!\!\!}&= f(\bar{\xb}^{k+1}) + g(\bar{\yb}^{k+1}) + \phi_{\rho_k}(\bar{\zb}^{k+1}, \hat{\lbd}^k) \vspace{1ex}\\
&\overset{\tiny\eqref{eq:lm_a1_est1},\eqref{eq:lm_a1_est2}}{\leq}{\!\!\!} f(\xb) + \iprods{\nabla{f}(\bar{\xb}^{k+1}), \bar{\xb}^{k+1} - \xb} + g(\yb) + \iprods{\nabla{g}(\bar{\yb}^{k+1}), \bar{\yb}^{k+1} - \yb}  + \hat{\bar{\Qc}}_{\rho_k}^k(\bar{\yb}^{k+1}) \vspace{1ex}\\
&\overset{\tiny\eqref{eq:lm_a1_opt_cond1}}{=} F(\zb) +  \hat{\bar{\ell}}_{\rho_k}^k(\zb) + \gamma_k\iprods{\hat{\xb}^k - \bar{\xb}^{k+1}, \bar{\xb}^{k+1} - \xb}  
+ \beta_k\iprods{\hat{\yb}^k - \bar{\yb}^{k+1}, \bar{\yb}^{k+1} - \yb} \vspace{1ex}\\
&+~ \tfrac{\rho_k\norms{\Bb}^2}{2}\norms{\bar{\yb}^{k+1} - \hat{\yb}^k}^2  \vspace{1ex}\\
&\overset{\tiny\eqref{eq:linear_func}}{=} F(\zb) + \hat{\bar{\ell}}_{\rho_k}^k(\zb)  + \gamma_k\iprods{\bar{\xb}^{k+1} - \hat{\xb}^k, \xb - \hat{\xb}^k} - \gamma_k\norms{\bar{\xb}^{k+1} - \hat{\xb}^k}^2 \vspace{1ex}\\
& +~ \beta_k\iprods{\bar{\yb}^{k+1} - \hat{\yb}^k, \yb - \hat{\yb}^k} - \tfrac{(2\beta_k - \rho_k\norms{\Bb}^2)}{2}\norms{\bar{\yb}^{k+1} -\hat{\yb}^k}^2,
\end{array}{\!\!\!\!}
\end{equation*}
which is exactly \eqref{eq:descent_pro2}.

(b)~ The optimality conditions of  $\prox_{\gamma_kf}$ and $\prox_{\beta_kg}$ in \eqref{eq:admm_scheme1c} can be written as
\begin{equation}\label{eq:lm_a1c_est1}
\left\{\begin{array}{lll}
0 &= \nabla{f}(\bar{\xb}^{k+1}) + \nabla_x{\phi_{\rho_k}}(\hat{\zb}^k,\hat{\lbd}^k) + \gamma_k(\bar{\xb}^{k+1} - \hat{\xb}^k), &\nabla{f}(\bar{\xb}^{k+1}) \in \partial{f}(\bar{\xb}^{k+1}) \vspace{1ex}\\
0 &= \nabla{g}(\bar{\yb}^{k+1}) + \nabla_y{\phi_{\rho_k}}(\hat{\zb}^k,\hat{\lbd}^k) + \beta_k(\bar{\yb}^{k+1} - \hat{\yb}^k), &\nabla{g}(\bar{\yb}^{k+1}) \in \partial{g}(\bar{\yb}^{k+1}).
\end{array}\right.
\end{equation}
With $\hat{\ell}_{\rho_k}^k$ defined in \eqref{eq:linear_func}, using \eqref{eq:l_property2b}, \eqref{eq:lm_a1_est2}, and \eqref{eq:lm_a1c_est1}, we can derive
\begin{equation*}
\begin{array}{ll}
\Lc_{\rho_k}(\bar{\zb}^{k+1}, \hat{\lbd}^k) {\!\!\!\!\!}&= f(\bar{\xb}^{k+1}) + g(\bar{\yb}^{k+1}) + \phi_{\rho_k}(\bar{\zb}^{k+1}, \hat{\lbd}^k) \vspace{1ex}\\
&\overset{\tiny\eqref{eq:l_property2b}, \eqref{eq:lm_a1_est2}}{\leq} F(\zb) + \iprods{\nabla{f}(\bar{\xb}^{k+1}), \bar{\xb}^{k+1} - \xb} + \iprods{\nabla{g}(\bar{\yb}^{k+1}), \bar{\yb}^{k+1} - \yb}  + \hat{\Qc}_{\rho_k}(\bar{\zb}^{k+1})\vspace{1ex}\\
&\overset{\tiny\eqref{eq:lm_a1c_est1}}{=}  F(\zb)  + \hat{\ell}_{\rho_k}^k(\zb) + \gamma_k\iprods{\bar{\xb}^{k+1} - \hat{\xb}^k, \xb - \hat{\xb}^k} - \tfrac{(2\gamma_k - \rho_k\norms{\Ab}^2)}{2}\norms{\bar{\xb}^{k+1} - \hat{\xb}^k}^2 \vspace{1ex}\\
& +~ \beta_k\iprods{\bar{\yb}^{k+1} - \hat{\yb}^k, \yb - \hat{\yb}^k} - \tfrac{(2\beta_k - \rho_k\norms{\Bb}^2)}{2}\norms{\bar{\yb}^{k+1} -\hat{\yb}^k}^2,
\end{array}
\end{equation*}
which proves \eqref{eq:descent_pro3}.
\end{proof}
\beforesubsec
\subsection{The proof of Theorem~\ref{th:admm-convergence1}: The convergence of Algorithm \ref{alg:A0}}\label{apdx:th:admm-convergence1}
\aftersubsec
%% The proof of Lemma 2.
The proof of Theorem \ref{th:admm-convergence1} requires the following key lemma.

%%% Lemma B.2.
\begin{lemma}\label{le:admm_descent_pro}
Let $\sets{(\bar{\zb}^k, \hat{\lbd}^k,  \tilde{\zb}^{k})}$ be the sequence generated by Step \ref{step:admm_step} of  Algorithm~\ref{alg:A0}.
If $\eta_k \leq \frac{\rho_k\tau_k}{2}$, then we have
\begin{align}\label{eq:descent_pro}
\Lc_{\rho_k}(\bar{\zb}^{k+1}, \lbd) & -~ F^{\star} +  \tfrac{\tau_k}{2\eta_k}\norms{\hat{\lbd}^{k+1} - \lbd}^2 + \tfrac{\gamma_k\tau_k^2}{2}\norms{\tilde{\xb}^{k+1} - \xb^{\star}}^2 + \tfrac{\beta_k\tau_k^2}{2}\norms{\tilde{\yb}^{k+1} - \yb^{\star}}^2 \nonumber\\
&{\!\!} \leq (1-\tau_k)\big[ \Lc_{\rho_{k-1}}(\bar{\zb}^k, \lbd) - F^{\star}\big]  +  \tfrac{\tau_k}{2\eta_k}\norms{\hat{\lbd}^k - \lbd}^2   + \tfrac{\gamma_k\tau_k^2}{2}\norms{\tilde{\xb}^k - \xb^{\star}}^2 +  \tfrac{\beta_k\tau_k^2}{2}\norms{\tilde{\yb}^k - \yb^{\star}}^2 \nonumber\\
&{\!\!} - \tfrac{(\beta_k - 2\rho_k\norms{\Bb}^2)\tau_k^2}{2}\norms{\tilde{\yb}^{k+1} - \tilde{\yb}^k}^2 - \tfrac{(1-\tau_k)}{2}\left[\rho_{k-1} - \rho_k(1-\tau_k)\right]\norms{\bar{\sb}^k}^2,{\!\!\!\!\!\!}
\end{align}
where  $\bar{\sb}^k := \Ab\bar{\xb}^k + \Bb\bar{\yb}^k - \cb$.
\end{lemma}

%% The proof of Lemma B.2.
\begin{proof}
Substituting $\zb := \bar{\zb}^k$ and $\zb := \zb^{\star}$ respectively into \eqref{eq:descent_pro2} of Lemma \ref{le:descent_lemma}, and using \eqref{eq:l_property2} and \eqref{eq:s_terms}, we obtain % with  $\bar{\sb}^k := \Ab\bar{\xb}^k + \Bb\bar{\yb}^k - \cb$, and $\hat{\bar{\sb}}^{k+1} := \Ab\bar{\xb}^{k+1} + \Bb\hat{\yb}^k - \cb$, we obtain
\begin{equation*}
{\!\!\!\!}\begin{array}{ll}
\Lc_{\rho_k}(\bar{\zb}^{k+1}, \hat{\lbd}^k) &{\!\!\!\!\!}\overset{\tiny\eqref{eq:l_property2}}{\leq} \Lc_{\rho_k}(\bar{\zb}^k, \hat{\lbd}^k) - \frac{\rho_k}{2}\norms{\bar{\sb}^k -\hat{\bar{\sb}}^{k+1}}^2 + \gamma_k\iprods{\bar{\xb}^{k+1} - \hat{\xb}^k, \bar{\xb}^k - \hat{\xb}^k} - \gamma_k\norms{\bar{\xb}^{k+1} - \hat{\xb}^k}^2  \vspace{1ex}\\
& +~ \beta_k\iprods{\bar{\yb}^{k+1} - \tilde{\yb}^k, \bar{\yb}^k - \hat{\yb}^k}  - \frac{(2\beta_k - \rho_k\norms{\Bb}^2)}{2}\norms{\bar{\yb}^{k+1} -\hat{\yb}^k}^2, \vspace{1ex}\\
\Lc_{\rho_k}(\bar{\zb}^{k+1}, \hat{\lbd}^k) &{\!\!\!\!\!}\overset{\tiny\eqref{eq:l_property2}}{\leq} F(\zb^{\star}) - \frac{\rho_k}{2}\norms{\hat{\bar{\sb}}^{k+1}}^2   + \gamma_k\iprods{\bar{\xb}^{k+1} - \hat{\xb}^k, \xb^{\star} - \hat{\xb}^k} - \gamma_k\norms{\bar{\xb}^{k+1} - \hat{\xb}^k}^2  \vspace{1ex}\\
& +~ \beta_k\iprods{\bar{\yb}^{k+1} - \hat{\yb}^k, \yb^{\star} - \hat{\yb}^k}  - \frac{(2\beta_k - \rho_k\norms{\Bb}^2)}{2}\norms{\bar{\yb}^{k+1} -\hat{\yb}^k}^2.
\end{array}{\!\!\!\!}
\end{equation*}
Multiplying the first inequality by $(1-\tau_k) \in [0, 1]$ and the second one by $\tau_k \in [0, 1]$, then summing up the results and  using $\tau_k\tilde{\zb}^k = \hat{\zb}^k - (1-\tau_k)\bar{\zb}^k$ and $\tau_k(\tilde{\zb}^{k+1} - \tilde{\zb}^k) = \bar{\zb}^{k+1} - \hat{\zb}^k$ from Step \ref{step:admm_step} of Algorithm \ref{alg:A0} and \eqref{eq:3points_equality}, we get
\begin{equation}\label{eq:lm2_est5}
{\!\!\!\!\!\!\!\!}\begin{array}{ll}
\Lc_{\rho_k}(\bar{\zb}^{k+1}, \hat{\lbd}^k) {\!\!\!\!}&\leq (1-\tau_k)\Lc_{\rho_k}(\bar{\zb}^k, \hat{\lbd}^k) + \tau_k F(\zb^{\star}) - \tfrac{(1-\tau_k)\rho_k}{2}\norms{\bar{\sb}^k - \hat{\bar{\sb}}^{k+1}}^2 - \tfrac{\tau_k\rho_k}{2}\norms{\hat{\bar{\sb}}^{k+1}}^2 \vspace{1ex}\\
&+~ \gamma_k\tau_k\iprods{\bar{\xb}^{k+1} - \hat{\xb}^k, \xb^{\star} - \tilde{\xb}^k} - \gamma_k\norms{\bar{\xb}^{k+1} - \hat{\xb}^k}^2 + \beta_k\tau_k\iprods{\bar{\yb}^{k+1} - \hat{\yb}^k, \yb^{\star} - \tilde{\yb}^k} \vspace{1ex}\\
& -~ \tfrac{\beta_k}{2}\norms{\bar{\yb}^{k+1} - \hat{\yb}^k}^2 - \tfrac{(\beta_k - \rho_k\norms{\Bb}^2)}{2}\norms{\bar{\yb}^{k+1} - \hat{\yb}^k}^2 \vspace{1ex}\\
&\overset{\tiny\eqref{eq:3points_equality}}{\leq}~ (1-\tau_k)\Lc_{\rho_k}(\bar{\zb}^k, \hat{\lbd}^k) + \tau_k F(\zb^{\star}) - \tfrac{(1-\tau_k)\rho_k}{2}\norms{\bar{\sb}^k - \hat{\bar{\sb}}^{k+1}}^2 - \tfrac{\tau_k\rho_k}{2}\norms{\hat{\bar{\sb}}^{k+1}}^2 \vspace{1ex}\\
& +~ \tfrac{\gamma_k\tau_k^2}{2}\left[ \norms{\tilde{\xb}^k - \xb^{\star}}^2 - \norms{\tilde{\xb}^{k+1} - \xb^{\star}}^2\right]   + \tfrac{\beta_k\tau_k^2}{2}\left[ \norms{ \tilde{\yb}^k - \yb^{\star} }^2 - \norms{ \tilde{\yb}^{k+1} - \yb^{\star} }^2\right] \vspace{1ex}\\
& -~ \tfrac{\gamma_k}{2}\norms{\bar{\xb}^{k+1} - \hat{\xb}^k}^2 -  \tfrac{(\beta_k - \rho_k\norms{\Bb}^2)}{2}\norms{\bar{\yb}^{k+1} - \hat{\yb}^k}^2.
\end{array}{\!\!\!\!}
\end{equation}
Using \eqref{eq:basic_pro3} and $\bar{\zb}^{k+1} - (1-\tau_k)\bar{\zb}^k = \tau_k\tilde{\zb}^{k+1}$, for any $\lbd\in\R^n$, \eqref{eq:lm2_est5} implies
\begin{equation}\label{eq:lm2_est6}
{\!\!\!\!\!\!}\begin{array}{ll}
\Lc_{\rho_k}(\bar{\zb}^{k+1}, \lbd) &{\!\!\!\!}\overset{\tiny\eqref{eq:basic_pro3}}{\leq} (1-\tau_k)\Lc_{\rho_{k-1}}(\bar{\zb}^k, \lbd) + \tau_k F(\zb^{\star}) - \frac{(1-\tau_k)\rho_k}{2}\norms{\bar{\sb}^k - \hat{\bar{\sb}}^{k+1}}^2 \vspace{1ex}\\
& -~ \frac{\tau_k\rho_k}{2}\norms{\hat{\bar{\sb}}^{k+1}}^2  +~ \frac{(1-\tau_k)(\rho_k - \rho_{k-1})}{2}\norms{\bar{\sb}^k}^2 - \frac{\gamma_k}{2}\norms{\bar{\xb}^{k+1} - \hat{\xb}^k}^2  \vspace{1ex}\\
&+~ \frac{\gamma_k\tau_k^2}{2}\left[ \norms{\tilde{\xb}^k - \xb^{\star}}^2 - \norms{\tilde{\xb}^{k+1} - \xb^{\star}}^2\right]   + \frac{\beta_k\tau_k^2}{2}\left[ \norms{ \tilde{\yb}^k - \yb^{\star} }^2 - \norms{ \tilde{\yb}^{k+1} - \yb^{\star} }^2\right]  \vspace{1ex}\\ 
&-~  \frac{(\beta_k - \rho_k\norms{\Bb}^2)}{2}\norms{\bar{\yb}^{k+1} - \hat{\yb}^k}^2  + \tau_k\iprods{\hat{\lbd}^k -  \lbd, \Ab\tilde{\xb}^{k+1} + \Bb\tilde{\yb}^{k+1} - \cb}.
\end{array}{\!\!\!\!\!}
\end{equation}
Next, using the update  $\hat{\lbd}^{k+1} = \hat{\lbd}^k - \eta_k(\Ab\tilde{\xb}^{k+1} + \Bb\tilde{\yb}^{k+1} - \cb)$ from the last line of Step \ref{step:admm_step} of Algorithm \ref{alg:A0}, we can estimate $M_k := \iprods{\hat{\lbd}^k - \lbd, \Ab\tilde{\xb}^{k+1} + \Bb\tilde{\yb}^{k+1} - \cb}$ in \eqref{eq:lm2_est6} as
\begin{equation}\label{eq:Mk_est}
\begin{array}{ll}
M_k &:= \iprods{\hat{\lbd}^k - \lbd, \Ab\tilde{\xb}^{k+1} + \Bb\tilde{\yb}^{k+1} - \cb} = \frac{1}{\eta_k}\iprods{\hat{\lbd}^k - \lbd, \hat{\lbd}^k - \hat{\lbd}^{k+1}} \vspace{1ex}\\
%&= \tfrac{1}{2\eta_k}\left[ \norms{\hat{\lbd}^k - \lbd}^2 - \norms{\hat{\lbd}^{k+1} - \lbd}^2 + \norms{\hat{\lbd}^k - \hat{\lbd}^{k+1}}^2 \right] \vspace{1ex}\\
&\overset{\tiny\eqref{eq:3points_equality}}{=} \tfrac{1}{2\eta_k}\left[ \norms{\hat{\lbd}^k - \lbd}^2 - \norms{\hat{\lbd}^{k+1} - \lbd}^2\right] + \frac{\eta_k}{2}\norms{\Ab\tilde{\xb}^{k+1} + \Bb\tilde{\yb}^{k+1} - \cb}^2.
\end{array}
\end{equation}
Moreover, if we define $\overline{\Tc}_k$ as below, then, in view of \eqref{eq:basic_pro1}, we can rearrange it as
\begin{equation}\label{eq:Tc_k_bar}
\begin{array}{ll}
\overline{\Tc}_k &:= \frac{(1-\tau_k)\rho_k}{2}\norms{\bar{\sb}^k - \hat{\bar{\sb}}^{k+1}}^2 + \frac{\tau_k\rho_k}{2}\norms{\hat{\bar{\sb}}^{k+1}}^2 - \frac{(1-\tau_k)(\rho_k - \rho_{k-1})}{2}\norms{\bar{\sb}^k}^2 \vspace{1ex}\\
&\overset{\tiny\eqref{eq:basic_pro1}}{=}  \frac{\rho_k}{2}\norms{\hat{\bar{\sb}}^{k+1} - (1-\tau_k)\bar{\sb}^k}^2 + \frac{(1-\tau_k)}{2}\left[ \rho_{k-1} - \rho_k(1-\tau_k)\right]\norms{\bar{\sb}^k}^2 \vspace{1ex}\\
&= \frac{\rho_k\tau_k^2}{2}\norms{\Ab\tilde{\xb}^{k+1} + \Bb\tilde{\yb}^{k} - \cb}^2 - \frac{(1-\tau_k)}{2}\left[ \rho_{k-1} - \rho_k(1-\tau_k)\right]\norms{\bar{\sb}^k}^2.
\end{array}
\end{equation}
Here, we use the fact that $\hat{\bar{\sb}}^{k+ 1} - (1-\tau_k)\bar{\sb}^k = \Ab\bar{\xb}^{k+ 1} +  \Bb\hat{\yb}^k  -  \cb - (1-\tau_k)\left(\Ab\bar{\xb}^k +  \Bb\bar{\yb}^k  - \cb\right) = \tau_k\left(\Ab\tilde{\xb}^{k+ 1} +  \Bb\tilde{\yb}^k  - \cb\right)$ from the first and fourth lines of Step \ref{step:admm_step} of Algorithm \ref{alg:A0}. 

\noindent Substituting \eqref{eq:Mk_est} and \eqref{eq:Tc_k_bar} into \eqref{eq:lm2_est6}, we can further estimate it as
\begin{equation*}%\label{eq:lm2_est3_v2}
{\!\!\!\!\!\!\!}\begin{array}{ll}
\Lc_{\rho_k}(\bar{\zb}^{k+1}, \lbd) {\!\!\!\!}& -~ F^{\star} +  \tfrac{\tau_k}{2\eta_k}\norms{\hat{\lbd}^{k+1} - \lbd}^2 + \tfrac{\gamma_k\tau_k^2}{2}\norms{\tilde{\xb}^{k+1} - \xb^{\star}}^2 + \tfrac{\beta_k\tau_k^2}{2}\norms{\tilde{\yb}^{k+1} - \yb^{\star}}^2\vspace{1ex}\\
& \leq (1-\tau_k)\big[ \Lc_{\rho_{k-1}}(\bar{\zb}^k, \lbd) - F^{\star}\big]  +  \tfrac{\tau_k}{2\eta_k}\norms{\hat{\lbd}^k - \lbd}^2   + \tfrac{\gamma_k\tau_k^2}{2}\norms{\tilde{\xb}^k - \xb^{\star}}^2 \vspace{1ex}\\
& + ~  \tfrac{\beta_k\tau_k^2}{2}\norms{\tilde{\yb}^k - \yb^{\star}}^2  + \frac{\tau_k\eta_k}{2}\norms{\Ab\tilde{\xb}^{k+1} + \Bb\tilde{\yb}^{k+1} - \cb}^2  - \frac{(\beta_k - \rho_k\norms{\Bb}^2)\tau_k^2}{2}\norms{\tilde{\yb}^{k+1} - \tilde{\yb}^k}^2 \vspace{1ex}\\
&- ~ \frac{\rho_k\tau_k^2}{2}\norms{\Ab\tilde{\xb}^{k+1} + \Bb\tilde{\yb}^k - \cb}^2 - \frac{(1-\tau_k)}{2}\left[\rho_{k-1} - \rho_k(1-\tau_k)\right]\norms{\bar{\sb}^k}^2.
\end{array}{\!\!\!\!\!\!}
\end{equation*}
If $\eta_k \leq \frac{\rho_k\tau_k}{2}$, then we can easily show that
\begin{equation*}
\begin{array}{ll}
\frac{\tau_k\eta_k}{2}\norms{\Ab\tilde{\xb}^{k+1} + \Bb\tilde{\yb}^{k+1} - \cb}^2 -  \frac{\rho_k\tau_k^2}{2}\norms{\Ab\tilde{\xb}^{k+1} + \Bb\tilde{\yb}^k - \cb}^2 - \frac{\rho_k\tau_k^2\norms{\Bb}^2}{2}\norms{\tilde{\yb}^{k+1} - \tilde{\yb}^k}^2 \leq 0.
\end{array}
\end{equation*}
Using this estimate into the last inequality, we can derive 
\begin{equation*} 
{\!\!\!\!}\begin{array}{ll}
\Lc_{\rho_k}(\bar{\zb}^{k+1}, \lbd) {\!\!\!\!}& -~ F^{\star} +  \tfrac{\tau_k}{2\eta_k}\norms{\hat{\lbd}^{k+1} - \lbd}^2 + \tfrac{\gamma_k\tau_k^2}{2}\norms{\tilde{\xb}^{k+1} - \xb^{\star}}^2 + \tfrac{\beta_k\tau_k^2}{2}\norms{\tilde{\yb}^{k+1} - \yb^{\star}}^2\vspace{1ex}\\
&{\!\!\!\!} \leq~ (1-\tau_k)\big[ \Lc_{\rho_{k-1}}(\bar{\zb}^k, \lbd) - F^{\star}\big]  +  \tfrac{\tau_k}{2\eta_k}\norms{\hat{\lbd}^k - \lbd}^2   + \tfrac{\gamma_k\tau_k^2}{2}\norms{\tilde{\xb}^k - \xb^{\star}}^2 +  \tfrac{\beta_k\tau_k^2}{2}\norms{\tilde{\yb}^k - \yb^{\star}}^2\vspace{1ex}\\
& -~ \frac{(1-\tau_k)}{2}\left[\rho_{k-1} - \rho_k(1-\tau_k)\right]\norms{\bar{\sb}^k}^2 -  \frac{(\beta_k - 2\rho_k\norms{\Bb}^2)\tau_k^2}{2}\norms{\tilde{\yb}^{k+1} - \tilde{\yb}^k}^2,
\end{array}{\!\!\!\!}
\end{equation*}
which is exactly \eqref{eq:descent_pro}.
\end{proof}
%%% End of the proof.

Next, we derive update rules for the parameters of Algorithm~\ref{alg:A0} in the following lemma.{\!\!\!}
\begin{lemma}\label{le:update_rules_A0}
Let $\tau_k$, $\gamma_k$, $\beta_k$, $\rho_k$, and $\tau_k$ in Algorithm \ref{alg:A0} be updated as
\begin{equation}\label{eq:update_rules_A0}
\begin{array}{ll}
&0 \leq \gamma_{k+1} \leq \left(\frac{k+2}{k+1}\right)\gamma_k, ~~~~~~~ \beta_{k+1} := 2\rho_0\norms{\Bb}^2(k+2), \vspace{1ex}\\
&\tau_k := \frac{1}{k+1},~~~~~~ \rho_k := \rho_0(k+1),~~~~~~ \text{and}~~~~~~ \eta_k := \frac{\rho_0}{2}.
\end{array}
\end{equation}
Then, the following inequality holds:
\begin{equation}\label{eq:key_est_A0}
{\!\!\!\!\!\!\!}\begin{array}{ll}
(k+1){\!\!\!\!\!}&\big(\Lc_{\rho_k}(\bar{\zb}^{k\!+\!1}, \lbd) - F^{\star}\big) +  \tfrac{1}{\rho_0}\norms{\hat{\lbd}^{k\!+\!1} \!-\! \lbd}^2 + \tfrac{\gamma_k}{2(k+1)}\norms{\tilde{\xb}^{k+1} - \xb^{\star}}^2  + \rho_0\norms{\Bb}^2\norms{\tilde{\yb}^{k\!+\!1} \!-\! \yb^{\star}}^2  \vspace{1.2ex}\\
& \leq k\big(\Lc_{\rho_{k-1}}(\bar{\zb}^k, \lbd) - F^{\star}\big) +  \tfrac{1}{\rho_0}\norms{\hat{\lbd}^k - \lbd}^2  + \tfrac{\gamma_{k-1}}{2k}\norms{\tilde{\xb}^k - \xb^{\star}}^2 + \rho_0\norms{\Bb}^2\norms{\tilde{\yb}^k - \yb^{\star}}^2.
\end{array}{\!\!\!\!\!\!\!}
\end{equation}
\end{lemma}

%%% Begin proof.
\begin{proof}
In order to telescope  \eqref{eq:descent_pro}, we impose the following conditions:
\begin{equation}\label{eq:param_coditions5}
\left\{\begin{array}{ll}
\tfrac{\gamma_k\tau_k^2}{(1-\tau_k)\tau_{k-1}^2} \leq \gamma_{k-1},~~~&\tfrac{\beta_k\tau_k^2}{(1-\tau_k)\tau_{k-1}^2 } \leq  \beta_{k-1} \vspace{1ex}\\
\tau_k\eta_k \leq \eta_{k-1}\tau_{k-1}(1-\tau_k), ~~~& \eta_k \leq \frac{\rho_k\tau_k}{2} \vspace{1ex}\\
2\rho_k\norms{\Bb}^2 \leq \beta_k, ~~~~&\rho_k(1-\tau_k) \leq \rho_{k-1}.
\end{array}\right.
\end{equation}
We first choose $\tau_k = \frac{1}{k+1}$. 
Then, we have $\frac{\tau_k^2}{(1-\tau_k)\tau_{k-1}^2} = \frac{k}{k+1}$ and  $\gamma_{k+1} \leq \big(\frac{k+2}{k+1}\big)\gamma_{k}$.
Next, we update $\rho_k := \rho_0(k+1)$. 
Then, it satisfies $\rho_k(1-\tau_k) \leq \rho_{k-1}$.
Now, we update $\beta_k := 2\rho_k\norms{\Bb}^2 = 2\rho_0\norms{\Bb}^2(k+1)$.
Then, we have $\tfrac{\beta_k\tau_k^2}{(1-\tau_k)\tau_{k-1}^2} = 2\rho_0\norms{\Bb}^2k = \beta_{k-1}$, which satisfies the second condition.
Finally, we choose $\eta_k := \frac{\rho_k\tau_k}{2} = \frac{\rho_0}{2}$.
These all lead to the update rules in \eqref{eq:update_rules_A0}.

\noindent Using the update rules \eqref{eq:update_rules_A0}, \eqref{eq:descent_pro} becomes
\begin{align*} 
\Lc_{\rho_k}(\bar{\zb}^{k+1}, \lbd) & -~ F^{\star} +  \tfrac{1}{\rho_0(k+1)}\norms{\hat{\lbd}^{k+1} - \lbd}^2 + \tfrac{\gamma_k}{2(k+1)^2}\norms{\tilde{\xb}^{k+1} - \xb^{\star}}^2 + \tfrac{\rho_0\norms{\Bb}^2}{k+1}\norms{\tilde{\yb}^{k+1} - \yb^{\star}}^2 \nonumber\\
& \leq \left(\tfrac{k}{k+1}\right)\big[ \Lc_{\rho_{k-1}}(\bar{\zb}^k, \lbd) - F^{\star}\big]  +  \tfrac{1}{\rho_0(k+1)}\norms{\hat{\lbd}^k - \lbd}^2  \nonumber\\
& + \tfrac{\gamma_{k-1}}{2k(k+1)}\norms{\tilde{\xb}^k - \xb^{\star}}^2 +  \tfrac{\rho_0\norms{\Bb}^2}{k+1}\norms{\tilde{\yb}^k - \yb^{\star}}^2,
\end{align*}
which leads to \eqref{eq:key_est_A0}.
\end{proof}
%% End of the proof.

%%% The proof of Theorem 3.1.
\begin{proof}[\textbf{The proof of Theorem \ref{th:admm-convergence1}}]
From \eqref{eq:descent_pro}, by induction, we have 
\begin{equation*}
(k+1)\big(\Lc_{\rho_k}(\bar{\zb}^{k+1}, \lbd) - F^{\star}\big) +  \tfrac{1}{\rho_0(k+1)}\norms{\hat{\lbd}^{k+1} - \lbd}^2 \leq \tfrac{1}{\rho_0}\norms{\hat{\lbd}^0 - \lbd}^2 + \frac{\gamma_0}{2}\norms{\tilde{\xb}^1 - \xb^{\star}}^2 + \rho_0\norms{\Bb}^2\norms{\tilde{\yb}^1 - \yb^{\star}}^2.
\end{equation*}
Let us define $\Rc_k(\lbd) := \Lc_{\rho_k}(\bar{\zb}^k, \lbd) - F^{\star}$.
By using $\tilde{\zb}^0 = \bar{\zb}^0$,  the last inequality implies that 
\begin{equation*}
\Rc_k(\lbd) \leq \frac{1}{k}\left[ \tfrac{1}{\rho_0}\norms{\hat{\lbd}^0 - \lbd}^2 + \frac{\gamma_0}{2}\norms{\bar{\xb}^0 - x^{\star}}^2 + \rho_0\norms{\Bb}^2\norms{\bar{\yb}^0 - \yb^{\star}}^2 \right].
\end{equation*}
Therefore, for any $\rho > 0$, we can show that
\begin{equation*}
{\!\!\!\!}\begin{array}{ll}
\sup\set{ \Rc_k(\lbd) \mid \norms{\lbd} \leq \rho} &{\!\!\!\!\!}\leq \dfrac{1}{k}\left[ \sup\set{ \tfrac{1}{\rho_0}\norms{\hat{\lbd}^0 \!-\! \lbd}^2 \mid \norms{\lbd} \leq \rho} + \frac{\gamma_0}{2}\norms{\bar{\xb}^0 \!-\! x^{\star}}^2 + \rho_0\norms{\Bb}^2\norms{\bar{\yb}^0 \!-\! \yb^{\star}}^2\right] \vspace{1ex}\\
&= \dfrac{1}{k}\left[ \tfrac{1}{\rho_0}\big(\rho - \norms{ \hat{\lbd}^0 }\big)^2 + \frac{\gamma_0}{2}\norms{\bar{\xb}^0 - x^{\star}}^2 + \rho_0\norms{\Bb}^2\norms{\bar{\yb}^0 - \yb^{\star}}^2\right].
\end{array}{\!\!\!\!}
\end{equation*}
By choosing $\rho = 2\norms{\lbd^{\star}}$, and combining the result and Lemma \ref{le:approx_opt_cond}, we obtain the bounds \eqref{eq:admm-convergence1} of Theorem \ref{th:admm-convergence1}.
\end{proof}
%%% End of the proof.

%% The proof of Corollary 3.1.
\beforesubsec
\subsection{The proof of Corollary \ref{co:convergence1c}: Parallel primal-dual decomposition variant}\label{apdx:co:convergence1c}
\aftersubsec
Substituting $\zb = \bar{\zb}^k$ and $\zb = \zb^{\star}$ into \eqref{eq:descent_pro3} of Lemma \ref{le:descent_lemma}, we have
\begin{equation*} 
{\!\!\!\!\!}\begin{array}{ll}
\Lc_{\rho_k}(\bar{\zb}^{k+1},\hat{\lbd}^k) {\!\!\!}&\overset{\tiny\eqref{eq:l_property2}}{\leq} \Lc_{\rho_k}(\bar{\zb}^k,\hat{\lbd}^k) +  \gamma_k\iprods{\bar{\xb}^{k+1} - \hat{\xb}^k, \bar{\xb}^k - \hat{\xb}^k} - \tfrac{(2\gamma_k - \rho_k\norms{\Ab}^2)}{2}\norms{\bar{\xb}^{k+1} - \hat{\xb}^k}^2 \vspace{1ex}\\
& +~ \beta_k\iprods{\bar{\yb}^{k+1} - \hat{\yb}^k, \bar{\yb}^k - \hat{\yb}^k} - \tfrac{(2\beta_k - \rho_k\norms{\Bb}^2)}{2}\norms{\bar{\yb}^{k+1} -\hat{\yb}^k}^2 - \frac{\rho_k}{2}\norms{\bar{\sb}^k - \hat{\sb}^k}^2,\vspace{1ex}\\
\Lc_{\rho_k}(\bar{\zb}^{k+1},\hat{\lbd}^k) {\!\!\!}&\overset{\tiny\eqref{eq:l_property2}}{\leq}  F(\zb^{\star}) +  \gamma_k\iprods{\bar{\xb}^{k+1} - \hat{\xb}^k, \xb^{\star} - \hat{\xb}^k} - \tfrac{(2\gamma_k - \rho_k\norms{\Ab}^2)}{2}\norms{\bar{\xb}^{k+1} - \hat{\xb}^k}^2 \vspace{1ex}\\
& +~ \beta_k\iprods{\bar{\yb}^{k+1} - \hat{\yb}^k, \yb^{\star} - \hat{\yb}^k} - \tfrac{(2\beta_k - \rho_k\norms{\Bb}^2)}{2}\norms{\bar{\yb}^{k+1} -\hat{\yb}^k}^2 - \frac{\rho_k}{2}\norms{\hat{\sb}^k}^2.
\end{array}{\!\!\!\!\!}
\end{equation*}
Multiplying the first inequality by $1-\tau_k\in [0, 1]$ and the second one by $\tau_k \in [0, 1]$, then summing up the results  and using the first and fourth lines of \eqref{eq:admm_scheme1c}, we obtain
\begin{equation*} 
{\!\!\!\!\!}\begin{array}{ll}
\Lc_{\rho_k}(\bar{\zb}^{k+1},\hat{\lbd}^k) {\!\!\!}&\leq (1-\tau_k)\Lc_{\rho_k}(\bar{\zb}^k,\hat{\lbd}^k) + \tau_kF(\zb^{\star}) +  \gamma_k\tau_k\iprods{\bar{\xb}^{k+1} - \hat{\xb}^k, \xb^{\star} - \tilde{\xb}^k} \vspace{1ex}\\
& - ~ \tfrac{(2\gamma_k - \rho_k\norms{\Ab}^2)}{2}\norms{\bar{\xb}^{k+1} - \hat{\xb}^k}^2   +  \beta_k\tau_k\iprods{\bar{\yb}^{k+1} - \hat{\yb}^k, \yb^{\star} - \tilde{\yb}^k} \vspace{1ex}\\
& - ~\tfrac{(2\beta_k - \rho_k\norms{\Bb}^2)}{2}\norms{\bar{\yb}^{k+1} -\hat{\yb}^k}^2 - \frac{(1-\tau_k)\rho_k}{2}\norms{\bar{\sb}^k - \hat{\sb}^k}^2 - \frac{\tau_k\rho_k}{2}\norms{\hat{\sb}^k}^2.
\end{array}{\!\!\!\!\!}
\end{equation*}
With the same proof as in \eqref{eq:lm2_est6}, we can derive from the last inequality that
\begin{equation}\label{eq:co31_proof1}
{\!\!\!\!\!}\begin{array}{ll}
\Lc_{\rho_k}(\bar{\zb}^{k+1}, \lbd) {\!\!\!}&\leq (1-\tau_k)\Lc_{\rho_{k-1}}(\bar{\zb}^k, \lbd) + \tau_kF(\zb^{\star}) +  \frac{\gamma_k\tau_k^2}{2}\big[\norms{\tilde{\xb}^k-\xb^{\star}}^2 - \norms{\tilde{\xb}^{k+1} -\xb^{\star}}^2\big] \vspace{1ex}\\
& - ~ \tfrac{(\gamma_k - \rho_k\norms{\Ab}^2)}{2}\norms{\bar{\xb}^{k+1} - \hat{\xb}^k}^2   +  \frac{\beta_k\tau_k^2}{2}\big[\norms{\tilde{\yb}^k - \yb^{\star}}^2 - \norms{\tilde{\yb}^{k+1} - \yb^{\star}}^2\big] \vspace{1ex}\\
& - ~ \tfrac{(\beta_k - \rho_k\norms{\Bb}^2)}{2}\norms{\bar{\yb}^{k+1} -\hat{\yb}^k}^2 + \tau_k\iprods{\hat{\lbd}^k - \lbd, \Ab\tilde{\xb}^{k+1} + \Bb\tilde{\yb}^{k+1} - \cb} \vspace{1ex}\\
&- ~ \frac{(1-\tau_k)\rho_k}{2}\norms{\bar{\sb}^k - \hat{\sb}^k}^2 - \frac{\tau_k\rho_k}{2}\norms{\hat{\sb}^k}^2 + \frac{(1-\tau_k)(\rho_k - \rho_{k-1})}{2}\norms{\bar{\sb}^k}^2.
\end{array}{\!\!\!\!\!}
\end{equation}
Next, using the update  $\hat{\lbd}^{k+1} = \hat{\lbd}^k - \eta_k(\Ab\tilde{\xb}^{k+1} + \Bb\tilde{\yb}^{k+1} - \cb)$ from \eqref{apdx:co:convergence1c}, \eqref{eq:Mk_est}, and \eqref{eq:Tc_k_bar} with $\hat{\sb}^k$ for $\hat{\bar{\sb}}^k$, we can derive from \eqref{eq:co31_proof1} using the same argument as in \eqref{eq:descent_pro} that
\begin{equation}\label{eq:co31_proof2}
{\!\!\!\!\!}\begin{array}{ll}
\Lc_{\rho_k}(\bar{\zb}^{k+1}, \lbd) {\!\!\!\!}& -~ F^{\star} +  \tfrac{\tau_k}{2\eta_k}\norms{\hat{\lbd}^{k+1} - \lbd}^2 + \tfrac{\gamma_k\tau_k^2}{2}\norms{\tilde{\xb}^{k+1} - \xb^{\star}}^2 + \tfrac{\beta_k\tau_k^2}{2}\norms{\tilde{\yb}^{k+1} - \yb^{\star}}^2\vspace{1ex}\\
& \leq (1-\tau_k)\big[ \Lc_{\rho_{k-1}}(\bar{\zb}^k, \lbd) - F^{\star}\big]  +  \tfrac{\tau_k}{2\eta_k}\norms{\hat{\lbd}^k - \lbd}^2   + \tfrac{\gamma_k\tau_k^2}{2}\norms{\tilde{\xb}^k - \xb^{\star}}^2 \vspace{1ex}\\
& + ~  \tfrac{\beta_k\tau_k^2}{2}\norms{\tilde{\yb}^k - \yb^{\star}}^2  + \frac{\tau_k\eta_k}{2}\norms{\Ab\tilde{\xb}^{k+1} + \Bb\tilde{\yb}^{k+1} - \cb}^2   - \frac{\rho_k\tau_k^2}{2}\norms{\Ab\tilde{\xb}^{k} + \Bb\tilde{\yb}^k - \cb}^2 \vspace{1ex}\\
&- ~ \frac{(\gamma_k - \rho_k\norms{\Ab}^2)\tau_k^2}{2}\norms{\tilde{\xb}^{k+1} - \tilde{\xb}^k}^2- \frac{(\beta_k - \rho_k\norms{\Bb}^2)\tau_k^2}{2}\norms{\tilde{\yb}^{k+1} - \tilde{\yb}^k}^2 \vspace{1ex}\\
&- ~ \frac{(1-\tau_k)}{2}\left[\rho_{k-1} - \rho_k(1-\tau_k)\right]\norms{\bar{\sb}^k}^2.
\end{array}{\!\!\!\!}
\end{equation}
We note that, if $\eta_k \leq \frac{\rho_k\tau_k}{2}$, then
\begin{equation*}
\eta_k\norms{\Ab\tilde{\xb}^{k\!+\!1} \!+\! \Bb\tilde{\yb}^{k\!+\!1} \!-\! \cb}^2  - \rho_k\tau_k\norms{\Ab\tilde{\xb}^{k} \!+\! \Bb\tilde{\yb}^k \!-\! \cb}^2 - \rho_k\tau_k\norms{\Ab}^2\norms{\tilde{\xb}^{k\!+\!1} - \tilde{\xb}^k}^2 - \rho_k\tau_k\norms{\Bb}^2\norms{\tilde{\yb}^{k\!+\!1} - \tilde{\yb}^k}^2 \leq 0.
\end{equation*}
Using this condition into \eqref{eq:co31_proof2}, we obtain
\begin{equation}\label{eq:co31_proof3}
{\!\!}\begin{array}{ll}
\Lc_{\rho_k}(\bar{\zb}^{k+1}, \lbd) {\!\!\!\!}& -~ F^{\star} +  \tfrac{\tau_k}{2\eta_k}\norms{\hat{\lbd}^{k+1} - \lbd}^2 + \tfrac{\gamma_k\tau_k^2}{2}\norms{\tilde{\xb}^{k+1} - \xb^{\star}}^2  +  \tfrac{\beta_k\tau_k^2}{2}\norms{\tilde{\yb}^{k+1} - \yb^{\star}}^2 \vspace{1ex}\\
& \leq (1-\tau_k)\big[ \Lc_{\rho_{k-1}}(\bar{\zb}^k, \lbd) - F^{\star}\big] +  \tfrac{\tau_k}{2\eta_k}\norms{\hat{\lbd}^k - \lbd}^2   + \tfrac{\gamma_k\tau_k^2}{2}\norms{\tilde{\xb}^k - \xb^{\star}}^2 \vspace{1ex}\\
& +  \tfrac{\beta_k\tau_k^2}{2}\norms{\tilde{\yb}^k - \yb^{\star}}^2  - \frac{(\gamma_k - 2\rho_k\norms{\Ab}^2)\tau_k^2}{2}\norms{\tilde{\xb}^{k+1} - \tilde{\xb}^k}^2 \vspace{1ex}\\
& - \frac{(\beta_k - 2\rho_k\norms{\Bb}^2)\tau_k^2}{2}\norms{\tilde{\yb}^{k+1} - \tilde{\yb}^k}^2 - \frac{(1-\tau_k)}{2}\left[\rho_{k-1} - \rho_k(1-\tau_k)\right]\norms{\bar{\sb}^k}^2.
\end{array}{\!\!}
\end{equation}
In order to telescope \eqref{eq:co31_proof3}, we impose the following conditions:
\begin{equation}\label{eq:co31_param_coditions5}
\left\{\begin{array}{lll}
\tfrac{\gamma_k\tau_k^2}{(1-\tau_k)\tau_{k-1}^2} \leq \gamma_{k-1}, &\tfrac{\beta_k\tau_k^2}{(1-\tau_k)\tau_{k-1}^2} \leq \beta_{k-1}, &\vspace{1ex}\\
2\rho_k\norms{\Ab}^2 \leq \gamma_k, & 2\rho_k\norms{\Bb}^2 \leq \beta_k, &\vspace{1ex}\\
\tau_k\eta_k \leq \eta_{k-1}(1-\tau_k)\tau_{k-1}, ~~& \eta_k \leq \frac{\rho_k\tau_k}{2}, &\text{and}~~ \rho_k(1-\tau_k) \leq \rho_{k-1}.
\end{array}\right.
\end{equation}
These conditions lead to the update as in Algorithm \ref{alg:A0} and \eqref{eq:update_etak}.
The rest of the proof follows  the same argument as that of Theorem \ref{th:admm-convergence1}, but using $\bar{R}_0^2 := \rho_0\norms{\Ab}^2\norms{\bar{\xb}^0 - \xb^{\star}}^2 + \rho_0\norms{\Bb}^2\norms{\bar{\yb}^0 - \yb^{\star}}^2 +  \tfrac{1}{\rho_0}\big(2\norms{\lbd^{\star}} - \norms{\hat{\lbd}^0}\big)^2$.
\Eproof
%%% End of the proof.

%%% Lower-bound on rate
\beforesubsec
\subsection{Lower bound on convergence rate}\label{subsec:lower_bound_rate_of_admm}
\aftersubsec
In order to show that the convergence rate of Algorithm \ref{alg:A0} and its variant \eqref{eq:admm_scheme1c} is optimal, we consider the following example:
\begin{equation}\label{eq:composite_exam}
\min_{\xb, \yb} \Big\{ F(\zb) := f(\xb) + g(\yb) ~~\mid~~ \xb - \yb = 0 \Big\},
\end{equation}
which is a split reformulation of $\min_{\xb}\set{ F(\xb) = f(\xb) + g(\xb)}$.
Algorithm \ref{alg:A0} and its parallel variant \eqref{eq:admm_scheme1c} for solving \eqref{eq:composite_exam} are special cases of the following algorithmic scheme:
\begin{equation}\label{eq:general_ADMM}
\left\{\begin{array}{ll}
(\hat{\yb}^k, \hat{\lbd}^k) &\in \mathrm{span}\set{(\bar{\yb}^i, \hat{\lbd}^i) \mid i=0,\cdots, k-1}\vspace{1ex}\\
\bar{\xb}^{k+1} &:= \kprox{\gamma_kf}{\hat{\yb}^k - \gamma_k^{-1}\hat{\lbd}^k} \vspace{1ex}\\
(\tilde{\xb}^{k+1}, \hat{\lbd}^{k+1}) &\in \mathrm{span}\set{(\bar{\xb}^{i+1}, \hat{\lbd}^i) \mid i=0,\cdots, k} \vspace{1ex}\\
\bar{\yb}^{k+1} &:= \kprox{\beta_kg}{\tilde{\xb}^{k+1}  - \beta_k^{-1}\hat{\lbd}^{k+1}} \vspace{1ex}\\
\end{array}\right.
\end{equation}
Then, there exist $f$ and $g$ defined on $\set{\xb\in\R^{6k + 5} \mid \norms{\xb} \leq B}$ which are convex and $L$-smooth such that the general ADMM scheme \eqref{eq:general_ADMM} exhibits the following lower bound:
\begin{equation*}
F(\breve{\xb}^k) \geq \frac{LB}{8(k+1)},
\end{equation*}
where $\breve{\xb}^k := \sum_{i=1}^k\alpha_i\bar{\xb}^i + \sum_{j=1}^k\sigma_j\bar{\yb}^j$ for any $\alpha_i$ and $\sigma_j$ with $i, j =1,\cdots, k$.
This example can be found in \cite{li2016accelerated,woodworth2016tight}.
Clearly, Algorithm \ref{alg:A0} and the scheme \eqref{eq:admm_scheme1c} for solving \eqref{eq:composite_exam} can be cast into \eqref{eq:general_ADMM}.
Hence, their $\BigO{\frac{1}{k}}$ convergence rate is optimal.

%%% The proof of Lemma 3.2.
\beforesec
\section{Convergence analysis of Algorithm~\ref{alg:A0b} and its parallel variant}\label{apdx:convergence_analysis_A0b}
\aftersec
Lemma \ref{le:descent_lemma2} provides key estimates to prove convergence of Algorithm \ref{alg:A0b} and its parallel variant.

%% Lemma 2.
\begin{lemma}\label{le:descent_lemma2}
Assume that $\Lc_{\rho}$ is defined by \eqref{eq:auLag_func}, $\hat{\bar{\ell}}_{\rho}^k$ and $\hat{\ell}_{\rho}^k$ are defined by \eqref{eq:linear_func}, and $\hat{\bar{\Qc}}_{\rho}^k$ and $\hat{\Qc}_{\rho}^k$ are defined by \eqref{eq:Qk_func}.
\begin{itemize}
\item[]$\mathrm{(a)}$
Let $\breve{\yb}^{k+1} := (1-\tau_k)\bar{\yb}^k + \tau_k\tilde{\yb}^{k+1}$ and $(\bar{\xb}^{k+1}, \tilde{\zb}^{k+1}, \hat{\zb}^k, \hat{\lbd}^k)$ be computed by Step \ref{step:scvx_admm_step} of Algorithm \ref{alg:A0b}. 
Then, for any $\zb \in\dom{F}$, we have
\begin{equation}\label{eq:descent_pro4}
\begin{array}{ll}
\breve{\Lc}_{\rho_k}^{k+1} &:= f(\bar{\xb}^{k+1}) + g(\breve{\yb}^{k+1}) + \hat{\bar{\Qc}}_{\rho_k}^k(\breve{\yb}^{k+1})   \leq (1-\tau_k)\left[F(\bar{\zb}^k) + \hat{\bar{\ell}}_{\rho_k}^k(\bar{\zb}^k)\right] \vspace{1ex}\\
& +~ \tau_k\left[ F(\zb) + \hat{\bar{\ell}}_{\rho_k}^k(\zb)\right]   + \frac{\gamma_k\tau_k^2}{2}\norms{\tilde{\xb}^k - \xb}^2 - \frac{\gamma_k\tau_k^2}{2}\norms{\tilde{\xb}^{k+1} - \xb}^2 \vspace{1ex}\\
& +~ \frac{\beta_k\tau_k^2}{2}\norms{\tilde{\yb}^k - \yb}^2 - \frac{\beta_k\tau_k^2 + \mu_g\tau_k}{2}\norms{\tilde{\yb}^{k+1} - \yb}^2 - \tfrac{(\beta_k - \rho_k\norms{\Bb}^2)\tau_k^2}{2}\norms{\tilde{\yb}^{k+1} - \tilde{\yb}^k}^2.
\end{array}
\end{equation}

\item[]$\mathrm{(b)}$
Let $\breve{\zb}^{k+1} := (1-\tau_k)\bar{\zb}^k + \tau_k\tilde{\zb}^{k+1}$ and $(\bar{\zb}^{k+1}, \tilde{\zb}^{k+1}, \hat{\zb}^k, \hat{\lbd}^k)$ be computed by \eqref{eq:admm_scheme3_admm2}-\eqref{eq:admm_scheme3b_admm2}. 
Then, for any $\zb \in\dom{F}$, we have
\begin{equation}\label{eq:descent_pro5}
{\!\!\!\!\!}\begin{array}{ll}
\breve{\bar{\Lc}}_{\rho_k}^{k+1} &:= F(\breve{\zb}^{k+1})  + \hat{\Qc}_{\rho_k}^k(\breve{\zb}^{k+1})   \leq (1-\tau_k)\left[F(\bar{\zb}^k) + \hat{\ell}_{\rho_k}^k(\bar{\zb}^k)\right]  + \tau_k\left[ F(\zb) + \hat{\ell}_{\rho_k}^k(\zb)\right]   \vspace{1ex}\\
& +~ \frac{\gamma_k\tau_k^2}{2}\norms{\tilde{\xb}^k - \xb}^2 - \frac{\gamma_k\tau_k^2 + \mu_f\tau_k}{2}\norms{\tilde{\xb}^{k+1} - \xb}^2 - \frac{(\gamma_k - \rho_k\norms{\Ab}^2)\tau_k^2}{2}\norms{\tilde{\xb}^{k+1} - \tilde{\xb}^k}^2 \vspace{1ex}\\
& +~ \frac{\beta_k\tau_k^2}{2}\norms{\tilde{\yb}^k - \yb}^2 - \frac{\beta_k\tau_k^2 + \mu_g\tau_k}{2}\norms{\tilde{\yb}^{k+1} - \yb}^2 - \tfrac{(\beta_k - \rho_k\norms{\Bb}^2)\tau_k^2}{2}\norms{\tilde{\yb}^{k+1} - \tilde{\yb}^k}^2.
\end{array}{\!\!\!\!\!}
\end{equation}
\end{itemize}
\end{lemma}

%%% The proof of Lemma 2.
\begin{proof}
(a)~
Since $\hat{\zb}^k = (1-\tau_k)\bar{\zb}^k + \tau_k\tilde{\zb}^k$, we have $\breve{\yb}^{k+1} - \hat{\yb}^k = \tau_k(\tilde{\yb}^{k+1} - \tilde{\yb}^k)$ and $\bar{\xb}^{k+1} = (1-\tau_k)\bar{\xb}^k + \tau_k\tilde{\xb}^{k+1}$.
Using these expressions, and the definitions of $\hat{\bar{\ell}}^k_{\rho_k}$ in \eqref{eq:linear_func} and $\Qc^k_{\rho_k}$ in \eqref{eq:Qk_func}, we can derive
\begin{align}\label{eq:lm_a2_proof2}
\hat{\bar{\Qc}}_{\rho_k}^k(\breve{\yb}^{k\!+\!1}) &= \phi_{\rho_k}(\hat{\bar{\zb}}^{k+1},\hat{\lbd}^k) + \iprods{\nabla_y\phi_{\rho_k}(\hat{\bar{\zb}}^{k+1},\hat{\lbd}^k), \breve{\yb}^{k+1} - \hat{\yb}^k} + \tfrac{\rho_k\norms{\Bb}^2}{2}\norms{\breve{\yb}^{k+1} - \hat{\yb}^k}^2 \nonumber\\
& = (1-\tau_k)\left[ \phi_{\rho_k}(\hat{\bar{\zb}}^{k\!+\!1},\hat{\lbd}^k) + \iprods{\nabla_x\phi_{\rho_k}(\hat{\bar{\zb}}^{k\!+\!1},\hat{\lbd}^k), \bar{\xb}^k - \bar{\xb}^{k\!+\!1}} +  \iprods{\nabla_y\phi_{\rho_k}(\hat{\bar{\zb}}^{k\!+\!1},\hat{\lbd}^k), \bar{\yb}^k \!-\! \hat{\yb}^k}\right] \nonumber\\
& + \tau_k\left[ \phi_{\rho_k}(\hat{\bar{\zb}}^{k\!+\!1},\hat{\lbd}^k) +  \iprods{\nabla_x\phi_{\rho_k}(\hat{\bar{\zb}}^{k\!+\!1},\hat{\lbd}^k), \tilde{\xb}^{k\!+\!1} - \bar{\xb}^{k\!+\!1}} +  \iprods{\nabla_y\phi_{\rho_k}(\hat{\bar{\zb}}^{k\!+\!1},\hat{\lbd}^k), \tilde{\yb}^{k\!+\!1} - \hat{\yb}^k} \right]\nonumber\\
&-  \iprods{\nabla_x\phi_{\rho_k}(\hat{\bar{\zb}}^{k+1},\hat{\lbd}^k), (1-\tau_k)\bar{\xb}^k + \tau_k\tilde{\xb}^{k+1} - \bar{\xb}^{k+1}} + \tfrac{\rho_k\tau_k^2\norms{\Bb}^2}{2}\norms{\tilde{\yb}^{k+1} - \tilde{\yb}^k}^2 \nonumber\\
&\overset{\tiny\eqref{eq:linear_func}}{=} (1-\tau_k)\hat{\bar{\ell}}_{\rho_k}^k(\bar{\zb}^k) + \tau_k \hat{\bar{\ell}}_{\rho_k}^k(\tilde{\zb}^{k+1}) + \tfrac{\rho_k\tau_k^2\norms{\Bb}^2}{2}\norms{\tilde{\yb}^{k+1} - \tilde{\yb}^k}^2.
\end{align}
Since $\bar{\xb}^{k+1} = (1-\tau_k)\bar{\xb}^k + \tau_k\tilde{\xb}^{k+1}$, by convexity of $f$, for any $\xb\in\dom{f}$, we can show that
\begin{equation}\label{eq:lm_a2_proof3}
f(\bar{\xb}^{k+1}) \leq (1-\tau_k)f(\bar{\xb}^k) + \tau_kf(\xb) + \tau_k\iprods{\nabla{f}(\bar{\xb}^{k+1}), \tilde{\xb}^{k+1} -  \xb},
\end{equation}
where $\nabla{f}(\bar{\xb}^{k+1})\in\partial{f}(\bar{\xb}^{k+1})$.
Since $\breve{\yb}^{k+1} := (1-\tau_k)\bar{\yb}^k + \tau_k\tilde{\yb}^{k+1}$, by $\mu_g$-convexity of $g$, for any $\yb\in\dom{g}$ and $\nabla{g}(\tilde{\yb}^{k+1})\in\partial{g}(\tilde{\yb}^{k+1})$, we have 
\begin{equation}\label{eq:lm_a2_proof4}
g(\breve{\yb}^{k+1}) \leq (1-\tau_k)g(\bar{\yb}^k) + \tau_kg(\yb) + \tau_k\iprods{\nabla{g}(\tilde{\yb}^{k+1}), \tilde{\yb}^{k+1} -  \yb} - \frac{\tau_k\mu_g}{2}\norms{\tilde{\yb}^{k+1} - \yb}^2.
\end{equation}
Moreover, we have
\begin{equation}\label{eq:lm_a2_proof5}
{\!\!\!\!\!}\begin{array}{ll}
\hat{\bar{\ell}}_{\rho_k}^k(\tilde{\zb}^{k+1}) {\!\!\!} 
&= \hat{\bar{\ell}}_{\rho_k}^k(\zb) + \iprods{\nabla_x\phi_{\rho_k}(\hat{\bar{\zb}}^{k+1},\hat{\lbd}^k), \tilde{\xb}^{k+1} - \xb} +  \iprods{\nabla_y\phi_{\rho_k}(\hat{\bar{\zb}}^{k+1},\hat{\lbd}^k), \tilde{\yb}^{k+1} - \yb}.
\end{array}{\!\!\!\!\!}
\end{equation}
Combining \eqref{eq:lm_a2_proof2}, \eqref{eq:lm_a2_proof3}, \eqref{eq:lm_a2_proof4} and \eqref{eq:lm_a2_proof5}, we can derive
\begin{align}\label{eq:lm_a2_proof6}
\breve{\Lc}_{\rho_k}^{k+1} &\overset{\tiny\eqref{eq:descent_pro4}}{=} f(\bar{\xb}^{k+1}) + g(\breve{\yb}^{k+1}) + \hat{\bar{\Qc}}_{\rho_k}^k(\breve{\yb}^{k+1}) \nonumber\\
&\overset{\tiny\eqref{eq:lm_a2_proof3}, \eqref{eq:lm_a2_proof4}, \eqref{eq:lm_a2_proof5}}{\leq} (1-\tau_k)\big[F(\bar{\zb}^k) + \hat{\bar{\ell}}_{\rho_k}^k(\bar{\zb}^k)\big] + \tau_k\big[ F(\zb) + \hat{\bar{\ell}}_{\rho_k}^k(\zb) \big] \nonumber\\
&+\tau_k \iprods{\nabla{f}(\bar{\xb}^{k\!+\!1}) + \nabla_x\phi_{\rho_k}(\hat{\bar{\zb}}^{k\!+\!1},\hat{\lbd}^k), \tilde{\xb}^{k\!+\!1} -  \xb}   + \tau_k \iprods{\nabla{g}(\tilde{\yb}^{k\!+\!1}) + \nabla_y\phi_{\rho_k}(\hat{\bar{\zb}}^{k\!+\!1},\hat{\lbd}^k), \tilde{\yb}^{k\!+\!1} -  \yb} \nonumber\\
& -  \tfrac{\tau_k\mu_g}{2}\norms{\tilde{\yb}^{k+1} - \yb}^2  +  \tfrac{\rho_k\tau_k^2\norms{\Bb}^2}{2}\norms{\tilde{\yb}^{k+1} - \tilde{\yb}^k}^2.
\end{align}
Next, from the optimality condition of two subproblems at Step \ref{step:scvx_admm_step} of Algorithm \ref{alg:A0b}, we have
\begin{equation}\label{eq:lm_a2_opt_cond1}
\left\{\begin{array}{lll}
0 &= \nabla{f}(\bar{\xb}^{k+1})  + \nabla_{\xb}{\phi_{\rho_k}}(\hat{\bar{\zb}}^{k+1}, \hat{\lbd}^k) + \gamma_k(\bar{\xb}^{k+1} - \hat{\xb}^k), ~&\nabla{f}(\bar{\xb}^{k+1})\in\partial{f}(\bar{\xb}^{k+1}), \vspace{1ex}\\
0 &= \nabla{g}(\tilde{\yb}^{k+1}) + \nabla_{\yb}{\phi_{\rho_k}}(\hat{\bar{\zb}}^{k+1}, \hat{\lbd}^k) + \tau_k\beta_k(\tilde{\yb}^{k+1} - \tilde{\yb}^k), ~&\nabla{g}(\tilde{\yb}^{k+1})\in\partial{g}(\tilde{\yb}^{k+1}).
\end{array}\right.
\end{equation}
Moreover, using \eqref{eq:3points_equality} and $\bar{\xb}^{k+1} - \hat{\xb}^k = \tau_k(\tilde{\xb}^{k+1} - \tilde{\xb}^k)$, we can derive
\begin{equation}\label{eq:lm_a2_proof8}
\begin{array}{ll}
2\tau_k\iprods{\hat{\xb}^k - \bar{\xb}^{k+1}, \tilde{\xb}^{k+1} - \xb} &=  \tau_k^2\norms{\tilde{\xb}^k - \xb}^2 - \tau_k^2\norms{\tilde{\xb}^{k+1} - \xb}^2 - \norms{\bar{\xb}^{k+1} - \hat{\xb}^k}^2 \vspace{1ex}\\
2\iprods{\tilde{\yb}^k - \tilde{\yb}^{k+1}, \tilde{\yb}^{k+1} - \yb} &=  \norms{\tilde{\yb}^k - \yb}^2 - \norms{\tilde{\yb}^{k+1} - \yb}^2 - \norms{\tilde{\yb}^{k+1} - \tilde{\yb}^k}^2.
\end{array}
\end{equation}
Using \eqref{eq:lm_a2_opt_cond1} and \eqref{eq:lm_a2_proof8} into \eqref{eq:lm_a2_proof6}, we can further derive
\begin{equation*} 
\begin{array}{ll}
\breve{\Lc}_{\rho_k}^{k+1} {\!\!\!\!}&\overset{\tiny\eqref{eq:lm_a2_opt_cond1}}{\leq} (1-\tau_k)\left[F(\bar{\zb}^k) + \hat{\bar{\ell}}_{\rho_k}^k(\bar{\zb}^k)\right] + \tau_k\left[ F(\zb) + \hat{\bar{\ell}}_{\rho_k}^k(\zb)\right] - \frac{\tau_k\mu_g}{2}\norms{\tilde{\yb}^{k\!+\!1} - \yb}^2\vspace{1ex}\\
&+ \tau_k\gamma_k \iprods{\hat{\xb}^k - \bar{\xb}^{k+1},\tilde{\xb}^{k+1} -  \xb} + \tau_k^2\beta_k\iprods{\tilde{\yb}^k - \tilde{\yb}^{k+1}, \tilde{\yb}^{k+1} -  \yb}   +  \tfrac{\rho_k\tau_k^2\norms{\Bb}^2}{2}\norms{\tilde{\yb}^{k+1}  - \tilde{\yb}^k}^2 \vspace{1ex}\\
&\overset{\tiny\eqref{eq:lm_a2_proof8}}{\leq} (1-\tau_k)\left[F(\bar{\zb}^k) + \hat{\bar{\ell}}_{\rho_k}^k(\bar{\zb}^k)\right] + \tau_k\left[ F(\zb) + \hat{\bar{\ell}}_{\rho_k}^k(\zb) \right]  \vspace{1ex}\\
&+ \tfrac{\gamma_k\tau_k^2}{2}\norms{\tilde{\xb}^k - \xb}^2 - \tfrac{\gamma_k\tau_k^2}{2}\norms{\tilde{\xb}^{k+1} - \xb}^2 - \tfrac{\gamma_k}{2}\norms{\bar{\xb}^{k+1} - \hat{\xb}^k}^2 \vspace{1ex}\\
&+ \tfrac{\beta_k\tau_k^2}{2}\norm{\tilde{\yb}^k - \yb}^2  - \tfrac{\left(\beta_k\tau_k^2 + \mu_g\tau_k\right)}{2}\norm{\tilde{\yb}^{k+1} - \yb}^2 - \tfrac{(\beta_k - \rho_k\norms{\Bb}^2)\tau_k^2}{2}\norm{\tilde{\yb}^{k+1} - \tilde{\yb}^k}^2,
\end{array} 
\end{equation*}
which is exactly \eqref{eq:descent_pro4}.

(b)~Since $\breve{\zb}^{k+1} := (1-\tau_k)\bar{\zb}^k + \tau_k\tilde{\zb}^{k+1}$, from the definition of $\hat{\ell}_{\rho_k}^k$ in \eqref{eq:linear_func}, we can derive
\begin{align}\label{eq:lm_a1b_proof2}
\hat{\ell}_{\rho_k}^k(\breve{\zb}^{k+1}) &= (1-\tau_k)\hat{\ell}_{\rho_k}^k(\bar{\zb}^{k}) + \tau_k\hat{\ell}_{\rho_k}^k(\tilde{\zb}^{k+1}).
\end{align}
By $\mu_f$-convexity of $f$ and $\mu_g$-convexity of $g$, and  $\breve{\zb}^{k+1} := (1-\tau_k)\bar{\zb}^k + \tau_k\tilde{\zb}^{k+1}$, for any $\zb = (\xb,\yb)\in\dom{F}$, we have
\begin{equation}\label{eq:lm_a1b_proof4}
\begin{array}{ll}
f(\breve{\xb}^{k+1}) &\leq (1-\tau_k)f(\bar{\xb}^k) + \tau_kf(\xb) + \tau_k\iprods{\nabla{f}(\tilde{\xb}^{k+1}), \tilde{\xb}^{k+1} -  \xb} - \frac{\tau_k\mu_f}{2}\norms{\tilde{\xb}^{k+1} - \xb}^2,\vspace{1ex}\\
g(\breve{\yb}^{k+1}) &\leq (1-\tau_k)g(\bar{\yb}^k) + \tau_kg(\yb) + \tau_k\iprods{\nabla{g}(\tilde{\yb}^{k+1}), \tilde{\yb}^{k+1} -  \yb} - \frac{\tau_k\mu_g}{2}\norms{\tilde{\yb}^{k+1} - \yb}^2,
\end{array}
\end{equation}
where $\nabla{f}(\tilde{\xb}^{k+1})\in\partial{f}(\tilde{\xb}^{k+1})$ and $\nabla{g}(\tilde{\yb}^{k+1})\in\partial{g}(\tilde{\yb}^{k+1})$.

\noindent Next, we note that
\begin{equation}\label{eq:lm_a1b_proof5}
{\!\!\!\!\!}\begin{array}{ll}
\hat{\ell}_{\rho_k}^k(\tilde{\zb}^{k+1}) 
%&= \phi_{\rho_k}(\hat{\zb}^k,\hat{\lbd}^k) +  \iprods{\nabla_x\phi_{\rho_k}(\hat{\zb}^k,\hat{\lbd}^k), \tilde{\xb}^{k+1} - \hat{\xb}^k}  +  \iprods{\nabla_y\phi_{\rho_k}(\hat{\zb}^k,\hat{\lbd}^k), \tilde{\yb}^{k+1} - \hat{\yb}^k} \vspace{1ex}\\
%&=~ \phi_{\rho_k}(\hat{\zb}^k,\hat{\lbd}^k) +  \iprods{\nabla_x\phi_{\rho_k}(\hat{\zb}^k,\hat{\lbd}^k), \xb - \hat{\xb}^k} +  \iprods{\nabla_y\phi_{\rho_k}(\hat{\zb}^k,\hat{\lbd}^k), \yb - \hat{\yb}^k} \vspace{1ex}\\
%& +~ \iprods{\nabla_x\phi_{\rho_k}(\hat{\zb}^k,\hat{\lbd}^k), \tilde{\xb}^{k+1} - \xb} +  \iprods{\nabla_y\phi_{\rho_k}(\hat{\zb}^k,\hat{\lbd}^k), \tilde{\yb}^{k+1} - \yb} \vspace{1ex}\\
&=~ \hat{\ell}_{\rho_k}^k(\zb) + \iprods{\nabla_x\phi_{\rho_k}(\hat{\zb}^k,\hat{\lbd}^k), \tilde{\xb}^{k+1} - \xb} +  \iprods{\nabla_y\phi_{\rho_k}(\hat{\zb}^k,\hat{\lbd}^k), \tilde{\yb}^{k+1} - \yb}.
\end{array}{\!\!}
\end{equation}
Combining \eqref{eq:lm_a1b_proof2}, \eqref{eq:lm_a1b_proof4} and \eqref{eq:lm_a1b_proof5}, we can derive
\begin{equation}\label{eq:lm_a1b_proof6}
\begin{array}{ll}
\breve{\bar{\Lc}}_{\rho_k}^{k+1} &\overset{\tiny\eqref{eq:Qk_func}}{=} F(\breve{\zb}^{k+1})  + \hat{\ell}_{\rho_k}^k(\breve{\zb}^{k+1})  + \frac{\rho_k\norms{\Ab}^2}{2}\norms{\breve{\xb}^{k+1} - \hat{\xb}^k}^2 + \frac{\rho_k\norms{\Bb}^2}{2}\norms{\breve{\yb}^{k+1} - \hat{\yb}^k}^2 \vspace{1ex}\\
&\overset{\tiny\eqref{eq:lm_a1b_proof2}, \eqref{eq:lm_a1b_proof4}, \eqref{eq:lm_a1b_proof5}}{\leq} (1-\tau_k)\big[F(\bar{\zb}^k) + \hat{\ell}_{\rho_k}^k(\bar{\zb}^k) \big] + \tau_k\big[ F(\zb) + \hat{\ell}_{\rho_k}^k(\zb) \big] \vspace{1ex}\\
&+ \tau_k \iprods{\nabla{f}(\tilde{\xb}^{k+1}) + \nabla_x\phi_{\rho_k}(\hat{\zb}^k,\hat{\lbd}^k), \tilde{\xb}^{k+1} -  \xb}  - \frac{\tau_k\mu_f}{2}\norms{\tilde{\xb}^{k+1} - \xb}^2 \vspace{1ex}\\
&+ \tau_k \iprods{\nabla{g}(\tilde{\yb}^{k+1}) + \nabla_y\phi_{\rho_k}(\hat{\zb}^k,\hat{\lbd}^k), \tilde{\yb}^{k+1} -  \yb} - \frac{\tau_k\mu_g}{2}\norms{\tilde{\yb}^{k+1} - \yb}^2 \vspace{1ex}\\
&+  \tfrac{\rho_k\tau_k^2\norms{\Ab}^2}{2}\norms{\tilde{\xb}^{k+1} - \tilde{\xb}^k}^2 +  \tfrac{\rho_k\tau_k^2\norms{\Bb}^2}{2}\norms{\tilde{\yb}^{k+1} - \tilde{\yb}^k}^2.
\end{array}
\end{equation}
The optimality condition of two subproblems in \eqref{eq:admm_scheme3_admm2} can be written as 
\begin{equation}\label{eq:lm_a1b_opt_cond1}
\left\{\begin{array}{lll}
0 &= \nabla{f}(\tilde{\xb}^{k+1})  + \nabla_{\xb}{\phi_{\rho_k}}(\hat{\zb}^k, \hat{\lbd}^k) + \tau_k\gamma_k(\tilde{\xb}^{k+1} - \tilde{\xb}^k), ~&\nabla{f}(\tilde{\xb}^{k+1})\in\partial{f}(\tilde{\xb}^{k+1}), \vspace{1ex}\\
0 &= \nabla{g}(\tilde{\yb}^{k+1}) + \nabla_{\yb}{\phi_{\rho_k}}(\hat{\zb}^k, \hat{\lbd}^k) + \tau_k\beta_k(\tilde{\yb}^{k+1} - \tilde{\yb}^k), ~&\nabla{g}(\tilde{\yb}^{k+1})\in\partial{g}(\tilde{\yb}^{k+1}).
\end{array}\right.
\end{equation}
Moreover, using  \eqref{eq:3points_equality}, we also have
\begin{equation}\label{eq:lm_a1b_proof8}
\begin{array}{ll}
2\iprods{\tilde{\zb}^k - \tilde{\zb}^{k+1}, \tilde{\zb}^{k+1} - \zb} &=  \norms{\tilde{\zb}^k - \zb}^2 - \norms{\tilde{\zb}^{k+1} - \zb}^2 - \norms{\tilde{\zb}^{k+1} - \tilde{\zb}^k}^2.
\end{array}
\end{equation}
Using \eqref{eq:lm_a1b_opt_cond1} and \eqref{eq:lm_a1b_proof8} into \eqref{eq:lm_a1b_proof6}, we can further derive
\begin{align*} 
\begin{array}{ll}
\breve{\bar{\Lc}}_{\rho_k}^{k+1} {\!\!\!\!}&\overset{\tiny\eqref{eq:lm_a1b_opt_cond1}}{\leq} (1-\tau_k)\big[ F(\bar{\zb}^k) + \hat{\ell}_{\rho_k}^k(\bar{\zb}^k)\big] + \tau_k\big[ F(\zb) + \hat{\ell}_{\rho_k}^k(\zb) \big]  + \tau_k^2\gamma_k \iprods{\tilde{\xb}^k - \tilde{\xb}^{k+1}, \tilde{\xb}^{k+1} -  \xb} \vspace{1ex}\\
&  +~ \tau_k^2\beta_k\iprods{\tilde{\yb}^k - \tilde{\yb}^{k+1}, \tilde{\yb}^{k+1} -  \yb}  - \frac{\tau_k\mu_f}{2}\norms{\tilde{\xb}^{k+1} - \xb}^2  - \frac{\tau_k\mu_g}{2}\norms{\tilde{\yb}^{k+1} - \yb}^2 \vspace{1ex}\\
& + ~ \tfrac{\rho_k\tau_k^2\norms{\Ab}^2}{2}\norms{\tilde{\xb}^{k+1} -  \tilde{\xb}^k}^2 + \tfrac{\rho_k\tau_k^2\norms{\Bb}^2}{2}\norms{\tilde{\yb}^{k+1} - \tilde{\yb}^k}^2 \vspace{1ex}\\
&\overset{\tiny\eqref{eq:lm_a1b_proof8}}{\leq} (1-\tau_k)\big[F(\bar{\zb}^k) + \hat{\ell}_{\rho_k}^k(\bar{\zb}^k)\big] + \tau_k\big[ F(\zb) + \hat{\ell}_{\rho_k}^k(\zb) \big]  \vspace{1ex}\\
&+ ~ \tfrac{\gamma_k\tau_k^2}{2}\norms{\tilde{\xb}^k - \xb}^2 - \tfrac{(\gamma_k\tau_k^2 + \mu_f\tau_k)}{2}\norms{\tilde{\xb}^{k+1} - \xb}^2 - \tfrac{(\gamma_k - \rho_k\norms{\Ab}^2)\tau_k^2}{2}\norms{\tilde{\xb}^{k+1} - \tilde{\xb}^k}^2 \vspace{1ex}\\
&+ ~ \tfrac{\beta_k\tau_k^2}{2}\norms{\tilde{\yb}^k - \yb}^2  - \tfrac{(\beta_k\tau_k^2 + \mu_g\tau_k)}{2}\norms{\tilde{\yb}^{k+1} - \yb}^2 - \tfrac{(\beta_k - \rho_k\norms{\Bb}^2)\tau_k^2}{2}\norms{\tilde{\yb}^{k+1} - \tilde{\yb}^k}^2,
\end{array}
\end{align*}
which is exactly \eqref{eq:descent_pro5}.
\end{proof}
%% End of the proof.

%%%%%%%%%%%%%%%%%%%%%%%%%%%%%%%%%%%%%%%%%%%
%%% The proof of Theorem 
%%%%%%%%%%%%%%%%%%%%%%%%%%%%%%%%%%%%%%%%%%%
\subsection{The proof of Theorem \ref{th:convergence3b}: Convergence analysis of Algorithm \ref{alg:A0b}}\label{apdx:th:convergence3b}
We divide our analysis into two lemmas as follows.

%% Lemma C.2.
\begin{lemma}\label{le:th41_step1}
If $\tau_k$, $\gamma_k$, $\beta_k$, $\rho_k$, and $\eta_k$ of Algorithm \ref{alg:A0b} are updated such that
\begin{equation}\label{eq:th41_param_coditions6}
\left\{\begin{array}{ll}
\tfrac{\gamma_k\tau_k^2}{(1-\tau_k)\tau_{k-1}^2} \leq  \gamma_{k-1}, ~~& \tfrac{\beta_k\tau_k^2}{(1-\tau_k)\tau_{k-1}} \leq \beta_{k-1}\tau_{k-1} + \mu_g,\vspace{1ex}\\
\tau_k\eta_k \leq \eta_{k-1}(1-\tau_k)\tau_{k-1}, &\eta_k \leq \frac{\rho_k\tau_k}{2}, \vspace{1ex}\\
\rho_k(1-\tau_k) \leq \rho_{k-1}, &2\rho_k\norms{\Bb}^2 \leq \beta_k,
\end{array}\right.
\end{equation}
then, for any $\lbd\in\R^n$, one has
\begin{align}\label{eq:th41_step1}
{\!\!\!\!\!\!\!}\begin{array}{ll}
\Lc_{\rho_k}(\bar{\zb}^{k\!+\!1},\lbd) {\!\!\!}& -~ F(\zb^{\star}) +  \tfrac{\tau_k}{2\eta_k}\norms{\hat{\lbd}^{k\!+\!1} - \lbd}^2 + \tfrac{\gamma_k\tau_k^2}{2}\norms{\tilde{\xb}^{k\!+\!1} - \xb^{\star}}^2 +  \tfrac{(\beta_k\tau_k^2 + \mu_g\tau_k)}{2}\norms{\tilde{\yb}^{k\!+\!1} - \yb^{\star}}^2 \vspace{1ex}\\
& \leq~ (1-\tau_k)\left[ \Lc_{\rho_{k-1}}(\bar{\zb}^k,\lbd)  - F(\zb^{\star})\right]  + \tfrac{\tau_k}{2\eta_k}\norms{\hat{\lbd}^k - \lbd}^2 + \tfrac{\gamma_k\tau_k^2}{2}\norms{\tilde{\xb}^k - \xb^{\star}}^2 \vspace{1ex}\\
&  + \tfrac{\beta_k\tau_k^2}{2}\norms{\tilde{\yb}^k - \yb^{\star}}^2.
\end{array}{\!\!\!\!\!}
\end{align}
\end{lemma}

%% Begin the proof.
\begin{proof}
Using \eqref{eq:basic_pro3} and \eqref{eq:l_property2}, we can derive from \eqref{eq:descent_pro4} that
{\!\!\!\!\!\!\!\!}\begin{align}\label{eq:th41_proof1}
\begin{array}{ll}
\breve{\Lc}_{\rho_k}^{k+1} &{\!\!\!}\leq (1-\tau_k)\Lc_{\rho_{k\!-\!1}}(\bar{\zb}^k,\hat{\lbd}^k) + \tau_kF(\zb^{\star}) - \frac{(1-\tau_k)\rho_k}{2}\norms{\hat{\bar{\sb}}^{k+1} - \bar{\sb}^k}^2 - \frac{\tau_k\rho_k}{2}\norms{\hat{\bar{\sb}}^{k+1}}^2 \vspace{1ex}\\
&{\!\!\!} + \frac{(1-\tau_k)(\rho_k - \rho_{k\!-\!1})}{2}\norms{\bar{\sb}^k}^2 + \tfrac{\gamma_k\tau_k^2}{2}\norms{\tilde{\xb}^k \!-\! \xb^{\star}}^2 - \tfrac{\gamma_k\tau_k^2}{2}\norms{\tilde{\xb}^{k\!+\!1} - \xb^{\star}}^2 - \tfrac{\gamma_k}{2}\norms{\bar{\xb}^{k\!+\!1} - \hat{\xb}^k}^2 \vspace{1ex}\\
&{\!\!\!} + \tfrac{\beta_k\tau_k^2}{2}\norms{\tilde{\yb}^k - \yb^{\star}}^2  - \tfrac{(\beta_k\tau_k^2 + \mu_g\tau_k)}{2}\norms{\tilde{\yb}^{k+1} - \yb^{\star}}^2 - \tfrac{(\beta_k - \rho_k\norms{\Bb}^2)\tau_k^2}{2}\norms{\tilde{\yb}^{k+1} - \tilde{\yb}^k}^2.
\end{array}{\!\!\!\!}
\end{align}
Now, we consider two cases corresponding to two options in Algorithm \ref{alg:A0b}.

\noindent\textbf{Option 1:} If $\bar{\yb}^{k+1} = \breve{\yb}^{k+1}$, i.e., the \textbf{averaging step} is used, then we have 
\begin{equation}\label{eq:th41_opt1_est1}
{\!\!\!\!\!\!\!}\begin{array}{ll}
\Lc_{\rho_k}(\bar{\zb}^{k+1},\hat{\lbd}^k) {\!\!\!\!}&= F(\bar{\zb}^{k+1}) + \phi_{\rho_k}(\bar{\zb}^{k+1}, \hat{\lbd}^k) 
\overset{\tiny\eqref{eq:l_property2b}}{\leq} f(\bar{\xb}^{k+1}) + g(\breve{\yb}^{k+1}) + \hat{\bar{\Qc}}_{\rho_k}^k(\breve{\yb}^{k+1}) 
 = \breve{\Lc}_{\rho_k}^{k+1}.
\end{array}{\!\!\!\!}
\end{equation}
\noindent\textbf{Option 2:} If we compute $\bar{\yb}^{k+1}$ by the \textbf{proximal step} at Step \ref{step:scvx_admm_step} of Algorithm \ref{alg:A0b}, then 
\begin{equation*}
\bar{\yb}^{k+1} = \argmin_{\yb}\Big\{ g(\yb) + \iprods{\nabla_y{\phi_{\rho_k}}(\hat{\bar{\zb}}^{k+1},\hat{\lbd}^k), \yb - \hat{\yb}^k} + \tfrac{\rho_k\norms{\Bb}^2}{2}\norms{\yb - \hat{\yb}^k}^2 \Big\}.
\end{equation*}
Using this fact, and \eqref{eq:l_property2b}, we can derive
\begin{equation}\label{eq:th41_opt1_est2}
{\!\!\!\!\!\!\!}\begin{array}{ll}
\Lc_{\rho_k}(\bar{\zb}^{k+1},\hat{\lbd}^k) &= f(\bar{\xb}^{k+1}) + g(\bar{\yb}^{k+1}) + \phi_{\rho_k}(\bar{\zb}^{k+1},\hat{\lbd}^k) \vspace{1ex}\\
&\overset{\tiny\eqref{eq:l_property2b}}{\leq} 
f(\bar{\xb}^{k+1}) + g(\bar{\yb}^{k+1})  +  \phi_{\rho_k}(\hat{\bar{\zb}}^{k+1},\hat{\lbd}^k) + \iprods{\nabla_y\phi_{\rho_k}(\hat{\bar{\zb}}^{k+1},\hat{\lbd}^k), \bar{\yb}^{k+1} - \hat{\yb}^k} \vspace{1ex}\\
& + \tfrac{\rho_k\norms{\Bb}^2}{2}\norms{\bar{\yb}^{k+1} - \hat{\yb}^k}^2 \vspace{1ex}\\
&\leq f(\bar{\xb}^{k+1}) + g(\breve{\yb}^{k+1})  +  \phi_{\rho_k}(\hat{\bar{\zb}}^{k+1},\hat{\lbd}^k) + \iprods{\nabla_y\phi_{\rho_k}(\hat{\bar{\zb}}^{k+1},\hat{\lbd}^k), \breve{\yb}^{k+1} - \hat{\yb}^k} \vspace{1ex}\\
& + \tfrac{\rho_k\norms{\Bb}^2}{2}\norms{\breve{\yb}^{k+1} - \hat{\yb}^k}^2 - \frac{\rho_k\norms{\Bb}^2}{2}\norms{\bar{\yb}^{k+1} - \breve{\yb}^{k+1}}^2 \vspace{1ex}\\
&= \breve{\Lc}_{\rho_k}^{k+1} - \frac{\rho_k\norms{\Bb}^2}{2}\norms{\bar{\yb}^{k+1} - \breve{\yb}^{k+1}}^2.
\end{array}{\!\!\!\!\!\!\!}
\end{equation}
Similar to the proof of \eqref{eq:descent_pro}, we can derive from \eqref{eq:th41_proof1} and either \eqref{eq:th41_opt1_est1} or \eqref{eq:th41_opt1_est2} that 
\begin{equation}\label{eq:th41_proof2}
{\!\!\!\!\!\!\!}\begin{array}{ll}
\Lc_{\rho_k}(\bar{\zb}^{k+1},\lbd) {\!\!\!}& -~ F(\zb^{\star}) +  \tfrac{\tau_k}{2\eta_k}\norms{\hat{\lbd}^{k+1} - \lbd}^2 + \tfrac{\gamma_k\tau_k^2}{2}\norms{\tilde{\xb}^{k+1} - \xb^{\star}}^2 +  \tfrac{(\beta_k\tau_k^2 + \mu_g\tau_k)}{2}\norms{\tilde{\yb}^{k+1} - \yb^{\star}}^2 \vspace{1ex}\\
& \leq~ (1-\tau_k)\left[ \Lc_{\rho_{k-1}}(\bar{\zb}^k,\lbd)  - F(\zb^{\star})\right]  + \tfrac{\tau_k}{2\eta_k}\norms{\hat{\lbd}^k - \lbd}^2 + \tfrac{\gamma_k\tau_k^2}{2}\norms{\tilde{\xb}^k - \xb^{\star}}^2 \vspace{1ex}\\
&  + \tfrac{\beta_k\tau_k^2}{2}\norms{\tilde{\yb}^k - \yb^{\star}}^2  - \frac{(1-\tau_k)}{2}\big[\rho_{k-1} - \rho_k(1-\tau_k)\big]\norms{\bar{\sb}^k}^2 \vspace{1ex}\\
&+~ \frac{\eta_k\tau_k}{2}\norms{\Ab\tilde{\xb}^{k\!+\!1} + \Bb\tilde{\yb}^{k\!+\!1} - \cb}^2 \!-\! \frac{\rho_k\tau_k^2}{2}\norms{\Ab\tilde{\xb}^{k\!+\!1} + \Bb\tilde{\yb}^k \!-\! \cb}^2 \vspace{1ex}\\
& - \frac{\rho_k\tau_k^2\norms{\Bb}^2}{2}\norms{\tilde{\yb}^{k\!+\!1} - \tilde{\yb}^k}^2  - \frac{\left(\beta_k - 2\rho_k\norms{\Bb}^2\right)\tau_k^2}{2}\norms{\tilde{\yb}^{k\!+\!1} - \tilde{\yb}^k}^2.
\end{array}{\!\!\!\!\!\!}
\end{equation}
In order to telescope \eqref{eq:th41_proof2}, we need to impose the following conditions
\begin{equation*} 
\left\{\begin{array}{ll}
\tfrac{\gamma_k\tau_k^2}{(1-\tau_k)\tau_{k-1}^2} \leq  \gamma_{k-1}, ~~& \tfrac{\beta_k\tau_k^2}{(1-\tau_k)\tau_{k-1}^2} \leq \beta_{k-1} + \frac{\mu_g}{\tau_{k-1}},\vspace{1ex}\\
\tau_k\eta_k \leq \eta_{k-1}(1-\tau_k)\tau_{k-1}, &\eta_k \leq \frac{\rho_k\tau_k}{2}, \vspace{1ex}\\
\rho_k(1-\tau_k) \leq \rho_{k-1}, &2\rho_k\norms{\Bb}^2 \leq \beta_k,
\end{array}\right.
\end{equation*}
which is exactly \eqref{eq:th41_param_coditions6}.

\noindent Under the condition $\eta_k \leq \frac{\rho_k\tau_k}{2}$, we have
\begin{equation*}
\eta_k\norms{\Ab\tilde{\xb}^{k+1} + \Bb\tilde{\yb}^{k+1} - \cb}^2 - \rho_k\tau_k\norms{\Ab\tilde{\xb}^{k+1} + \Bb\tilde{\yb}^k - \cb}^2 - \rho_k\tau_k\norms{\Bb}^2\norms{\tilde{\yb}^{k+1} - \tilde{\yb}^k}^2 \leq 0.
\end{equation*}
Using this inequality and \eqref{eq:th41_param_coditions6} into \eqref{eq:th41_proof2}, we can simplify it as
\begin{align*} 
{\!\!\!\!\!}\begin{array}{ll}
&\Lc_{\rho_k}(\bar{\zb}^{k+1},\lbd)  - F(\zb^{\star}) +  \tfrac{\tau_k}{2\eta_k}\norms{\hat{\lbd}^{k+1} - \lbd}^2 + \tfrac{\gamma_k\tau_k^2}{2}\norms{\tilde{\xb}^{k+1} - \xb^{\star}}^2 +  \tfrac{(\beta_k\tau_k^2 + \mu_g\tau_k)}{2}\norms{\tilde{\yb}^{k+1} - \yb^{\star}}^2 \vspace{1ex}\\
&~~~ \leq~ (1-\tau_k)\left[ \Lc_{\rho_{k-1}}(\bar{\zb}^k,\lbd)  - F(\zb^{\star})\right]  + \tfrac{\tau_k}{2\eta_k}\norms{\hat{\lbd}^k - \lbd}^2 + \tfrac{\gamma_k\tau_k^2}{2}\norms{\tilde{\xb}^k - \xb^{\star}}^2  + \tfrac{\beta_k\tau_k^2}{2}\norms{\tilde{\yb}^k - \yb^{\star}}^2,
\end{array}{\!\!\!\!}
\end{align*}
which is exactly \eqref{eq:th41_step1}.
\end{proof}
%%% End of the proof.

\noindent Using \eqref{eq:th41_param_coditions6}, we can derive  update rules for parameters in the following lemma.

%%% Lemma C.3.
\begin{lemma}\label{le:th41_proof_step2}
Assume that $\tau_k$, $\gamma_k$, $\beta_k$, $\rho_k$, and $\eta_k$ in Algorithm \ref{alg:A0b} are updated as
\begin{equation}\label{eq:th41_update_rules}
\begin{array}{ll}
\tau_{k+1} := \frac{\tau_k}{2}\left(\sqrt{\tau_k^2 + 4} - \tau_k\right)~\text{with}~\tau_0 := 1,\vspace{1ex}\\
\gamma_k := \gamma_0 \geq 0, ~~\beta_k := 2\rho_k\norms{\Bb}^2, ~~\rho_k := \frac{\rho_0}{\tau_k^2}, ~~\text{and}~~\eta_k := \frac{\rho_k\tau_k}{2},
\end{array}
\end{equation}
where $\rho_0 \in \left(0, \frac{\mu_g}{4\norms{\Bb}^2}\right]$.
Then, for any $\lbd\in\R^n$, the following estimate holds
\begin{equation}\label{eq:th41_proof_step2}
{\!\!\!\!\!\!\!\!}\begin{array}{ll}
\Lc_{\rho_k}(\bar{\zb}^{k\!+\!1},\lbd) {\!\!\!\!}&-~ F(\zb^{\star}) +  \tfrac{\tau_k^2}{\rho_0}\norms{\hat{\lbd}^{k\!+\!1} - \lbd}^2 + \frac{\gamma_0\tau_k^2}{2}\norms{\tilde{\xb}^{k\!+\!1} - \xb^{\star}}^2 + \frac{2\rho_0\norms{\Bb}^2 + \mu_g\tau_k}{2}\norms{\tilde{\yb}^{k\!+\!1} - \yb^{\star}}^2  \vspace{1ex}\\
& \leq~ \left(1-\tau_k\right)\big[ \Lc_{\rho_{k-1}}(\bar{\zb}^k,\lbd)  - F(\zb^{\star}) \big]   + \tfrac{\tau_{k-1}^2(1-\tau_k)}{\rho_0}\norms{\hat{\lbd}^k - \lbd}^2 \vspace{1ex}\\
& +~ \tfrac{\gamma_0\tau_{k-1}^2(1-\tau_k)}{2}\norms{\tilde{\xb}^k - \xb^{\star}}^2  + \frac{(2\rho_0\norm{\Bb}^2 + \mu_g\tau_{k-1})(1-\tau_k)}{2}\norms{\tilde{\yb}^k - \yb^{\star}}^2.
\end{array}{\!\!\!\!}
\end{equation}
\end{lemma}

%% Begin proof.
\begin{proof}
Let us first update $\tau_k$ as 
\begin{equation*}
\tau_{k+1} := \tfrac{\tau_k}{2}\left[ (\tau_k^2 + 4)^{1/2} - \tau_k\right]~~\text{with}~~\tau_0 := 1.
\end{equation*}
Then, we have $\frac{\tau_k^2}{(1-\tau_k)\tau_{k-1}^2} = 1$ and $\frac{1}{k+1} \leq \tau_k \leq \frac{2}{k+2}$. 
Hence, we can update $\gamma_{k+1} = \gamma_k = \gamma_0 \geq 0$.

Next, we update $\rho_k := \frac{\rho_{k-1}}{1-\tau_k} = \frac{\rho_{k-1}\tau_{k-1}^2}{\tau_k^2} = \frac{\rho_0}{\tau_k^2}$.
Then, $\rho_k$ satisfies the fifth condition of \eqref{eq:th41_param_coditions6}.
Now, we update $\beta_k := 2\rho_k\norms{\Bb}^2$.
We need to check the second condition, which is equivalent to 
\begin{equation*}
\tfrac{2\rho_0\norms{\Bb}^2}{\tau_k^2} \leq \tfrac{2\rho_0\norms{\Bb}^2}{\tau_{k-1}^2} + \tfrac{\mu_g}{\tau_{k-1}}.
\end{equation*}
Hence, $2\rho_0\norms{\Bb}^2\big(\frac{1}{\tau_k^2} - \frac{1}{\tau_{k-1}^2}\big)\tau_{k-1} \leq \mu_g$.
We note that $\big(\frac{1}{\tau_k^2} - \frac{1}{\tau_{k-1}^2}\big)\tau_{k-1} = \frac{\tau_{k-1}}{\tau_k} \leq 2$.
The condition $2\rho_0\norms{\Bb}^2\big(\frac{1}{\tau_k^2} - \frac{1}{\tau_{k-1}^2}\big)\tau_{k-1} \leq \mu_g$ holds if $4\rho_0\norms{\Bb}^2 \leq \mu_g$.
Therefore, we need to choose $\rho_0$ such that $\rho_0 \leq \frac{\mu_g}{4\norms{\Bb}^2}$.

It remains to choose $\eta_k$. We choose $\eta_k := \frac{\rho_k\tau_k}{2} = \frac{\rho_0}{\tau_k}$.
We choose $\eta_k$ from the third condition of \eqref{eq:th41_param_coditions6}, which leads to $\eta_k\tau_k = (1-\tau_k)\tau_{k-1}\eta_{k-1}$.
Hence, we have $\eta_k := \frac{(1-\tau_k)\tau_{k-1}}{\tau_k}\eta_{k-1} = \frac{\tau_k}{\tau_{k-1}}\eta_{k-1} = \eta_0\tau_k$.
This update leads to $\eta_k \leq \frac{\rho_k\tau_k}{2}$ if $\eta_0 \leq \frac{\rho_0}{2}$, which holds if $\eta_0 = \frac{\rho_0}{2}$.

\noindent Using the update rules from \eqref{eq:th41_param_coditions6}, \eqref{eq:th41_step1} implies
\begin{align*} 
{\!\!\!}\begin{array}{ll}
\Lc_{\rho_k}(\bar{\zb}^{k+1},\lbd) - F(\zb^{\star}) &+  \tfrac{\tau_k^2}{\rho_0}\norms{\hat{\lbd}^{k+1} - \lbd}^2 + \frac{\gamma_0\tau_k^2}{2}\norms{\tilde{\xb}^{k+1} - \xb^{\star}}^2 + \frac{2\rho_0\norms{\Bb}^2 + \mu_g\tau_k}{2}\norms{\tilde{\yb}^{k+1} - \yb^{\star}}^2  \vspace{1ex}\\
& \leq~ \left(1-\tau_k\right)\left[ \Lc_{\rho_{k-1}}(\bar{\zb}^k,\lbd)  - F(\zb^{\star})\right]   + \tfrac{\tau_{k-1}^2(1-\tau_k)}{\rho_0}\norms{\hat{\lbd}^k - \lbd}^2 \vspace{1ex}\\
&+~ \tfrac{\gamma_0\tau_{k-1}^2(1-\tau_k)}{2}\norms{\tilde{\xb}^k - \xb^{\star}}^2  + \rho_0\norm{\Bb}^2\norms{\tilde{\yb}^k - \yb^{\star}}^2.
\end{array}{\!\!\!}
\end{align*}
Using the fact that $2\rho_0\norms{\Bb}^2 \leq (2\rho_0\norms{\Bb}^2 + \mu_g\tau_{k-1})(1-\tau_k)$, we obtain \eqref{eq:th41_proof_step2}.
\end{proof}
%%% End of the proof.

\begin{proof}[\textbf{The proof of Theorem \ref{th:convergence3b}}]
From \eqref{eq:th41_proof_step2}, by induction we obtain
\begin{align*} 
{\!\!\!}\begin{array}{ll}
&\Lc_{\rho_k}(\bar{\zb}^{k+1},\lbd) - F(\zb^{\star})  +  \tfrac{\tau_k^2}{\rho_0}\norms{\hat{\lbd}^{k+1} - \lbd}^2 + \frac{2\gamma_0\tau_k^2}{2}\norms{\tilde{\xb}^{k+1} - \xb^{\star}}^2  + \frac{2\rho_0\norms{\Bb}^2 + \mu_g\tau_k}{2}\norms{\tilde{\yb}^{k+1} - \yb^{\star}}^2 \vspace{1ex}\\
&   \leq \prod_{i=1}^k(1-\tau_i)\Big[ \Lc_{\rho_0}(\bar{\zb}^1,\lbd)  - F(\zb^{\star}) + \tfrac{1}{\rho_0}\norms{\hat{\lbd}^1 - \lbd}^2 + \tfrac{\gamma_0}{2}\norms{\tilde{\xb}^1 - \xb^{\star}}^2  + \tfrac{2\rho_0\norms{\Bb}^2 + \mu_g}{2}\norms{\tilde{\yb}^1 - \yb^{\star}}^2 \Big].
\end{array}{\!\!\!}
\end{align*}
Using \eqref{eq:th41_step1} with $k = 0$, we get
\begin{equation*}
\begin{array}{ll}
\Lc_{\rho_0}(\bar{\zb}^1,\lbd)  - F(\zb^{\star}) &+ \tfrac{1}{\rho_0}\norms{\hat{\lbd}^1 - \lbd}^2 + \tfrac{\gamma_0}{2}\norms{\tilde{\xb}^1 - \xb^{\star}}^2  + \tfrac{2\rho_0\norms{\Bb}^2 + \mu_g}{2}\norms{\tilde{\yb}^1 - \yb^{\star}}^2 \vspace{1ex}\\
&\leq \frac{1}{\rho_0}\norms{\lbd - \hat{\lbd}^0}^2 + \frac{\gamma_0}{2}\norms{\tilde{\xb}^0 - \xb^{\star}}^2 + \rho_0\norms{\Bb}^2\norms{\tilde{\yb}^0 - \yb^{\star}}^2.
\end{array}
\end{equation*}
Combining these two inequalities, and using $\prod_{i=1}^k(1-\tau_i) = \tau_k^2$ and $\tilde{\zb}^0 = \bar{\zb}^0$, we finally get
\begin{equation*}
\Lc_{\rho_k}(\bar{\zb}^{k+1},\lbd) - F(\zb^{\star})  \leq \frac{\tau_k^2}{2}\left[\tfrac{2}{\rho_0}\norms{\hat{\lbd}^0 - \lbd}^2 + \gamma_0\norms{\bar{\xb}^0 - \xb^{\star}}^2  + 2\rho_0\norms{\Bb}^2\norms{\bar{\yb}^0 - \yb^{\star}}^2\right].
\end{equation*}
Similar to the proof of Theorem~\ref{th:admm-convergence1}, using the fact that $F(\bar{\zb}^{k+1}) - F^{\star} - \iprods{\lbd,\Ab\bar{\xb}^{k+1} + \Bb\bar{\yb}^{k+1} - \cb} = \Lc(\bar{\zb}^{k+1},\lbd) - F^{\star} \leq \Lc_{\rho_k}(\bar{\zb}^{k+1},\lbd) - F^{\star}$ and the last estimate into Lemma~\ref{le:approx_opt_cond}, we can show that
\begin{equation*}
\norms{\Ab\bar{\xb}^k + \Bb\bar{\yb}^k - c} \leq \frac{2\bar{R}_0^2}{\norms{\lbd^{\star}}(k+2)^2}~~\text{and}~~ \vert F(\bar{\zb}^k) - F^{\star}\vert \leq \frac{2\bar{R}_0^2}{(k+2)^2},
\end{equation*}
where $R_0^2 := \tfrac{2}{\rho_0}\big(\norms{\hat{\lbd}^0} - 2\norms{\lbd^{\star}}\big)^2 + \gamma_0\norms{\bar{\xb}^0 - \xb^{\star}}^2  + 2\rho_0\norms{\Bb}^2\norm{\bar{\yb}^0 - \yb^{\star}}^2$, which is \eqref{eq:convergence3b}.
\end{proof}
%%% End of the proof.

%%% The proof of Theorem 4.
\beforesubsec
\subsection{The proof of Corollary \ref{co:convergence3c}: Parallel variant with strong convexity}\label{apdx:co:convergence3c}
\aftersubsec
From \eqref{eq:descent_pro5}, following the same proof of \eqref{eq:th41_step1}, if $\eta_k \leq \frac{\rho_k\tau_k}{2}$, then we can derive
\begin{align}\label{eq:co41_proof1}
\Lc_{\rho_k}(\bar{\zb}^{k\!+\!1},\lbd) & -~ F(\zb^{\star})  + \tfrac{(\gamma_k\tau_k^2 \!+\! \mu_f\tau_k)}{2}\norms{\tilde{\xb}^{k\!+\!1} - \xb^{\star}}^2 + \tfrac{(\beta_k\tau_k^2 \!+\! \mu_g\tau_k)}{2}\norms{\tilde{\yb}^{k\!+\!1} - \yb^{\star}}^2 +  \tfrac{\tau_k}{2\eta_k}\norms{\hat{\lbd}^{k\!+\!1} \!-\! \lbd}^2 \nonumber\\
& \leq~ (1-\tau_k)\big[ \Lc_{\rho_{k-1}}(\bar{\zb}^k,\lbd)  - F(\zb^{\star}) \big] - \tfrac{(1-\tau_k)(\rho_{k-1} - \rho_k(1-\tau_k))}{2}\norms{\bar{\sb}^k}^2 \nonumber\\
&+~ \tfrac{\gamma_k\tau_k^2}{2}\norms{\tilde{\xb}^k - \xb^{\star}}^2  + \tfrac{\beta_k\tau_k^2}{2}\norms{\tilde{\yb}^k - \yb^{\star}}^2 + \tfrac{\tau_k}{2\eta_k}\norms{\hat{\lbd}^k - \lbd}^2 \nonumber\\
& - \tfrac{(\gamma_k - 2\rho_k\norms{\Ab}^2)\tau_k^2}{2}\norms{\tilde{\xb}^{k+1} - \tilde{\xb}^k}^2 - \tfrac{(\beta_k - 2\rho_k\norms{\Bb}^2)\tau_k^2}{2}\norms{\tilde{\yb}^{k+1} - \tilde{\yb}^k}^2.
\end{align}
In order to telescope \eqref{eq:co41_proof1}, we need to impose the following conditions
\begin{equation*} 
\left\{\begin{array}{lll}
\tfrac{\gamma_k\tau_k^2}{(1-\tau_k)\tau_{k-1}^2} \leq \gamma_{k-1} + \frac{\mu_f}{\tau_{k-1}},~~~&\tfrac{\beta_k\tau_k^2}{(1-\tau_k)\tau_{k-1}^2} \leq \beta_{k-1}  + \frac{\mu_g}{\tau_{k-1}}, & \vspace{1ex}\\
\eta_k\tau_k \leq \eta_{k-1}\tau_{k-1}(1-\tau_k),~~& \eta_k \leq \frac{\rho_k\tau_k}{2},& \vspace{1ex}\\
2\rho_k\norms{\Ab}^2 \leq \gamma_k, ~~&2\rho_k\norms{\Bb}^2 \leq \beta_k, ~~&\rho_k(1-\tau_k) \leq \rho_{k-1}.
\end{array}\right.
\end{equation*}
Using these conditions, we can derive the update rules for the parameters as in \eqref{eq:co41_param_update10}.
The rest of the proof is similar to the proof of Theorem \ref{th:convergence3b} but using \eqref{eq:co41_proof1}.
We omit the details here.
\Eproof
%%% End of the proof.

%%% Lower-bound on rate
\beforesubsec
\subsection{Lower bound on convergence rate of the strongly convex case}\label{subsec:lower_bound_rate_of_admm2}
\aftersubsec
We consider again example \eqref{eq:composite_exam}, where we assume that $g$ is $\mu_g$-strongly convex.
Algorithm \ref{alg:A0b} and its parallel variant \eqref{eq:admm_scheme3_admm2} for solving \eqref{eq:composite_exam} are special cases of \eqref{eq:general_ADMM} if $f$ and/or $g$ are strongly convex.
Then, by \cite[Theorem 2]{woodworth2016tight}, the lower bound complexity of \eqref{eq:general_ADMM} to achieve $\hat{\xb}$ such that $F(\hat{\xb}) - F^{\star} \leq \varepsilon$ is $\Omega\left( \frac{1}{\sqrt{\varepsilon}}\right)$. 
Consequently, the rate of Algorithm \ref{alg:A0b} and its parallel variant \eqref{eq:admm_scheme3_admm2} stated in Theorem \ref{th:convergence3b} and Corollary \ref{co:convergence3c}, respectively, is optimal.

\bibliographystyle{siamplain}
%\bibliography{references}
%\bibliography{/Users/quoctd/Dropbox/E-Books/tran_bibtex_new}

\end{document}